\documentclass[11pt]{amsart}
 
\theoremstyle{plain}

\usepackage{amsmath,amssymb,amsfonts}
\usepackage{enumerate,comment}
\usepackage{amscd, amssymb, latexsym, amsmath, amscd,xcolor}
\usepackage[all]{xy}
\usepackage{pb-diagram}
\usepackage{picture}
\usepackage{graphicx,tikz}
\usepackage{multicol,lipsum}
\usepackage{url}
\newtheorem{thm}[equation]{Theorem}
\newtheorem{prop}[equation]{Proposition}
\newtheorem{cor}[equation]{Corollary}

\newtheorem{lemma}[equation]{Lemma}

\newtheorem{definition}[equation]{Definition}
\newtheorem{remark}[equation]{Remark}
\numberwithin{equation}{section}

\newcommand{\Q}{\mathbb Q}
\newcommand{\Z}{\mathbb Z}
\newcommand{\R}{\mathbb R}
\newcommand{\C}{\mathbb C}

\newcommand{\A}{\mathbb A}

\newcommand{\ps}{\par \medskip}

\def\Hom{{\rm Hom}}
\def\Aut{{\rm Aut}}

\def\SL{{\rm SL}}

\def\GSp{{\rm GSp}}

\def\PGSp{{\rm PGSp}}
\def\GSpin{{\rm GSpin}}
\def\PGSO{{\rm PGSO}}
\def\PGO{{\rm PGO}}
\def\Spin{{\rm Spin}}
\def\Sp{{\rm Sp}}
\def\Spin{{\rm Spin}}

\def\GO{{\rm GO}}
\def\GL{{\rm GL}}
\def\PGL{{\rm PGL}}
\def\GSO{{\rm GSO}}

\def\SO{{\rm SO}}

\def\Sp{{\rm Sp}}
\def\Mp{{\rm Mp}}

\def\O{{\rm O}}

\def\Irr{{\rm Irr}}
\newcommand{\peacealt}{%
  \tikz[scale=0.1, line width=0.4pt, baseline=-0.5ex]{
    \draw (0,0) circle (1);
    \draw (0,0) -- (210:1);
    \draw (0,0) -- (330:1);
    \draw (0,0) -- (90:1);
  }%
}
\def\GGGSpin{{}^{\peacealt}{\rm Spin}}
\usepackage[all]{xy}

\def\A{{\mathbb A}}

\def\R{{\mathbb R}}

\def\Z{{\mathbb Z}}

\def\C{{\bf C}}

\def\C{{\mathbb C}}

\let\ssim\sim   
\def\sim{{\rm sim}}

\begin{document}

\title{Triality and Functoriality} 
 \author{Gaetan Chenevier and Wee Teck Gan}
 
 \address{D\'epartement de math\'ematiques et applications, \'Ecole Normale Sup\'erieure PSL, 45 rue d'Ulm, 75230 Paris Cedex, France}
\email{gaetan.chenevier@math.cnrs.fr}

\address{Department of Mathematics, National University of Singapore, 10 Lower Kent Ridge Road
Singapore 119076} 
\email{matgwt@nus.edu.sg}


\begin{abstract}
We use the triality automorphism of simple algebraic groups of type $D_4$ 
to prove some new instances of global Langlands functorial lifting. 
In particular, we prove the (weak) spin lifting from $\GSp_6$ to $\GL_8$ 
and the tensor product lifting from $\GL_2 \times \GSp_4$ to $\GL_8$. 
As an arithmetic application, we establish the expected properties 
of the spinor $L$-function attached to an arbitrary Siegel modular cusp form 
for $\Sp_6(\Z)$ generating a holomorphic discrete series. 
\end{abstract}	

\maketitle

\section{\bf Introduction}

It is a truth universally acknowledged, that a single mathematical object in possession of a good many symmetries compared to its peers, must be  in want of our attention.\footnote{a tribute to Jane Austen, on the occasion of her 250th anniversaire}
\vskip 5pt

Taking this principle to heart and putting aside our own {\it Pride and Prejudice}, we  consider in this paper the series $D_n$  of simple Lie groups or Lie algebras, whose  Dynkin diagram is given by:
\vskip 10pt

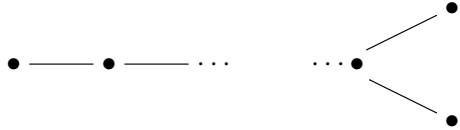
 \[
\xymatrix@R=10pt{
&   &   &    &      \bullet  \\
 \bullet  \ar@{-}[r] & \bullet  \ar@{-}[r] & \cdots& \cdots \bullet  \ar@{-}[ru] \ar@{-}[rd] &      \\
&   &   &    &      \bullet
}
\]

\vskip 10pt
\noindent 
with $n$ equal to the number of vertices. 
Hence,  one may be looking at the split linear algebraic groups $\SO_{2n}$ (over some field $F$) or its Lie algebra $\mathfrak{so}_{2n}$.
By Lie theory, the outer automorphism group ${\rm Out}(D_n)$ of $\mathfrak{so}_{2n}$ is equal to the group of symmetries of its Dynkin diagram. Thus one finds that
\[  {\rm Out}(D_n)  = \begin{cases}
\Z/2\Z, \text{  if $n \ne 4$;} \\
\text{the symmetric group $S_3$,  if $n=4$,}  \end{cases} \] 
so that  the Lie algebra of type $D_4$ has more symmetries than the generic Lie algebra of type $D_n$, as is evident from its Dynkin diagram:
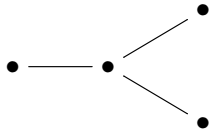
\[
\xymatrix@R=10pt{
&   & \bullet   \ar@{-}[ld]  \\
\bullet \ar@{-}[r] &   \bullet &  \\
 &  &  \bullet \ar@{-}[lu]  \\
} \]

\noindent  In particular, it has an outer automorphism $\theta$ of order $3$ and this phenomenon is aptly named {\em triality}.  The purpose of this paper is to investigate some implications of  the existence of this extra symmetry  in the Langlands program.

\vskip 5pt

Let us be more precise. If one considers the special orthogonal group $\SO_{2n}$, then its outer automorphism group is indeed $\Z/2\Z$ for all $n$, with an outer automorphism  realised by an element in $\O_{2n} \setminus \SO_{2n}$. This is because the triality automorphism $\theta$ of $\mathfrak{so}_8$ cannot be realised on $\SO_{8}$, but only on the simply-connected or adjoint form of the group, i.e. the group $\Spin_8$ or $\PGSO_8$ respectively.  
 The center $Z_{\Spin_8}$ of $\Spin_8$ is the finite group scheme $\mu_2 \times\mu_2$ which has 3 nontrivial subgroups of order $2$. The triality automorphism  preserves the center, permuting the 3 subgroups of order $2$, so that it descends to an automorphism of the adjoint group $\PGSO_8$. This implies that one has a commutative diagram of isogenies:
 \[
\xymatrix@R=10pt{
 & \Spin_8  \ar[ld]  \ar[d] \ar[rd] & \\
\SO_8    \ar[rd]   &    \SO_8    \ar[d]   &   \ar[ld]   \SO_8 \\
& \PGSO_8 & 
}
\]
In particular, there are three non-conjugate maps
 \[  \rho_j :  \Spin_8 \twoheadrightarrow \SO_8, \]
whose kernels are the three central $\mu_2$-subgroups in $\Spin_8$. These  three maps are cyclically permuted by the triality automorphism, and together give an embedding 
 \[  \rho = \prod_j \rho_j : \Spin_8 \hookrightarrow \SO_8^3. \]
 In Section 2, following \cite{KMRT}, we describe a construction of the group $\Spin_8$ as a subgroup of $\SO_8^3$, equipped with an outer automorphism of order $3$ which is simply the restriction of the cyclic permutation on $\SO_8^3$.  The construction in \S 2 thus gives a realisation of the three 8-dimensional irreducible fundamental representations of $\Spin_8$. 
  Likewise, there are 3 non-conjugate maps
 \[  f_j : \SO_8 \longrightarrow \PGSO_8 \]
 which are cyclically permuted by the triality automorphism.
  The existence of the  3 different maps $f_j$ (or $\rho_j$)  is a manifestation of the principle of  triality.
 \vskip 5pt

In this paper, we shall consider some  implications of the triality automorphism in the  Langlands program.
 In particular, we shall  exploit triality to construct some new instances of Langlands functoriality.
 \vskip 10pt

  \subsection{\bf Langlands functoriality} \label{SS:Lfunctorial}
 We begin with a brief recollection of the notion of Langlands functoriality.
Suppose that $\pi = \otimes_v \pi_v$  is an irreducible automorphic representation  of a split reductive group $H$ over a number field $k$.
For $v$ outside a finite set $S$ of places of $k$, the local representation $\pi_v$ of $H(k_v)$ is unramified and gives rise to a semisimple conjugacy class $c(\pi_v)$ in the Langlands dual group $H^{\vee}$ via the Satake isomorphism. Thus, $\pi$ gives rise to a collection of semisimple conjugacy classes 
\[  c(\pi) := \{ c(\pi_v): v \notin S \}  \]
which encodes the Hecke eigenvalues of $\pi$ for the spherical Hecke algebras at almost all places of $k$. 
We consider two such collections of semisimple classes in $H^{\vee}$ to be equivalent if they agree at all but finitely many places; in other words, we are allowed to enlarge the finite set $S$ if necessary. We shall call (the equivalence class of) $c(\pi)$  the {\it Hecke-Satake family}  of $\pi$. 

\vskip 5pt

Now suppose that $G$ is another split reductive group and we have an algebraic group homomorphism
\[  \iota: H^{\vee} \longrightarrow G^{\vee}. \]
Then, by composition with $\iota$,  a Hecke-Satake family $c(\pi)$ gives rise to a family of semisimple conjugacy classes 
\[ \iota(c(\pi)) :=  \{ \iota(c(\pi_v)): v \notin S \}   \]
in $G^{\vee}$. One may ask if there exists an irreducible automorphic representation $\sigma$ of $G$ whose Hecke-Satake family agrees with  $\iota(c(\pi))$   for almost all $v$; we shall express this as:
\[  \iota(c(\pi)) = c(\sigma). \]
The Langlands functoriality conjecture asserts that the answer is affirmative, in which case one says that the automorphic representation $\sigma$ is a {\em weak functorial lifting} of $\pi$.

\vskip 5pt

As an example, consider the map of dual groups
\[  f_j^{\vee}:  \Spin_8(\C) \longrightarrow \SO_8(\C) \]
which is dual to the map 
\[ f_j: \SO_8 \longrightarrow \PGSO_8. \]
Then the weak functorial lifting associated to $f_j^{\vee}$ exists and is simply given by any irreducible constituent of the pullback of an automorphic representation by $f_j$ (see Proposition \ref{P:isogeny}). 
\vskip 5pt

As another example, suppose that $H$ is a symplectic or orthogonal  group, so that its Langlands dual group $H^{\vee}$ is  also a classical group. Then $H^{\vee}$ has a standard representation
\[  {\rm std}: H^{\vee} \rightarrow \GL_N(\C),\]
which should induce a weak functorial lifting of automorphic representations from $H$ to $\GL_N$.   The existence of this weak functorial lifting  is highly nontrivial and has been shown by Cogdell-Kim-Piatetski-Shapiro-Shahidi \cite{CKPSS} (for generic automorphic representations) and  Cai-Friedberg-Kaplan \cite{CFK} (for all automorphic representations) using Converse Theorems,  and independently by Arthur \cite{A} (for all automorphic  representations) using the stable twisted trace formula. 
The most  pertinent case for us is  $H = \SO_8$ with its standard representation ${\rm std}: \SO_8(\C) \longrightarrow \GL_8(\C)$. 

\vskip 10pt

\subsection{\bf Triality and  Langlands functoriality}
 We shall see that  the principle of triality can be exploited to obtain some interesting results towards Langlands functoriality. The basic idea is the following. Suppose one has  a morphism of dual groups 
\[  \begin{CD} 
H^{\vee} @>\iota>> \PGSO_8^{\vee} = \Spin_8(\C) @>f_1^{\vee}>> \SO_8(\C) \end{CD} \]
and we are able to construct the weak functorial lifting for $\iota$ and hence for $f_1^{\vee}\circ \iota$  (since the weak functorial lifting for $f_1^{\vee}$ is given by pullback of automorphic forms, as we have explained in the previous subsection). Now one may compose $\iota$ with the triality automorphism $\theta$ and obtain another map
\[  \begin{CD}
 H^{\vee} @>\iota>> \Spin_8(\C)  @>\theta>>  \Spin_8(\C) @>f_1^{\vee}>> \SO_8(\C).  \end{CD} \]
 This composite map may give a drastically different instance of Langlands functoriality from $H$ to $\SO_8$ since $\theta \circ \iota$ and $\iota$ may not be conjugate in $\Spin_8(\C)$. The map $f_1^{\vee} \circ\theta \circ \iota$ can  be more simply described  as
 \[  \begin{CD} 
 H^{\vee} @>\iota>> \Spin_8(\C)   @>f_2^{\vee}>> \SO_8(\C).   \end{CD} \]
 The existence of the triality automorphism essentially gives one this new functorial lifting for $f_2^{\vee} \circ \iota$ with minimal effort. This idea has already been observed and exploited in the monograph \cite{CL} of the first author and Lannes and we push it further in the present paper. It has also played a critical role in the proof of the local Langlands conjecture for ${\rm G}_2$ given in  \cite{GS2}. 
\vskip 5pt

By using this simple idea, we prove  the following instances of Langlands functorial lifting:
 \vskip 5pt

\begin{thm} \label{T:intro1}
(i) Consider the map of dual groups
\[  {\rm spin}: \Spin_7(\C) \longrightarrow \GL_8(\C) \]
given by the Spin representation.  Then  for any   cuspidal representation $\pi$ of $\PGSp_6$ whose restriction to $\Sp_6$ has a generic (or tempered) A-parameter, the corresponding weak Spin lifting of $\pi$ to $\GL_8$ exists. 
\vskip 5pt

%(ii) The {\em Ikeda lifting} from $\PGL_2 = \PGSp_2$ to $\Sp_8$ corresponding to the map of dual groups:
%\[ \begin{CD}
 % \mathcal{L}_F \times \SL_2(\C)  @>\phi \boxtimes \Sym^3>>  \SL_2(\C) \times  \Sp_4(\C)   @>\boxtimes>> \SO_8(\C) @>>> \SO_9(\C) = \Sp_8^{\vee}
  %\end{CD}  \]
  %exists.
  
%  \vskip 5pt
  
%  (ii) The lifting from $\PGSp_4$ to $\SO_8$ given by the following maps of dual groups:
%\[  \begin{CD}
 % \mathcal{L}_F \times \SL_2(\C)  @>\phi \times {\rm Id}>> \Sp_4(\C) \times \SL_2(\C) @>\boxtimes>> \SO_8(\C)  \end{CD} \]
%exists.

%\vskip 5pt

(ii) The Rankin-Selberg lifting of automorphic representations {\em of symplectic type} from $\GL_4 \times \GL_2$ to $\GL_8$, corresponding to the following maps of dual groups:
\[  \begin{CD} 
 \Sp_4(\C) \times \SL_2(\C)  @>\boxtimes>>  \SO_8(\C)   @>{\rm std}>> \GL_8(\C)
 \end{CD} \]
exists.
\end{thm}

We refer the reader to Theorem \ref{thm:weakspinlift}, Theorem \ref{T:strong}  and  Theorem \ref{T:RS} for more details about the Spin lifting and Rankin-Selberg lifting of the theorem.  In  \S \ref{S:GSp(6)}, we remove the hypothesis of ``trivial central character" in the above theorem (see Theorems \ref{T:GSpin lifting} and \ref{T:RS2}).
In particular, we demonstrate the Spin lifting from $\GSp_6$ to $\GL_8$ induced by $\GSpin_7(\C) \rightarrow \GL_8(\C)$. We also give an application to a lifting from selfdual cuspidal automorphic representations of $\PGL_7$ to $\SL_8$. In \S \ref{S:Siegel}, we refine Theorems~\ref{thm:weakspinlift} and~\ref{T:strong} further 
in the setting of automorphic representations of ${\rm PGSp}_6$ over $\Q$ 
generated by holomorphic Siegel modular forms for the full Siegel modular group ${\rm Sp}_6(\Z)$. 
In this arithmetic situation, we show that these theorems also hold 
if we do not impose the genericity (of A-parameters) assumption. These improvements are possible by arguments using  Galois representations. 
As an example, we prove the following result (see Theorem~\ref{thm:Lspinsiegel}):
\begin{cor} Assume $\pi$ is a cuspidal automorphic representation of ${\rm PGSp}_6$ over $\Q$
generated by a Siegel modular cusp form for ${\rm Sp}_6(\Z)$ with weights $k_1 \geq k_2 \geq k_3 \geq 4$.
Then:
\begin{itemize}
\item[(i)] ${\rm L}(s,\pi,{\rm spin})$ has a meromorphic continuation to all of $\C$, 
with at most a simple pole at $s=0$ and $1$, and no other poles. It satisfies ${\rm L}(s,\pi,{\rm spin})={\rm L}(1-s,\pi,{\rm spin})$. \ps
\item[(ii)] Moreover, ${\rm L}(s,\pi,{\rm spin})$  has a pole at $s=1$ if, and only if, $\pi$ is of type ${\rm G}_2$, in which case  ${\rm L}(s,\pi,{\rm spin})=\zeta(s) \cdot  {\rm L}(s,\pi,{\rm std})$. 
\end{itemize}
\end{cor}

A weaker statement had been proved by Pollack in \cite[Thm. 1.2]{pollack}, assuming that the associated Siegel modular form has a nonzero Fourier coefficient at the maximal order of a definite quaternion algebra.
For ${\rm PGSp}_6$ over a general number field, note that Theorem~\ref{T:intro1} shows that the partial spinor $L$-function of any cuspidal $\pi$ with a generic standard A-parameter is a product of Godement-Jacquet $L$-functions, providing a rather different approach to the constructions in \cite{BG,Vo}.

  \vskip 10pt

Finally, we remark that  Theorem \ref{T:intro1}  depends on results of Arthur  \cite{A}, B. Xu \cite{X1, X2} and Moeglin-Waldspurger \cite{MW1,MW2} as stated, through the use of the stable twisted trace formula and the notions of A-parameters and A-packets. As such, they are currently subject to  the hypothesis of the twisted weighted fundamental lemma. However, one can obtain slightly weaker versions of these results  unconditionally. For example, in Theorem \ref{T:intro1}(i), if we had assumed that $\pi$ is globally generic, there will be no need to invoke \cite{A}.   But since the various hypotheses present in the aforementioned references should soon be fully verified (see for example \cite{AGIKMS}), we prefer not to split hairs on such issues, so as  to preserve our own  {\it Sense and Sensibility}.
 \vskip 10pt

\noindent{\bf Acknowledgments:}  The authors thank Bin Xu for several helpful conversations about his work, and Dipendra Prasad for his comments on the first version of this paper. The first author also thanks Olivier Ta\"ibi for useful discussions. This project was initiated when the authors participated in the Oberwolfach workshop ``Harmonic Analysis and Trace Formula'' during May 22-26, 2017; we thank the organizers Werner Mueller, Sug Woo Shin,  Birgit Speh and Nicholas Templier for their invitation, and the Mathematisches Forschungsinstitut Oberwolfach  for the excellent working conditions. The main work for this project was completed when the second author visited the IHES during May 15-July 12, 2017; the second author thanks the IHES for its support, as well as providing excellent academic atmosphere and living conditions.  
The first author is partially supported by the C.N.R.S. and 
the project ANR-19-CE40-0015-02 (COLOSS). The second author is  supported by a Tan Chin Tuan Centennial Professorship at NUS. 
\vskip 10pt

\vskip 10pt
\section{\bf Triality} \label{S:triality}
In this section, we fix our notations for the algebraic groups used in the article and describe some background material on the triality automorphism $\theta$ of $D_4$. As we mentioned in the introduction,
the automorphism $\theta$ can only be realised on the simply-connected group $\Spin_8$ or the adjoint group $\PGSO_8$. 
In the process, we shall give a construction of $\Spin_8$ that makes triality especially transparent, following the treatment of  \cite{KMRT}. 
\vskip 5pt

\subsection{\bf Spin groups and representations} \label{SS:spingp}
We begin by reviewing the classical construction of Spin groups and their spin representations.
Let  $F$ be an arbitrary field with ${\rm char}(F) \neq 2$. 
 Recall that for any non-degenerate quadratic space $(V,Q)$ over $F$, 
we have the naturally associated special orthogonal and proper similitude groups $\SO(V) \subset \GSO(V)$.
When we want to emphasize the quadratic form $Q$, we also write $\SO(V,Q)$ for $\SO(V)$ for example.
Via the theory of Clifford algebras, one may construct the associated Spin and general Spin groups $\Spin(V) \subset \GSpin(V)$: these are classically defined as 
subgroups of the invertible elements  of the even Clifford algebra ${\rm Cl}^0(V)$ of $V$. 
One has a natural projection %commutative diagram:
%\[ \begin{CD} 
%\Spin(V) @>>> \GSpin(V) \\
%@V{\rho}VV  @VV{\rho}V \\ 
%\SO(V) @>>> \GSO(V). \end{CD} \]
\[ \begin{CD} \rho : \GSpin(V) @>>> \SO(V) \end{CD} \]
often called the {\it standard} morphism. 
The resulting representation of $\GSpin(V)$ on $V$ is its {\it standard representation} 
\[  {\rm std}: \GSpin(V) \longrightarrow \SO(V) \longrightarrow \GL(V). \]
The restrictions to $\Spin(V) \subset \GSpin(V)$ of $\rho$ and ${\rm std}$ will still be denoted by $\rho$ and ${\rm std}$. 
   
\vskip 5pt

%If $V$ is a split quadratic space of rank $r$ over $F$, we will frequently denote the above groups by  
% ${\rm SO}_r$ and  ${\rm Spin}_r$.
 
% Assume now that $V$ is split with $\dim V = r$; we shall often denote the groups introduced above by $\SO_r \subset \GSO_r$ and $\Spin_r \subset \GSpin_r$.

In the case $V$ is split with $\dim V = r$, we shall often denote the groups introduced above by $\SO_r \subset \GSO_r$ and $\Spin_r \subset \GSpin_r$. Assume $V$ is split, or more generally that $V$ has trivial discriminant and Clifford invariant (see~\cite[\S 35A]{KMRT}). If $\dim V =2n+1$ is odd, then ${\rm Cl}^0(V)$ has a spinor module $S_V$, unique up to isomorphism, which defines a spin representation 
 \[   {\rm spin}: \GSpin(V)  \longrightarrow \GL(S_V) \]
 of dimension $2^n$. On the other hand, if  $\dim V = 2n$ is even,  then any choice of spinor modules $S_{V,\pm}$ for ${\rm Cl}^0(V)$ 
defines two half-spin representations 
\[ {\rm spin}_{\pm}: {\rm GSpin}(V) \rightarrow {\rm GL}(S_{V,\pm}),\]
each being of dimension $2^{n-1}$ over $F$.  These half-spin representations are known to be 
orthogonal when $2n \equiv 0 \bmod 8$. Of course, we shall also denote by ${\rm spin}$ or ${\rm spin}^{\pm}$ the restriction of these representations to ${\rm Spin}(V)$.
\vskip 5pt

Suppose that $V = V_1 \oplus V_2$ is an orthogonal decomposition of $V$ into the sum of two non-degenerate quadratic spaces. Then one has a natural  commutative diagram
\[  \begin{CD} 
\GSpin(V_1) \times \GSpin(V_2) @>\iota>> \GSpin(V) \\
@V\rho_1\times \rho_2VV  @VV{\rho}V \\
\SO(V_1) \times \SO(V_2) @>\iota_{\flat}>> \SO(V)
\end{CD} \]
so that
\[  {\rm std}_V \circ  \iota \simeq  {\rm std}_{V_1} \oplus {\rm std}_{V_2}  \]
as representations of $\GSpin(V_1) \times \GSpin(V_2)$. On the other hand, one may consider the pullback of the spin representations under $\iota$. For this, we have:
\begin{itemize}
\item  when  $\dim V$ is even but  each $\dim V_i$ is odd:
\[  S_{V, \pm}  \circ \iota  \simeq  S_{V_1} \boxtimes S_{V_2} \] 
as representations of $\GSpin(V_1) \times \GSpin(V_2)$;

\item  when $\dim V$ and $\dim V_i$ are all even:
\[    S_{V, +}  \circ \iota  \simeq  S_{V_1, +} \boxtimes S_{V_2, +} \oplus S_{V_1, -} \boxtimes S_{V_2, -}  \] 
and
\[   S_{V, -}  \circ \iota  \simeq   S_{V_1, +} \boxtimes S_{V_2, -} \oplus S_{V_1, -} \boxtimes S_{V_2, +}  \] 
up to relabelling the half-spin modules;

\item when $\dim V$ and $\dim V_1$ are odd:
\[   S_V  \simeq  S_{V_1} \boxtimes S_{V_2,+} \oplus S_{V_1} \boxtimes S_{V_2,-}. \]
\end{itemize}

\vskip 5pt

 We now specialize to the case when  $\dim V = 8$. Observe that in this case, the standard and half-spin representations are all orthogonal of dimension $8$. As we mentioned in the introduction, these three inequivalent irreducible representations of $\Spin(V)$ are permuted transitively by the triality (outer) automorphism. However, the above classical construction does not readily reveal the triality automorphism. This is not surprising since the construction applies to quadratic spaces of any rank. In the rest of this section, we shall explain an alternative  construction which applies only when $\dim V = 8$ and  which visibly exhibits the triality automorphism. This alternative construction makes use of more structures on $V$ than that of a quadratic space, namely the structure of an octonion algebra.
 
\vskip 5pt

\subsection{\bf Octonion algebras.}

Let $\mathbb{O}$ be an octonion algebra over $F$. Thus, $\mathbb{O}$ is an $8$-dimensional (non-commutative and non-associative) composition algebra with unit, see \cite[\S 33]{KMRT}. In particular, it is equipped with an $F$-linear involution $x \mapsto \bar{x}$ such that $N(x) = x \cdot \bar{x}$ is the quadratic form permitting composition: 
\[  N(x \cdot y) = N(x) \cdot N(y) \quad \text{ for $x,y \in \mathbb{O}$.} \]
We denote by $b_N$ the symmetric bilinear form associated to $N$ and $1$ the unit of $\mathbb{O}$.
% and let $Tr(x) = x+ \bar{x} = b_N(x,1)$ be the linear trace map.
The automorphism group of $\mathbb{O}$ is an exceptional simple algebraic group over $F$ of type ${\rm G}_2$; it is a subgroup of 
$\SO(\mathbb{O}, N)$ that we denote by ${\rm G}_2^{\mathbb{O}}$. It is split over $F$ if $\mathbb{O}$ is,
in which case we often simply denote ${\rm G}_2^{\mathbb{O}}$ by ${\rm G}_2$. The action of ${\rm G}_2^\mathbb{O}$ on $\mathbb{O}$ fixes $1$, and hence its orthogonal complement with respect to $b_N$, which is the 7-dimensional standard  representation (note ${\rm char}(F) \neq 2$). 

\vskip 5pt

\subsection{\bf Symmetric Composition Algebras.}
For the purpose of constructing the Spin group with its triality automorphism, it will be more convenient to work with a modified multiplication law on $\mathbb{O}$. We set
\[  x \ast y := \bar{x} \cdot \bar{y} \quad \text{for $x,y \in \mathbb{O}$.} \]
Then $(\mathbb{O}, \ast, N)$ is called a {\it para-octonion} algebra and satisfies the axioms of a {\it symmetric composition algebra} \cite[Pg. 463-464]{KMRT}: it is  a non-unital composition algebra with respect to $N$ which satisfies:
\[   b_N(x \ast y, z)  = b_N(x, y \ast z) \quad \text{  for all $x,y,z \in \mathbb{O}$.}  \]
  However, automorphisms of $\mathbb{O}$ give rise to automorphisms of $(\mathbb{O}, \ast)$ and vice versa, so that $\Aut(\mathbb{O}, \ast) = {\rm G}_2^\mathbb{O}$ \cite[Cor. 34.6]{KMRT}.  
 
%  \vskip 5pt
%
%
%By \cite[Theorem 34.37]{KMRT}, it turns out that  in dimension $8$, there is one other split symmetric composition algebra if the field $F$ is algebraically closed and  ${\rm char}(F) \ne 3$, in which  case $F$  contains a primitive 3rd root of unity $\omega$. This is realised on the space $\mathbb{M}$ of trace zero $3 \times 3$ matrices with multiplication given (miraculously) by \cite[Pg. 472, (34.18)]{KMRT}
%\[   x \ast y = \omega xy + (1-\omega)yx  - \frac{1}{3} T(yx), \quad \text{  for $x,y \in \mathbb{M}$,}  \]
%where $T$ is the trace of a $3 \times 3$ matrix. If the characteristic polynomial of  $x \in \mathbb{M}$ is
%\[ P_x(t) =  t^3  + S(x) t  - \det(x),   \]
%then the quadratic form on $\mathbb{M}$ which permits composition is  $\frac{1}{3} \cdot S(x)$ \cite[Prop. 34.19]{KMRT}. The automorphism group  of $(\mathbb{M}, \ast)$ is the group $\PGL_3 \subset \SO(\mathbb{M}, S)$ acting on $\mathbb{M}$ by conjugation. In other words, the action of $\PGL_3$ on $\mathbb{M}$ is isomorphic to the adjoint representation. 
 
\vskip 5pt

\subsection{\bf The Spin group in dimension $8$.}   \label{SS:VQ}

%Given a symmetric composition algebra $(V, \ast, Q)$ of dimension $8$, the associated Spin group is defined by
%\[  \Spin(V,\ast, Q) = \{ (g_1, g_2, g_3) \in \SO(V,Q)^3:  g_1(x \ast y) = g_2(x) \ast g_3(y) \text{ for $x,y \in V$} \}. \]

Assume $(V, \ast, Q)$ is a symmetric composition algebra of dimension $8$. It follows from \cite[\S 35A]{KMRT} 
that the even Clifford algebra of $(V,Q)$ is naturally isomorphic to ${\rm End}(V) \times {\rm End}(V)$. In particular, $(V,Q)$ has trivial discriminant and Clifford invariant, and we may take $S_{V,\pm}=V$.
Better, let us define
\[  \Spin(V,\ast, Q) = \{ (g_1, g_2, g_3) \in \SO(V,Q)^3:  g_1(x \ast y) = g_2(x) \ast g_3(y) \text{ for $x,y \in V$} \}. \]
Then $\Spin(V,\ast, Q)$ is an algebraic group over $F$ isomorphic to ${\rm Spin}(V,Q)$ by \cite[\S 35C]{KMRT}.
A charm of this rather simple definition is that not only does it not involve anymore the Clifford algebra, but also that it renders triality entirely natural and transparent.
Indeed, by construction the group $\Spin(V, \ast, Q)$ is equipped with the following extra structures:
\vskip 5pt

\begin{itemize}
\item the three projections to $\SO(V, Q)$ gives 3 homomorphisms
\[  \rho_j:  \Spin(V,\ast, Q) \longrightarrow \SO(V,Q), \]
with ${\rm Ker}(\rho_j) \simeq \mu_2 \subset Z_{\Spin(V, \ast, Q)}$. 
\vskip 5pt

\item the cyclic permutation of the three coordinates in $\SO(V,Q)^3$ preserves the subgroup  $\Spin(V, \ast, Q)$ and thus defines an order $3$ automorphism 
\[  \theta:  \Spin(V, \ast, Q) \longrightarrow \Spin(V, \ast, Q). \]
This is a triality automorphism which permutes the 3 maps $\rho_j$ cyclically. 
The fixed group of $\theta$ is, by definition, the subgroup $\Aut(V,\ast, Q) \subset \SO(V,Q)$.  

\item Because the automorphism $\theta$ necessarily preserves the center $Z_{\Spin(V,\ast, Q)}$, it descends to an automorphism of the adjoint group $\PGSO(V,Q)$, still denoted by $\theta$.  The projection from the three $\SO(V,Q)$'s thus gives 3 maps
\[  f_j:  \SO(V,Q) \longrightarrow \PGSO(V,Q), \]
which are permuted by $\theta$.
\end{itemize}
\vskip 10pt

The maps $f_j$ and $\rho_j$ alluded to in the introduction are those defined above. By  \cite[Prop. 35.7]{KMRT}, we may identify $\Spin(V, \ast, Q)$ with ${\rm Spin}(V,Q)$ in such a way that we have $\rho_1=\rho$, and that $\rho_2$ and $\rho_3$ induce the two half-spin representations of ${\rm Spin}(V,Q)$. 
This (breaking of symmetry!) being done, \cite[Prop. 35.1]{KMRT} shows that the two isogenies $\rho_2$ and $\rho_3$ naturally extend to $\GSpin(V,Q) \supset \Spin(V,Q)$ and give rise to two isogenies 
\[  \widetilde{\rho_j}:  \GSpin(V,Q) \longrightarrow \GSO(V,Q), \,\,\,j=2,3\]
inducing as well the two half-spin representations of $\GSpin(V,Q)$ on $S_{V,\pm}=V$.

%\subsection{\bf Splitting of $\Aut$-sequence.}
%Recall that one has a short exact sequence
%\begin{equation} \label{E:Aut}
%  \begin{CD}
%1 @>>> \PGSO_8 @>>>  \Aut(\Spin_8)  @>>>   {\rm Out}(\Spin_8) = S_3 @>>> 1. \end{CD} \end{equation}
%The data of a symmetric composition algebra $(V,\ast, Q)$ on a split quadratic space $(V,Q)$ gives a splitting of this short exact sequence over the order 3 subgroup $A_3 \subset S_3$. 
%
%Since we could apply the above discussion to the two symmetric composition algebra $(\mathbb{O}, \ast, N)$ and $(\mathbb{M}, \ast, S)$ (when $F = \overline{F}$ say), we obtain two splittings of this short exact sequence over $A_3$ whose fixed groups are ${\rm G}_2^\mathbb{O}$ and $\PGL_3$ respectively. According to \cite[Pg. 487, Cor. 35.10 and Prop. 36.14]{KMRT}, these are the only two possible splittings (up to inner conjugation) over $\overline{F}$. 
%\vskip 5pt
%
%In fact, when  $(\mathbb{O}, \ast, N)$ is a para-octonion algebra, one gets a splitting over the whole outer automorphism group $S_3$. For $g \in \SO(\mathbb{O}, N)$, let us set $\hat{g} \in \SO(\mathbb{O}, N)$ to be:
%\[  \hat{g} : x \mapsto  \overline{g(\bar{x})}. \]
%Then, as an example,  the map
%\[  (g_1, g_2, g_3)  \mapsto  (\hat{g}_1,  \hat{g}_3, \hat{g}_2)  \]
%gives a lifting of the transposition $(23)$ in $S_3$ to $\Aut(\Spin(\mathbb{O}, \ast, N))$. One has 
%\[  \Spin(\mathbb{O}, \ast, N)^{S_3} =  \Spin(\mathbb{O}, \ast, N)^{A_3}  = {\rm G}_2^\mathbb{O}. \] 
%
%\vskip 10pt

\vskip 10pt

\section{\bf Spin Lifting for $\PGSp_6$} \label{S:PGSp(6)}
 We shall now investigate the consequences of triality for Langlands functoriality. 
 In particular, in this section, we shall establish the so-called Spin lifting from $\PGSp_6$ to $\GL_8$. 
\vskip 5pt

Let $k$ be a number field with ring of ad\`eles $\A$. For a reductive group $G$ over $k$, set $[G] = G(k) \backslash G(\A)$ and let $\mathcal{A}(G)$ denote the space of automorphic forms for $G$. The subspace of cusp forms is denoted by $\mathcal{A}_{cusp}(G)$.
\vskip 5pt

\subsection{\bf  Simple functorial lifting}
We have introduced the notion of weak Langlands functorial lifting in \S \ref{SS:Lfunctorial}. For the sake of convenient reference, we document an instance of functorial lifting already alluded to in the introduction and which we will appeal to repeatedly later on.
\vskip 5pt

With $G$ and $H$ split connected reductive algebraic groups over $k$, $Z$ a commutative algebraic group of multiplicative type, and $S$ a torus over $k$, we shall consider an exact sequence of $k$-groups of the form:
\begin{equation}
\label{cd:defphiZQ}
1 \longrightarrow Z \longrightarrow G \overset{\phi}{\longrightarrow} H \longrightarrow S \longrightarrow 1.
\end{equation}
 \vskip 5pt
\noindent In particular, $Z$ is isomorphic to a central subgroup of $G$, and $S$ is a co-central quotient of $H$ (and is necessarily split). 
Here are some examples we have in mind:
\vskip 5pt

\begin{itemize}
\item  $G = H$, $Z$ and $S$ are trivial, and $\phi$ is an automorphism of $G$ (not necessarily inner);
\vskip 5pt

\item $Z$ is a finite central subgroup and $S$ is trivial, so that $\phi$ is an isogeny;
\vskip 5pt

\item $Z$ is trivial and $S=\mathbb{G}_m$. 
\end{itemize}
\vskip 5pt

The above exact sequence induces a corresponding map of Langlands dual groups:
\[   \phi^{\vee}:  H^{\vee}  \longrightarrow G^{\vee}. \] 

 The map $\phi^{\vee}$  gives rise to a simple instance of weak Langlands functoriality from $H$ to $G$, as the following folklore proposition documents. 

\vskip 5pt

\begin{prop} \label{P:isogeny}
Let $\phi : G \longrightarrow H$ be as in~\eqref{cd:defphiZQ}.
\begin{itemize}
\item[(i)] Assume that $v$ is a finite place of $k$ and $\pi_v$ is an unramified representation of $H(k_v)$. 
Then any constituent $\sigma_v$ of ${\pi_v}|_{G(k_v)}$ is unramified, and we have 
$c(\sigma_v) = \phi^\vee (c(\pi_v))$.\par \smallskip
\item[(ii)] If $\pi$ is an automorphic representation of $H$, and if $\sigma$ is any irreducible automorphic constituent 
of $\pi|_{G(\mathbb{A})}$, then $\sigma$ is a weak functorial lifting of $\pi$ with respect to $\phi^\vee$.
\end{itemize}
\end{prop}

Before we give the proof, recall that if $I$ is a split connected reductive group over a non-Archimedean local field $F$ (of characteristic $0$ in what follows), and if $\tau$ is an irreducible smooth representation of $I(F)$, we say that $\tau$ is {\it unramified} if it has nonzero invariant vectors under some compact open subgroup $K \subset I(F)$ 
which is {\it hyperspecial} in the sense of Tits \cite[\S 1.10]{tits}. When we want to specify $K$, as we shall do in the proof below, we rather say that $\tau$ is {\it $K$-unramified}. Hyperspecial (compact) subgroups exist as $I$ is split \cite[\S 1.10.2]{tits}. 
We shall use below their following properties, in which $\mathcal{O}$ denotes the valuation ring of $F$ :
\vskip 5pt
 (HSa) the hyperspecial subgroups of $I(F)$ are exactly the subgroups of the form $f(\mathcal{I}(\mathcal{O}))$ where $(\mathcal{I},\,f : \mathcal{I} \times_\mathcal{O} F \overset{\ssim}{\longrightarrow} I)$ is a reductive $\mathcal{O}$-model of $I$ \cite[\S 3.8.1]{tits}, 
\vskip 5pt
 (HSb)  the natural action $(\varphi,K) \mapsto \varphi(K)$ of ${\rm Aut}_F(I)$ on the subgroups of $I(F)$ preserves the hyperspecial subgroups, and $I_{\rm ad}(F) \subset {\rm Aut}_F(I)$ permutes them transitively \cite[\S 2.5]{tits}. 
\vskip 5pt

\noindent For example, if $I$ is a split torus, then its unique maximal compact subgroup is its unique hyperspecial subgroup by (HSa). The following lemma is presumably classical, but we provide a proof for lack of reference.
\vskip 5pt
\begin{lemma} 
\label{lem:existhyperspecial} Assume $\phi : G \rightarrow H$ is a morphism of split connected reductive groups over the non-Archimedean local field $F$ belonging to an exact sequence as in \eqref{cd:defphiZQ}. Then for any hyperspecial  subgroup $K$ of $H(F)$, there is a hyperspecial subgroup $K'$ of $G(F)$ with $\phi(K') \subset K$. 
\end{lemma}

\begin{proof} Observe first that it is enough to find some hyperspecial subgroups $A \subset G(F)$ and $B \subset H(F)$ such that $\phi(A) \subset B$. Indeed, let $K \subset H(F)$ be hyperspecial. By (HSb) we may choose $h \in H_{\rm ad}(F)$ with $K=hBh^{-1}$. But $\phi$ induces by \eqref{cd:defphiZQ} an isomorphism $G_{\rm ad} \simeq H_{\rm ad}$, so there is $g \in G_{\rm ad}(F)$ with $\phi(g)=h$, and $K':=gAg^{-1} \subset G(F)$ works.
\vskip 5pt
As a second and independent observation, we claim that we may reduce the lemma to the special case where $\phi$ is surjective (hence $S$ trivial).
Indeed, let us denote by $C$ the neutral component of the center of $H$, so that $C$ is a split torus.
Then the morphism $\phi': C \times G \rightarrow H, (c,g) \mapsto c\phi(g)$ is surjective with central kernel. 
Now any hyperspecial compact subgroup of $C(F) \times G(F)$ has the form $A' \times A$ with $A'$ and $A$ hyperspecial subgroups in $C(F)$ and $G(F)$ respectively, {\it e.g.} by (HSb). Hence, if the desired result is known for the surjective morphism $\phi'$, we can find a hyperspecial subgroup $B$ of $H(F)$  such that $\phi(A) \subset \phi'( A' \times A) \subset B$, this proving the desired result for $\phi$.

\vskip 5pt
It remains to prove the lemma when $\phi$ is surjective. In this case, $\phi$ factors as $$G \overset{f}{\longrightarrow} G/Z \overset{g}{\longrightarrow} H,$$ with $f$ the canonical map and $g$ an isomorphism. As the lemma is obvious for isomorphisms, we may and do assume $\phi=f$. Choose $\mathcal{G}$ a reductive (necessarily split) $\mathcal{O}$-model of $G$, let $\mathcal{T}$ be a maximal $\mathcal{O}$-split torus in $\mathcal{G}$ and set $T= \mathcal{T} \times_{\mathcal{O}} F$. The closed subgroup $Z \subset T$ certainly has an $\mathcal{O}$-model $\mathcal{Z} \subset \mathcal{T}$ (with same character group as $Z$).  By \cite[Prop. 4.3.1 (i)]{sga}, the quotient group scheme $\mathcal{G}/\mathcal{Z}$ is reductive over $\mathcal{O}$, and $\mathcal{G} \longrightarrow \mathcal{G}/\mathcal{Z}$ is a model of $\phi$ over $\mathcal{O}$.
By (HSa), the subgroups $A:=\mathcal{G}(\mathcal{O})$ and $B:=(\mathcal{G}/\mathcal{Z})(\mathcal{O})$ are hyperspecial in $G(F)$ and $H(F)=(G/Z)(F)$ and satisfy $\phi(A) \subset B$, and we are done by the first observation.
\end{proof}

\begin{proof}  (Of Proposition~\ref{P:isogeny}) 
Part (i) trivially implies (ii), so we focus on (i). We first prove the first assertion of (i).
Let $K$ be a hyperspecial subgroup of $H(k_v)$ such that 
$\pi_v$ is $K$-unramified. 
By Lemma~\ref{lem:existhyperspecial}, there is a hyperspecial subgroup $K'$ of $G(k_v)$ 
such that $\phi(K') \subset K$. 
As $k_v$ is a local field, the normal subgroup $\phi(G(k_v)) \subset H(k)$ generates, together with the center of $H(k_v)$, a finite index subgroup of $H(k_v)$, by Tate's finiteness of Galois cohomology. 
It follows that ${\pi_v}|_{G(k_v)}$ is a finite direct sum of irreducible representations which are permuted transitively by $H(k_v)$. These constituents are even $G_{\rm ad}(k_v)$-conjugate, as $\phi$ induces an isomorphism $G_{\rm ad} \simeq H_{\rm ad}$. 
As $\phi(K') \subset K$, some constituent of ${\pi_v}|_{G(k_v)}$ is $K'$-unramified,
so any constituent of ${\pi_v}|_{G(k_v)}$ is unramified with respect to a suitable $G_{\rm ad}(k_v)$-conjugate 
of $K'$. \par
We now prove the second assertion of (i).
Fix a constituent $\sigma$ of ${\pi_v}|_{G(k_v)}$. Choose a Borel pair $(B,T)$ defined over $k$ in $H$ and set $B'=\phi^{-1}(B)$ and $T'=\phi^{-1}(T)$ in $G$.
Then $\pi_v$ is an irreducible constituent of the (normalized) principal series 
${\rm Ind}_{B(k_v)}^{H(k_v)} \chi$ for some unramified character $\chi : T(k_v) \rightarrow \C^\times$.
Recall that $T$ and $T'$ being split, we have the decomposition $H(k_v) = T(k_v) \phi(G(k_v))$, by a standard Galois cohomology argument using Hilbert 90. 
% Replacing G with G x T if necessary, we may assume phi is surjective. 
% I write k for kv, and kb for an algebraic closure of k.
% Fix h in H(kb). 
% There is g in G(kb) mapping to h. 
% So for s in galois(kb/k), a(s) := s(g)g^{-1} a 1-cocycle in Z(kb).
% But there is also an exact sequence 1 -> Z -> T' -> T -> 1. 
% By Hilbert 90, H^1(T')=0, so any 1-cocycle in Z(kb) has the form 
% s(t)t^{-1} for some t' in T'(kb) mapping to some t in T(k).
% So we have s(t^{-1}g)=t^{-1}g for all s, and t^{-1}g is in G(k).
% So h = phi(g) = phi( t) phi(t^{-1} g) is in T(kb) phi(G(k)).
% As h is in H(k), we have phi(t) in T(k), QED.
This shows that the map $f \mapsto f \circ \varphi$ 
induces an injective $G(k_v)$-equivariant map
$$\left(\,\,{\rm Ind}_{B(k_v)}^{H(k_v)} \,\,\chi\,\,\right)|_{G(k_v)} \longrightarrow {\rm Ind}_{B'(k_v)}^{G(k_v)} \,\,\chi',$$ 
with $\chi' = \chi \circ \phi$, an unramified character of $T'(k_v)$, and thus $\sigma$ is a constituent of the right-hand side. This shows that $c(\sigma)=\phi^\vee(c(\pi_v))$ by the properties of the Satake isomorphism.
\end{proof}

\subsection{\bf Spin lifting.}  \label{SS:spin}
Consider the split group $\PGSp_{2n}$ whose dual group is $\Spin_{2n+1}(\C)$. Recall from \S \ref{SS:spingp} that one has the standard morphism
\[  \rho_{\C} :  \Spin_{2n+1}(\C) \longrightarrow \SO_{2n+1}(\C) \]
and by Proposition \ref{P:isogeny}, the associated functorial lifting is simply the restriction of automorphic forms 
\[  \begin{CD} 
\mathcal{A}(\PGSp_{2n}) @>{\rm rest.}>> \mathcal{A}(\Sp_{2n}). \end{CD} \]
On the other hand, by \S \ref{SS:spingp} again, one has the faithful Spin representation
\[  {\rm spin}:  \Spin_{2n+1}(\C) \longrightarrow \GL_{2^n}(\C). \]
Accordingly, there should be an associated Langlands functorial lifting 
\[  \mathcal{A}(\PGSp_{2n}) \longrightarrow \mathcal{A}(\GL_{2^n}) \]
from automorphic forms of $\GSp_{2n}$ to $\GL_{2^n}$; we call this the {\it Spin lifting} of $\PGSp_{2n}$.  
 \vskip 5pt
 
 \subsection{\bf The case $n=3$}  \label{SS:n=3}
 We shall now specialize to the case $n=3$.  
 In this case, one can describe the spin representation of $\Spin_7(\C)$ using triality.  
 The point is that there are 3 conjugacy classes of embeddings 
\[  \Spin_7(\C) \longrightarrow \Spin_8(\C) \]
permuted by the triality automorphism of $\Spin_8(\C)$.  
These 3 conjugacy classes of embeddings are characterized by their intersection with the center $Z_{\Spin_8} \simeq \mu_2 \times \mu_2$ of $\Spin_8(\C)$.
One may fix such an embedding $\iota$ so that the standard representation of $\Spin_7(\C)$ is compatible with the map $f_1^{\vee}$ (dual to the standard $f_1: \SO_8 \longrightarrow \PGSO_8$), in the sense that one has  a commutative diagram:
\begin{equation} \label{D:spin}
 \begin{CD}
  \Spin_7(\C) @>\iota>> \Spin_8(\C) \\
 @V{\rho_{\C}}VV   @VVf_1^{\vee}V \\
 \SO_7(\C)    @>>>  \SO_8(\C) 
   \end{CD} \end{equation}
We may designate the map $f_1^{\vee}$ as the standard morphism. On the other hand,  recall from the discussion in \S \ref{SS:VQ} that there are two other isogenies $f_2, f_3 : \SO_8 \longrightarrow \PGSO_8$ which factor as $f_2 = \theta \circ f_1$ and $f_3 = \theta^2 \circ f_1$, with $\theta$ a triality automorphism of $\PGSO_8$.
These induce dual isogenies
$$f_2^\vee : \Spin_8(\C) \overset{\theta^\vee}{\longrightarrow} \Spin_8(\C)  \overset{f_1^\vee}{\longrightarrow} \SO_8(\C),$$
and likewise $f_3^{\vee}$, with $\theta^\vee$ a triality automorphism of $\Spin_8(\C)$. We may designate $f_2^{\vee}$ and $f_3^{\vee}$ as the half-spin representations  of $\Spin_8(\C)$. By our discussion in \S \ref{SS:spingp}, the restriction of these  half-spin representations,  via $\iota : \Spin_7(\C) \rightarrow \Spin_8(\C)$, give rise to the spin representation of $\Spin_7(\C)$. In other words,  the spin representation of $\Spin_7(\C)$ is given by the composite map
\[  \begin{CD}
\Spin_7(\C) @>\iota>>  \Spin_8(\C)  @>f_2^{\vee}>> \SO_8(\C) @>{\rm std}>> \GL_8(\C).  \end{CD} \]
This then suggests a  construction of the spin lifting which is summarised by the following sequence of liftings:
\[  \begin{CD}
\mathcal{A}(\PGSp_6) @>\iota_*>>\mathcal{A}(\PGSO_8) @>f_2^*>>  \mathcal{A}(\SO_8)  @>\text{\cite{A, CKPSS}}>> \mathcal{A}(\GL_8) \end{CD} \]
 As we explain next, the functorial lifting $\iota_*$  can be constructed by the similitude theta correspondence.
\vskip 5pt

 \vskip 5pt

\subsection{\bf Theta correspondence.}
For our discussion of theta correspondence, we 
consider the dual reductive pair $\Sp_{2n} \times \O_{2m}$  of symplectic and orthogonal groups associated to a skew-symmetric and quadratic space of dimension $2n$ and $2m$ respectively. 
 For simplicity, assume that  the quadratic space underlying $\O_{2m}$ has  trivial discriminant. In our applications later on in the paper, we will assume that $m  >n$.
\vskip 5pt

Given a nontrivial additive character $\psi$ of $k \backslash \A$, 
the dual pair $\Sp_{2n} \times \O_{2m}$ is equipped with a Weil representation $\Omega_{\psi}$. One has an automorphic realization 
\[  \theta: \Omega_{\psi} \longrightarrow C^{\infty}([\Sp_{2n} \times \O_{2m}]) \]
given by the formation of theta series.
The global theta lifting of isometry groups is an equivariant  map
\[ \begin{CD}
\Theta:  \Omega_{\psi} \otimes \overline{\mathcal{A}_{cusp}(\Sp_{2n})} @>>> \mathcal{A}(\O_{2m})  \end{CD} \]
defined by
\[  \Theta(\phi, f)(h)  = \int_{[\Sp_{2n}]} \theta(\phi)(gh) \cdot \overline{f(g)} \, dg, \quad \text{for $\phi \in \Omega_{\psi}$ and $f \in \mathcal{A}_{cusp}(\Sp_{2n})$.} \]
\vskip 5pt

It is known that this theory of theta correspondence can be extended to the setting of the similitude dual pair $\GSp_{2n} \times \GO_{2m}$. More precisely, the Weil representation 
has a natural extension to the group
\[  (\GSp_{2n} \times \GO_{2m})^{\sim} = \{ (g,h) \in \GSp_{2n} \times \GO_{2m}: \sim(g) \cdot \sim (h) =1 \}. \]
Observe that this group sits in the short exact sequences:
\[  \begin{CD}
1 @>>> \Sp_{2n} @>>>  (\GSp_{2n} \times \GO_{2m})^{\sim} @>p>> \GO_{2m} @>>> 1. 
\end{CD} \]
\[   \begin{CD}
1 @>>> \O_{2m} @>>>  (\GSp_{2n} \times \GO_{2m})^{\sim} @>q>> \GSp_{2n} @>>> 1
\end{CD} \]
where $p$ and $q$ are the natural projections on the two factors.
Hence,  for $\phi \in \Omega_{\psi}$ and a cusp form $f \in \mathcal{A}_{cusp}([\GSp_{2n}])$ with a fixed central character, one can define an automorphic form $\theta(\phi,f)$ on $\GO_{2m}$ by:
\[ \Theta(\phi, f)(h) = \int_{[\Sp_{2n}]} \theta(\phi)(t_h g, h) \cdot \overline{f(g)} \, dg,  \]
for any $t_h \in \GSp_{2n}(\A)$ such that  $(t_h,h) \in (\GSp_{2n}(\A) \times \GO_{2m}(\A))^{\sim}$.  Note that the above definition is independent of the choice of $t_h$.
It can be more elegantly expressed as:
\[  \Theta(\phi,f) =  p_{!} \left( \theta(\phi) \cdot q^*(\overline{f}) \right). \]
The automorphic form  $\Theta(\phi,f)$ has the same central character as $f$, on identifying the centers of $\GO_{2m}$ and $\GSp_{2n}$ with $\mathbb{G}_m$ via their action on the underlying quadratic and skew-symmetric spaces. 
\vskip 5pt

To summarize,  one has a 
commutative diagram of global theta liftings, with the vertical arrows given by pullback and restriction of automorphic forms:
\vskip 5pt

\[
\begin{CD}
 \Omega_{\psi} \otimes  \overline{\mathcal{A}_{cusp}(\PGSp_{2n}) }@>\Theta>> \mathcal{A}(\PGO_{2m})  \\
 @VVV   @VVV  \\
\Omega_{\psi} \otimes  \overline{\mathcal{A}_{cusp}(\GSp_{2n})} @>\Theta>> \mathcal{A}(\GO_{2m})  \\
@VVV  @VVV \\
 \Omega_{\psi} \otimes \overline{\mathcal{A}_{cusp}(\Sp_{2n})} @>\Theta>> \mathcal{A}(\O_{2m})  
\end{CD} 
\]
If $\pi \in \mathcal{A}_{cusp}(\GSp_{2n})$ or $\mathcal{A}_{cusp}(\Sp_{2n})$ is a cuspidal representation, then its global theta lift to $\GO_{2m}$ or  $\O_{2m}$ is the subrepresentation
\[  \Theta(\pi) = \langle \Theta(\phi,f): \phi \in \Omega_{\psi}, \, f \in \pi \rangle \subset \mathcal{A}(\GO_{2m}) \quad \text{or} \quad \mathcal{A}(\O_{2m}). \]
%If $\Theta(\pi)$ is nonzero and contained in the space of square-integrable automorphic forms, then $\Theta(\pi)$ is irreducible (by the Howe duality theorem). 
\vskip 5pt

  For our purpose of constructing the spin lifting, we shall consider the case $m = n+1$. In this case, one knows that if $\pi \in \mathcal{A}_{cusp}(\GSp_{2n})$ is a cuspidal representation which is globally generic, then $\Theta(\pi)$ on $\GO_{2n+2}$ is globally generic and thus is nonzero \cite{GRS1}.
  % In general, such issues of nonvanishing may be addressed by the Rallis inner product formula \cite{GQT}, for example. 
\vskip 5pt
\vskip 5pt

There is also an analogous theory of local (isometry and similitude) theta correspondence, for which we refer the reader to \cite{Ro} and \cite{GT1, GT2}. 
Another result we need is the  local theta correspondence of unramified representations. More precisely, we have:
 \vskip 5pt
 
 \begin{prop}  \label{P:unram}
 Assume that $\Sp_{2n} \times \O_{2n+2}$ is an unramified dual pair over a non-Archime\-dean local field $F$ of characteristic $0$.
 Let $\pi$ be an unramified irreducible representation of $\GSp_{2n}(F)$ and consider its local (small) theta lift $\theta(\pi)$ on $\GO_{2n+2}(F)$. One has:
 \vskip 5pt
 
 \begin{itemize}
 \item[(i)] $\theta(\pi)$    is nonzero, irreducible and unramified; 
 \item[(ii)] $\theta(\pi)$ has the same central character as $\pi$;
 \item[(iii)] $\theta(\pi)$ remains irreducible when restricted to $\GSO_{2n+2}(F)$. 
 \end{itemize}
 Hence, the local theta correspondence gives a map
 \[  \theta: \Irr_{ur, \chi}(\GSp_{2n}) \longrightarrow \Irr_{ur, \chi}(\GSO_{2n+2}) \]
 where $\Irr_{ur,\chi}$ denotes the set of irreducible unramified representations with central character $\chi$.
 At the level of Satake parameters, this map is given  by the top arrow $\iota$ in  the following natural diagram  of dual groups:
 \begin{equation} \label{D:theta}
   \begin{CD}
\GSpin_{2n+1}(\C) @>\iota>>  \GSpin_{2n+2}(\C)  \\
@V\rho'VV  @VV\rho'V  \\
\SO_{2n+1}(\C)  @>\iota^{\flat}>> \SO_{2n+2}(\C).  \end{CD} \end{equation}
 \end{prop}
 
 \noindent So as not to disrupt our discussion, the proof of this proposition is given in Appendix A at the end of the paper.  
 We should also mention that the global and local (similitude) theta correspondences are compatible, in the following sense \cite[Prop. 3.1]{G}:
 \vskip 5pt
 
 \begin{prop} \label{P:compa}
 Suppose $\pi$ is a cuspidal automorphic representation of $\GSp_{2n}(\A)$ such that its theta lift $\Theta(\pi)$ to $\GO_{2m}(\A)$ is nonzero and contained in the space of square-integrable automorphic forms (with fixed central character). Then $\Theta(\pi)$ is irreducible and for all places $v$, 
 \[  \Theta(\pi)_v  \simeq \theta(\pi_v) \]
 where the RHS is the local theta lift of $\pi_v$.  
 More generally (i.e. without assuming that $\Theta(\pi)$ is contained in the space of square-integrable automorphic forms), for any irreducible subquotient $\sigma$ of 
$\Theta(\pi)$, 
\[  \sigma_v \simeq \theta(\pi_v) \quad \text{ for almost all $v$.} \]
  \end{prop}

 \vskip 5pt

% Suppose that $\pi \subset \mathcal{A}_{cusp}(\PGSp_{2n})$ is a cuspidal representation such that the restriction of $\pi$ to $\Sp_{2n}$ belongs to  a tempered (or generic) 
%A-packet.  
%By  results of B. Xu \cite{X}, the cuspidal representation $\pi$ is nearly equivalent to a globally generic cuspidal representation $\pi'$ of $\PGSp_{2n}$. Thus, every such $\pi$ has a weak Langlands functorial lift to $\PGSO_{2n+2}$, given by the nonzero global theta lift of $\pi'$ to $\PGSO_{2n+2}$.

\vskip 5pt

\subsection{\bf Spin lifting for  $n=3$.}  
\label{subsect:spinliftingnequals3}
We now specialise to the case $n=3$. We shall prove the following theorem which extends \cite[Thm. 2.3(ii)]{CL}:
\vskip 5pt

\begin{thm}
\label{thm:weakspinlift}
Let $\pi$ be a cuspidal representation of $\PGSp_6$ whose restriction to $\Sp_6$ possesses a generic A-parameter. 
Then the (weak) Spin lifting of $\pi$ to $\GL_8$ exists. 
\end{thm}
\vskip 5pt

Let us first establish a lemma which may be of independent interest.

\vskip 5pt
\begin{lemma} \label{L:generic}
Let $\pi$ be a cuspidal representation of $\GSp_{2n}$ whose restriction to $\Sp_{2n}$ possesses a generic A-parameter. 
\vskip 5pt

\noindent (i) The cuspidal representation $\pi$ is nearly equivalent to a globally generic cuspidal representation  $\pi'$ of $\GSp_{2n}$.
\vskip 5pt

\noindent (ii) Fix a place $v$ of $k$ and $\pi'$ as in (i). If $\pi_v$ is unramified, then $\pi'_v \simeq \pi_v$.
\end{lemma}

\begin{proof}
(i)  Consider 
 \[  \pi|_{\Sp_{2n}} := \{ f|_{\Sp_{2n}(\A)} : f \in \pi \} \subset \mathcal{A}_{cusp}(\Sp_{2n}). \]
 By hypothesis, this submodule of $\mathcal{A}_{cusp}(\Sp_{2n})$ has a generic A-parameter $\Psi_{\flat}$ with an associated submodule  $\mathcal{A}_{\Psi_{\flat}} \subset  \mathcal{A}_{cusp}(\Sp_{2n})$ (the global A-packet), so that
 \[  \pi|_{\Sp_{2n}}  \subset  \mathcal{A}_{\Psi_{\flat}} \subset \mathcal{A}_{cusp}(\Sp_{2n}). \]
 As shown by Ginzburg-Rallis-Soudry via automorphic descent \cite{GRS2},  the A-packet $\mathcal{A}_{\Psi_{\flat}}$ contains a globally generic cuspidal representation $\pi'_{\flat}$. In particular, the irreducible summands of $\pi|_{\Sp_{2n}}$ are nearly equivalent to $\pi'_{\flat}$.  In \cite{X2}, Bin Xu has constructed a global A-packet 
 \[ \tilde{\mathcal{A}}_{\Psi_{\flat}}  \subset \mathcal{A}_{cusp}(\GSp_{2n}), \]
 such that 
  \[  \tilde{\mathcal{A}}_{\Psi_{\flat}} |_{\Sp_{2n}} = \mathcal{A}_{\Psi_{\flat}}. \]
  Moreover, the global A-packet $\tilde{\mathcal{A}}_{\Psi_{\flat}}$ is unique up to twisting by automorphic quadratic characters.
  Up to replacing $\tilde{\mathcal{A}}_{\Psi_{\flat}}$ by an automorphic quadratic twist, we may thus assume that $\pi \subset \tilde{\mathcal{A}}_{\Psi_{\flat}}$. 
   Now there is some irreducible summand  $\pi' \subset  \tilde{\mathcal{A}}_{\Psi_{\flat}}$ such that 
   \[  \pi'_{\flat} \subset \pi'|_{\Sp_{2n}}. \]
  Then $\pi$ is nearly equivalent to $\pi'$, which is globally generic (since $\pi'_{\flat}$ is). 
  \vskip 5pt
  
  \noindent (ii) The representations $\pi$ and $\pi'$ both belong to the generic global A-packet $\tilde{\mathcal{A}}_{\Psi_{\flat}}$ introduced above. By \cite{X2}, the members of $\tilde{\mathcal{A}}_{\Psi_{\flat}}$ (that 
  Xu also calls a global $L$-packet) are constructed as a restricted tensor product, over all places $w$ of $k$, of the members of local L-packets $\tilde{\mathcal{A}}_{\Psi_{\flat,w}}$ defined by B. Xu in \cite{X1}. By \cite[Thm. 4.6]{X1} (see also the discussion {\it loc.\,cit.} after Prop. 3.10), for each finite place $w$, the irreducible constituents of the restriction of the members of $\tilde{\mathcal{A}}_{\Psi_{\flat,w}}$ form a local L-packet $\Pi^\flat_{w}$ of $\Sp_{2n}(k_w)$ (as defined in \cite{A}).
Moreover, the restriction to ${\rm Sp}_{2n}(k_w)$ also induces a bijection between 
  $\tilde{\mathcal{A}}_{\Psi_{\flat,w}}$ and the ${\rm GSp}_{2n}(k_w)$-orbits in $\Pi^\flat_{w}$, by~\cite[Prop. 4.4 (2)]{X1}. 
  
  \vskip 5pt
  
  Assume $\pi_w$ is unramified. An argument given in the proof of Proposition~\ref{P:isogeny} (i) shows that the restriction of $\pi_w$ to $\Sp_{2n}(k_w)$ is a finite direct sum of unramified representations which are permuted transitively by $\GSp_{2n}(k_w)$. But by the unramified case of the local Langlands correspondence for ${\rm Sp}_{2n}(k_w)$ in \cite{A}, the set of these constituents is the full L-packet $\Pi^\flat_{w}$. By the properties of $\tilde{\mathcal{A}}_{\Psi_{\flat,w}}$ recalled above, this shows $\tilde{\mathcal{A}}_{\Psi_{\flat,w}}\,=\,\{ \pi_w\}$. 
  %Indeed, if two irreducible representations of $\GSp_{2n}(k_w)$ contains a same irreducible representation of ${\rm Sp}_{2n}(k_w)$ in their restriction, they are twist of each others by a character by \cite[\S 2.3]{X1}, hence are either isomorphic or in two different Xu's packets by \cite[Thm. 4.6 (1)]{X1}.  
  % also, \tilde{\mathcal{A}}_{\Psi_{\flat,w}} has the same central character as \pi, which defines Xu's \zeta:
  % we omit the discussion here.
 \end{proof}

We can now prove Theorem \ref{thm:weakspinlift}.

\begin{proof} (of Theorem \ref{thm:weakspinlift})
By Lemma \ref{L:generic} (i), 
$\pi$ is nearly equivalent to a globally generic cuspidal representation $\pi'$ of $\PGSp_6$.
  By \cite{GRS1},  the global theta lift $\Theta(\pi')$ is a nonzero globally generic (not necessarily cuspidal) automorphic representation of $\PGSO_8$. In any case,
by the commutative diagram (\ref{D:theta}) and Proposition \ref{P:compa}, 
the Hecke-Satake family of $\Theta(\pi')$ is the image of $c(\pi')=c(\pi)$
under the natural map $\iota : {\rm Spin}_7(\C) \rightarrow {\rm Spin}_8(\C)$, and
the pullback of $\Theta(\pi')$ to $\SO_8$ via the natural map $ f_1: \SO_8  \rightarrow  \PGSO_8$ has A-parameter 
\[  \Psi'  = \Psi \boxplus {\bf 1}, \]
where $\Psi$ is the A-parameter of $\pi|_{\Sp_6}$.
On the other hand, the discussion in \S \ref{SS:spin} and \S \ref{SS:n=3}, together with Proposition~\ref{P:isogeny} (ii), show that using the pullback of $\Theta(\pi')$ via $f_2:  \SO_8 \rightarrow \PGSO_8$, the map
\[  \pi \mapsto  \text{the constituents of $f_2^*(\Theta(\pi'))$} \]
exhibits the Spin lifting from $\PGSp_6$ to $\SO_8$.  Composing this with the lifting from $\SO_8$ to $\GL_8$ due to \cite{CKPSS} or \cite{A}, we obtain the desired Spin lifting from $\PGSp_6$ to $\GL_8$.   
\end{proof}
\vskip 5pt

  In fact, by being more careful with the above proof, one has the following slight strengthening of Theorem \ref{thm:weakspinlift}, which says that  the weak Spin lifting provided by Theorem \ref{thm:weakspinlift} is {\it strong} at unramified places and Archimedean places.

 \vskip 5pt
 
 \begin{thm} \label{T:strong}
Let  $\pi \subset \mathcal{A}_{cusp}(\PGSp_6)$ be as in Theorem \ref{thm:weakspinlift}  and let $\sigma$ be its automorphic weak spin lift constructed on $\GL_8$ therein.  
 \vskip 5pt
 
 \noindent (i) For each finite place $v$ of $k$, 
  \[    \text{$\pi_v$ unramified}  \Longrightarrow \text{$\sigma_v$ unramified}  \]
  with $c(\sigma_v) = {\rm spin}(c(\pi_v))$.
    
    \vskip 5pt
    
    \noindent (ii) Let  $v$ be an Archimedean place of $v$ and let $c(\pi_v)$ be the infinitesimal character of $\pi_v$, regarded as a semisimple element in $\mathfrak{spin}_7 = {\rm Lie}(\PGSp_6^{\vee}) \otimes_{k_v} \C$. Then the infinitesimal character of $\sigma_v$ is given by
  \[  c(\sigma_v) = {\rm d}{\rm spin} (c(\pi_v)) \in \mathfrak{gl}_8. \]
  \vskip 5pt
  
  \noindent (iii)  Assume that  the A-parameter of $\pi|_{\Sp_6}$ does not contain the trivial representation ${\bf 1}$. Then the automorphic representation $\sigma$ of $\GL_8$ is an isobaric sum   
  \[  \sigma = \sigma_1 \boxplus \sigma_2 \boxplus\dots  \boxplus \sigma_k \]
where each $\sigma_i$ is a self-dual cuspidal representation of some $\GL_{n_i}$ which is of orthogonal type (i.e. $L^S(s, \sigma_i, {\rm Sym}^2)$ has a pole at $s=1$) and $\sigma_i \ne \sigma_j$ if $i \ne j$.
    \end{thm}
 
 \vskip 5pt
 \begin{proof}
 (i)  Consider the globally generic $\pi'$ provided by Lemma \ref{L:generic}(i) and used in the proof of Theorem \ref{thm:weakspinlift}, so that $\pi$ and $\pi'$ both belong to the global A-packet $\tilde{\mathcal{A}}_{\Psi_{\flat}}$.   By Lemma \ref{L:generic}(ii), 
 $\pi'_v \simeq \pi_v$ is unramified and generic. It follows by Proposition \ref{P:unram} and Proposition \ref{P:compa} that 
 the global theta lift $\Theta(\pi')$ is  unramified at the place $v$ with 
 \[ c(\Theta(\pi')_v)  = c(\theta(\pi_v)) =  \iota(c(\pi_v)),\]
where $\iota: \Spin_7(\C) \hookrightarrow \Spin_8(\C)$ is as in (\ref{D:spin}).
In the proof of Theorem \ref{thm:weakspinlift}, the weak lift $\sigma$ on $\GL_8$ is obtained  by considering $f_2^*(\Theta(\pi'))$  on $\SO_8$  followed by the Arthur transfer ${\rm std}_*$ from $\SO_8$ to $\GL_8$. Both of these respect unramified representations (or rather unramified L-packets) with the expected effect on their Satake parameters: for $f_2^*$ this follows from Proposition~\ref{P:isogeny} (i), and for ${\rm std}_*$ from \cite{A}. Thus we deduce that $\sigma_v$ is unramified  with 
\[  c(\sigma_v) = {\rm std}(f_2^{\vee}( \iota (c(\pi_v)))) = {\rm spin}(c(\pi_v)), \]
 as desired.
  
 \vskip 10pt
 
 \noindent (ii) For an Archimedean place $v$, note that the elements in the local L-packet $\tilde{\mathcal{A}}_{\Psi_{\flat,v}}$ all have the same infinitesimal character.
 In particular, $c(\pi_v) = c(\pi'_v)$.  We also know that  $c(\theta(\pi'_v)) = d\iota(c(\pi'_v))$ by the correspondence of infinitesimal character under local theta correspondence (see \cite{Pr, Li}).  Since infinitesimal characters behave in the expected way under pulling back by an isogeny and the Arthur transfer from $\SO_8$  to $\GL_8$ (by \cite{A}), we deduce the desired statement in (ii).
 
 \vskip 5pt
 
 \noindent (iii) If the A-parameter of $\pi|_{\Sp_6}$ does not contain the trivial representation ${\bf 1}$, then the nonzero global theta lift $\Theta(\pi')$ in the proof of Theorem \ref{thm:weakspinlift} is a globally generic cuspidal representation of $\PGSO_8$ (by \cite{GRS1}).  Under pullback by $f_2$, $f_2^*(\Theta(\pi'))$  contains a globally generic cuspidal representation of $\SO_8$ and hence has a  generic discrete A-parameter, i.e. its transfer to $\GL_8$ is a multiplicity-free isobaric sum of self-dual cuspidal representations of orthogonal type (by \cite{A} or \cite{CKPSS, GRS2}).
   \end{proof}
 \vskip 5pt
 
%what do you think about this option ?
 \begin{remark} {\rm If the A-parameter $\Psi$ of $\pi|_{\Sp_6}$ contains the trivial representation ${\bf 1}$, say $\Psi = {\bf 1} \boxplus \Psi'$, then $\Theta(\pi')$ is globally generic but not cuspidal (by \cite{GRS1}). Indeed, under pullback by $f_1$, $f_1^*(\Theta(\pi'))$ has A-parameter ${\bf 1} \boxplus \Psi = {\bf 1} \boxplus {\bf 1} \boxplus \Psi'$, which contains the trivial representation ${\bf 1}$ two times and hence is not a discrete A-parameter. In this case, the pullback $f_2^*(\Theta(\pi'))$ under $f_2$ has A-parameter of the form $\tau \boxplus \tau^{\vee}$ for a cuspidal representation $\tau$ of $\GL_4$. Since we do not need this case later on, we will skip the details here.}
 \end{remark} 
 \vskip 5pt
 
 In \S \ref{S:Siegel}, we will consider the case of cuspidal representations of $\PGSp_6$ associated to  holomorphic Siegel cusp forms of level 1.  In that case, we shall give a full determination of the possible shapes of the A-parameter of the Spin lift of $\pi$.  

%In a companion paper, we have used this same idea in the local setting to construct the local Langlands correspondence for $\PGSp_6(F)$ when $F$ is a $p$-adic field.

  \vskip 10pt
  
 \section{\bf Variants} 
We may exploit the above combination of similitude theta correspondence  and triality  in a couple of other situations, providing cheap constructions of interesting 
nontempered discrete automorphic representations of $\PGSO_8$. More precisely, we may consider the global theta lifting associated to the tower:

 \[
\xymatrix@R=10pt{
 \PGSp_6  \ar@{-}[rd]  & \\
\PGSp_4  \ar@{-}[r]  &  \PGSO_8  \\
\PGSp_2   \ar@{-}[ru]   &    
}
\]
We have considered the theta lifting from $\PGSp_6$ to $\PGSO_8$ in the previous section.  Let us consider the theta lifting from the two smaller symplectic similitude groups in this section.
\vskip 5pt

\subsection{\bf Theta lifting} 
The isometry  theta lifting from $\Sp_2$ to $\SO_8$  takes a cuspidal representation $\pi$ of $\Sp_2$ to the near equivalence class on $\SO_8$ given by the nontempered A-parameter
\[   \Psi_{\pi}   \boxplus    S_5,   \]
where $\Psi_{\pi}$ is the L-parameter of $\pi$, which is a self-dual cuspidal representation of $\GL_3$, and $S_k$ will denote the $k$-dimensional irreducible representation of the Arthur $\SL_2$. Similarly, the isometry theta correspondence from $\Sp_4$ to $\SO_8$ takes a cuspidal representation $\pi$ of $\Sp_4$ to the near equivalence class given by the A-parameter
\[   \Psi_{\pi} \boxplus S_3, \]
where now $\Psi_{\pi}$ is a self-dual automorphic representation of $\GL_5$  (the A-parameter of $\pi$).

\vskip 5pt

In both these cases, we see that the Hecke-Satake family of $\Theta(\pi)$ (or equivalently its A-parameter)
is valued in $\SO_3(\C) \times \SO_5(\C) \subset \SO_8(\C)$.  If one considers the similitude theta lifting from $\PGSp_2$ or $\PGSp_4$ to $\PGSO_8$, then  as we show in Appendix A (Proposition \ref{P:unram2}), the corresponding functoriality is given by the top row of the commutative diagram:
\[  \begin{CD}
\SL_2(\C) \times \Sp_4(\C) @= \Spin_3(\C) \times \Spin_5(\C)  @>j>> \Spin_8(\C) \\
@.   @VVV  @VVf_1^{\vee}V  \\
 @.  \SO_3(\C) \times \SO_5(\C)  @>>> \SO_8(\C).  \end{CD} \]
More precisely, 
\begin{itemize}
\item if $\pi$ is a cuspidal representation of $\PGSp_2$ with Hecke-Satake family 
\[  c(\pi) = \{ c({\pi_v})\}_{v \notin S}  \subset \Spin_3(\C),\]
 and $c({\rm triv}) \subset \Spin_5(\C)$ is the Hecke-Satake family for the trivial  representation of $\PGSp_4$, then the Hecke-Satake family for the theta lift of $\pi$ to $\PGSO_8$ is
\[  \{ j  (c(\pi_v), c({\rm triv}_v))  \}_{v \notin S}  \subset  \Spin_8(\C). \]

\vskip 5pt

\item if $\pi$ is a cuspidal representation of $\PGSp_4$ with Hecke-Satake family $c(\pi) \subset \Spin_5(\C)$ and $c({\rm triv}) \subset \Spin_3(\C)$ is the Hecke-Satake family for the trivial representation of $\PGSp_2$, then the Hecke-Satake family of the theta lift of $\pi$ to $\PGSO_8$ is
\[   \{ j  (c({\rm triv}_v), c(\pi_v))\}_{v \notin S}  \subset  \Spin_8(\C). \]
\end{itemize}
 In other words, the Hecke-Satake family of the similitude theta lift is contained in the subgroup $j(\Spin_3(\C) \times \Spin_5(\C)) \subset \Spin_8(\C)$.
 \vskip 5pt
 
Further, as we are working with the split $\PGSO_8$, the global theta lift $\Theta(\pi)$  of a cuspidal representation $\pi$  of $\PGSp_2$ or $\PGSp_4$ is nonzero, because one is in the so-called stable range. Moreover, $\Theta(\pi)$ is an irreducible square-integrable automorphic representation (typically not cuspidal).
For the above statements, see \cite[Prop. 3.2]{G} and the references therein.
 Hence, in these cases, we do not need to impose further conditions (such as the restriction of $\pi$ to the isometry groups having a generic A-parameter) to ensure the nonvanishing of the global theta lifting. 
\vskip 5pt

\subsection{\bf Application of triality}
Now   observe that
\[  f_2^{\vee} \circ j  =  {\rm spin}_3 \boxtimes {\rm spin}_5  \]
where the RHS denotes the Spin representations of $\Spin_3(\C)$ and $\Spin_5(\C)$ respectively.  These are nothing but the standard representations of $\SL_2(\C)$ and 
$\Sp_4(\C)$ respectively. By this discussion (and Proposition \ref{P:unram2}), we deduce the following result, which extends \cite[Thm. 2.2 and Thm. 2.3(i)]{CL}:

\vskip 10pt

\begin{thm}  \label{T:ikeda}
(i) If $\pi$ is a cuspidal representation of $\PGL_2 \simeq \PGSp_2$ and  $\Theta(\pi)$ is its global theta lift to  the split $\PGSO_8$ (which is nonzero), then $f_2^*(\Theta(\pi))$  is a square integrable automorphic representation with 
nontempered A-parameter 
\[   \Psi_{\pi}  \boxtimes S_4  \]
where $\Psi_{\pi}=\pi$ denotes the A-parameter of $\pi$.  Moreover, if $\pi_v$ is unramified, then $[f_2^*(\Theta(\pi))]_v$ is also unramified, whose L-parameter (or Satake parameter) is the L-parameter associated to $\Psi_{\pi} \boxtimes S_4$.  
\vskip 5pt

(ii) If $\pi$ is a cuspidal representation of $\PGSp_4$ and $\Theta(\pi)$  is its global theta lift to the split  $\PGSO_8$ (which is nonzero), then $f_2^*(\Theta(\pi))$ 
is a square integrable automorphic representation with nontempered  A-parameter 
\[   \Psi_{\pi} \boxtimes  S_2, \]
where $\Psi_{\pi}$ denotes the A-parameter of $\pi$ viewed as a representation of ${\rm PGSp}_4 \simeq {\rm SO}_5$. Moreover, if $\pi_v$ is unramified, then $[f_2^*(\Theta(\pi))]_v$ is also unramified, whose L-parameter (or Satake parameter) is the L-parameter associated to $\Psi_{\pi} \boxtimes S_2$. 
\end{thm}

\vskip 5pt

\subsection{\bf Ikeda's lifting}
%{\rm (On the Ikeda lifting from ${\rm PGL_2}$ to ${\rm PGSp}_8$)} 
In the context of Theorem 4.1 (i),
we may further consider the global Theta lift $\Pi$ 
of $f_2^*(\Theta(\pi))$ to ${\rm Sp}_8$. 
If nonzero, it provides a rather cheap construction,
 of some automorphic representation of ${\rm Sp}_8$ 
with standard A-parameter $(\Psi_\pi \boxtimes S_4) \boxplus {\bf 1}$ (not relying on \cite{A} nor on \cite{CKPSS}). 
\vskip 5pt

Better,  in the construction of $\Theta(\pi)$, let us replace
the split ${\rm GSO}_8$ by ${\rm GSO}(V)$, 
where $V$ is an octonion $F$-algebra whose set $\Sigma$ 
of (necessarily real) non-split places is nonempty. 
By \cite{R2}, $\Theta(\pi)$ is nonzero if, and only if, 
$\pi_v$ is a holomorphic discrete series of weight $\geq 4$ for all $v \in \Sigma$. 
The triality automorphism (hence $f_2$) still exists on ${\rm PGSO}(V)$.  
Assuming again that the theta lift $\Pi$ of $f_2^*(\Theta(\pi))$ to ${\rm Sp}_8$ is nonzero,  
then $\Pi_v$ is a holomorphic discrete series for all $v \in V$, 
providing an alternative construction of the liftings in \cite{I} and \cite{IY} 
in the special case of ${\rm Sp}_8$ (under slightly different assumptions). 
\vskip 5pt

This idea was used in \cite[\S 5.4]{CL} to give an elementary proof 
that the Schottky form on ${\rm Sp}_8(\Z)$ is an Ikeda lift 
of the discriminant cusp form of weight $12$ on ${\rm SL}_2(\Z)$. 
For a general $\pi$, the non-vanishing of $\Pi$ may be addressed using, for example,  \cite[Cor. 7.9(c)]{GT3}.
   
\vskip 15pt
\section{\bf Rankin-Selberg Lifting $\PGSp_4 \times \PGL_2 \rightarrow \SO_8$}
The examples considered in the previous section all arise through an application of the triality automorphism to a nontempered  A-parameter factoring through 
the map
\[  \begin{CD}
  \Spin_3(\C) \times \Spin_5(\C) @>j>> \Spin_8(\C) \\
  @VVV  @VVV \\
  \SO_3(\C) \times \SO_5(\C)  @>j^{\flat}>>  \SO_8(\C). \end{CD} \]
The map $j$ of dual groups lies over the corresponding map $j^{\flat}$ which gives the twisted endoscopic transfer $\Sp_2 \times \Sp_4 \longrightarrow \SO_8$ (associated to an outer automorphism of order $2$) established in Arthur's work \cite{A}. In \cite{X1,X2}, Xu has constructed  this endoscopic transfer at the level of similitude groups associated to $j$. We summarise his results in this particular case:
\vskip 5pt

\begin{thm}
Let $\pi$ be a cuspidal representation of $\PGL_2 = \PGSp_2$ and let $\sigma$ be a cuspidal representation of $\PGSp_4$ with generic A-parameter. Then the endoscopic lifting of $(\pi , \sigma)$ associated to $j$ exists. 
\end{thm}

\vskip 10pt

Let $\Pi$ be a cuspidal representation of $\PGSO_8$ which is a weak lifting of $(\pi, \sigma)$  via $j$. 
Consider now the pullback of $\Pi$ to $\SO_8$ via $f_2$, followed by the transfer ${\rm std}_*$ to $\GL_8$. Then we have:
\vskip 5pt

\begin{cor}  \label{C:tensor}
On $\GL_8$, the representation ${\rm std}_*(f_2^*(\Pi))$  is a weak lift of $\pi \boxtimes \sigma$ relative to the tensor product map
\[  \begin{CD}
\SL_2(\C) \times \Sp_4(\C) @>\boxtimes>> \SO_8(\C)  @>{\rm std}>>   \GL_8(\C).  \end{CD} \]
\end{cor}
We may restate this as:
\vskip 5pt

\begin{thm} \label{T:RS}
If $\pi$ is a self-dual cuspidal representation of $\GL_2$ and $\sigma$ is a self-dual cuspidal representation of $\GL_4$ of symplectic type, then the weak Rankin-Selberg lifting $\pi \boxtimes \sigma$ exists as an automorphic representation of $\GL_8$. 
\end{thm}
\vskip 5pt

\begin{proof}
We examine the two cases:
\vskip 5pt

\begin{itemize}
\item if $\pi$ is of orthogonal type, then $\pi$ is obtained by automorphic induction from a Hecke character $\chi$ of a quadratic field extension $E$ of $k$, say 
\[  \pi  =  {\rm AI}_{E/k}( \chi). \]
One may consider the automorphic representation
\[   \Pi =  {\rm AI}_{E/k} ( {\rm BC}_{E/k}(\sigma) \otimes \chi), \]
where ${\rm BC}_{E/k}$ denotes the base change lifting with respect to $E/k$.  Then $\Pi$ is the Rankin-Selberg lift of $\pi \boxtimes \sigma$. Observe that no conditions need to be imposed on $\sigma$ here.
\vskip 5pt

\item the main case is when $\pi$ and $\sigma$ are both of symplectic type; this is precisely the case treated by Corollary \ref{C:tensor} above.

\end{itemize}
\end{proof}

\vskip 10pt

\section{\bf Spin Lifting for $\GSp_6$} \label{S:GSp(6)}
\label{Sect:spinGSp6etc}
In \S~\ref{S:PGSp(6)}, we have shown the Spin lifting from $\PGSp_6$ to $\GL_8$ via:
\vskip 5pt
\begin{itemize}
\item[(a)] similitude theta lifting to $\PGSO_8$, 
\item[(b)] followed by an application of the triality automorphism, before pulling back to (the standard) $\SO_8$ (or in one step, by pulling back via the nonstandard $f_2$), and
\item[(c)]  transferring from $\SO_8$ to $\GL_8$ via \cite{CKPSS} or \cite{A},
\end{itemize}
at least for cuspidal representations of $\PGSp_6$ with generic A-parameters when restricted to $\Sp_6$.
In this section, we first explain how to remove the trivial central character hypothesis and extend 
Theorem~\ref{thm:weakspinlift} to construct a spin lifting for cuspidal representations of $\GSp_6$ (with generic A-parameters when restricted to $\Sp_6$). This will be done in \S~\ref{sec:spinliftingwithcenter}, using a similar strategy as above. The first step is the same as in (a) above: using the global similitude theta lifting, we may lift a (globally generic) cuspidal representation of $\GSp_6$ to an automorphic representation of $\GSO_8$. For step (b), we need to replace $f_2$ with the natural isogeny 
\begin{equation} \label{E:tildef}
   \tilde{\rho}_2: \GSpin_8 \longrightarrow \GSO_8 \end{equation}
introduced in \S \ref{SS:VQ}. For the last step (c), we rather rely on the transfer from $\GSpin_8$ to $\GL_8$ via the results of Asgari-Shahidi \cite{AS}. 

\vskip 5pt

	The established spin lifting from $\GSp_6$ to $\GL_8$ also has interesting consequences to a certain lifting from $\PGL_7$ to $\SL_8$, discussed in the end of \S~\ref{sec:spinliftingwithcenter}. In the next two subsections \S~\ref{sec:RSliftingcenter} and \S~\ref{sec:yetanotherlifting}, we explore two other applications of the ideas above: first to a Rankin-Selberg (tensor-product) lifting from $\GL_2 \times \GSp_4$ to $\GL_8$ generalizing Theorem~\ref{T:intro1} (ii), and second to the study of the behavior of the endoscopic lifting ${\rm GSO}_4 \times {\rm GSO}_4 \longrightarrow {\rm GSO}_8$ {\it after applying triality}. The work of Bin Xu \cite{X1, X2} plays some role here.

\vskip 5pt The isogeny \eqref{E:tildef} naturally induces $f_2 : \SO_8 \rightarrow \PGSO_8$ (see the diagram~\eqref{E:diagr2tf2}), hence is related to triality since $f_2$ is; however, it is not the composition of a triality automorphism and a standard morphism contrary to $f_2$, since neither $\GSO_8$ nor $\GSpin_8$ do have a triality automorphism. We finally explain in \S~\ref{sect:shakespear} a way to 
restore an order $3$ symmetry in this picture by introducing in a certain larger group $\GGGSpin_8$, with derived subgroup ${\rm Spin}_8$ and center $\simeq \mathbb{G}_m^3$, over which a triality $\theta$ exists. We will show 
that $\tilde{\rho}_2$ naturally factors as 
\begin{equation}
\label{eq:factorisationtilderho2}
\GSpin_8 \longrightarrow \GGGSpin_8 \overset{\theta}{\longrightarrow} \GGGSpin_8 \longrightarrow \GSO_8,
\end{equation}
with ``standard'' first and last maps. As a consequence, the pullback of automorphic forms via $\tilde{\rho}_2$ used to perform the Spin lifting for $\GSp_6$ decomposes accordingly as a sequence of three pullbacks. Finally, we explain that $\GGGSpin_8$ actually appears as a Levi subgroup in the exceptional similitude group ${\rm GE}_6$, with $\theta$ induced by an inner automorphism of ${\rm GE}_6$.

	\vskip 10pt

\subsection{\bf The Spin lifting} 
\label{sec:spinliftingwithcenter}
\vskip 5pt
The Langlands dual group of $\GSp_6$ is $\GSpin_7(\C)$ and one has the Spin representation (see \S~\ref{SS:spingp})
\[ {\rm spin}: \GSpin_7(\C) \longrightarrow  \GL_8(\C). \]
 The corresponding weak lifting from $\PGSp_6$ to $\GL_8$ is the {\it spin lifting} for $\GSp_6$.

\begin{thm} \label{T:GSpin lifting}
Let $\pi$ be a cuspidal representation of $\GSp_6$ whose restriction to $\Sp_6$ has a generic A-parameter. Then the weak spin lifting of $\pi$ exists on $\GL_8$.\end{thm}
By the discussion in \S~\ref{SS:VQ}, we have a commutative diagram:
 \begin{equation} \label{E:diagr2tf2}
 \begin{CD}
 \GSpin_8 @>{\tilde{\rho}_2}>>   \GSO_8 \\
@VVV @VVV \\
\SO_8 @>{f_2}>> \PGSO_8   
 \end{CD}
\end{equation}
where the first vertical arrow is the standard morphism $\rho$ and the second is the canonical projection.
The Langlands dual of this diagram, incorporating both the natural embedding $\iota : \GSpin_7(\C) \rightarrow \GSpin_8(\C)$ (defined in \S\ref{SS:spingp}, and extending the $\iota$ in \eqref{D:spin}) and the standard representations, is thus the following commutative diagram:
\begin{equation} \label{E:splice-dual}
\begin{CD}
\GL_8(\C) @<{\rm std}<< \GSO_8(\C) @<{\tilde{\rho}_2^{\vee}}<< \GSpin_8(\C) @<\iota<< \GSpin_7(\C)\\
@| @AAA @AAA @AAA \\
\GL_8(\C) @<{\rm std}<< \SO_8(\C) @<f_2^{\vee}<<  \Spin_8(\C) @<\iota<< \Spin_7(\C), 
\end{CD}
\end{equation}
with canonical inclusions as vertical arrows.

\vskip 5pt

\begin{lemma} 
\label{lem:spin3composite}
The composite map in the first row of the diagram \eqref{E:splice-dual} is the spin representation of $\GSpin_7(\C)$.
\end{lemma}

\begin{proof} Let $h : \GSpin_7(\C) \rightarrow \GL_8(\C)$ be the composite map of the statement.
By commutativity of \eqref{E:splice-dual}, and the discussion in \S\ref{SS:n=3}, we know that $h|_{\Spin_7(\C)}$ is the spin representation of $\Spin_7(\C)$. It only remains to show that $h$, or equivalently ${\tilde{\rho}_2^{\vee}}$, induces the identity map $t \mapsto t$ on the natural central (or cocentral) $\mathbb{G}_m$ of its source and its target. But this follows as $\tilde{\rho}_2 : {\rm GSpin}_8 \rightarrow {\rm GSO}_8$ itself has this property, since it induces a half-spin representation of ${\rm GSpin}_8$.
\end{proof}

\begin{proof} (of Theorem~\ref{T:GSpin lifting}) 
By Lemma~\ref{lem:spin3composite}, the Spin lifting from $\GSp_6$ to $\GL_8$ has been broken down into a composite of three functorial liftings appearing in the first row in \eqref{E:splice-dual}, namely those induced by the three dual maps $\iota$, $\tilde{\rho}_2^\vee$ and ${\rm std}$. Starting from a cuspidal representation $\pi$ of $\GSp_6$
whose restriction to $\Sp_6$ has generic A-parameter and whose Hecke-Satake family $c(\pi)$ is contained in $\GSpin_7(\C)$, these three  weak functorial liftings can be obtained as follows:
\vskip 10pt
\begin{itemize}
\item ({\it lifting $\iota$}) As we saw  in the proof of Theorem \ref{thm:weakspinlift}, after replacing $\pi$ by a globally generic cuspidal representation in the same global A-packet (constructed by B. Xu),  the global similitude theta lifting provides an explicit construction of the weak functorial lifting associated to
$\iota$, giving rise to an automorphic representation $\Theta(\pi)$ on $\GSO_8$ with Hecke-Satake family 
\[    c(\Theta(\pi)) \,\,=\,\, \iota( c(\pi)) \subset \GSpin_8(\C). \]
\vskip 5pt

\item ({\it lifting ${\tilde{\rho}_2^{\vee}}$}) It follows by Proposition \ref{P:isogeny} that the functoriality for ${\tilde{\rho}_2^{\vee}}$ is simply given by the pullback of automorphic forms via the isogeny $\tilde{\rho}_2$ : 
any automorphic constituent $\Pi$ of the restriction of $\Theta(\pi)$ to ${\rm GSpin}_8$ via $\tilde{\rho}_2$
has Hecke-Satake family
 \[  c(\Pi) \,\,= \,\,{\tilde{\rho}_2^{\vee}}\left( c(\Theta(\pi)) \right) \,\,=\,\, {\tilde{\rho}_2^{\vee}}\left( \iota( c(\pi)) \right) \,\, \subset \GSO_8(\C).\]
 \vskip 5pt
 
 \item ({\it lifting {\rm std}}) Finally, the last functoriality from $\GSpin_8$ to $\GL_8$ induced by ${\rm std}$ has been shown by Asgari-Shahidi \cite{AS}, allowing us to produce an automorphic representation $\Pi'$ of $\GL_8$ with Hecke-Satake family
 \[ c(\Pi') \,\,= \,\,{\rm std} (c(\Pi) )\,\,=\,\,{\rm std} \left({\tilde{\rho}_2^{\vee}} \left( \iota( c(\pi))  \right) \right)\,\, \subset \GL_8(\C). \]
 \end{itemize}
This completes the proof of the theorem.\end{proof}

\vskip 5pt
\begin{remark} 
\label{rem:stringliftGSp6}
If one knows that the Asgari-Shahidi lift from $\GSpin$ groups to $\GL$ is strong at unramified places, then one has a similar strengthening of the theorem as in Theorem \ref{T:strong} (by the same proof).
\end{remark}

\vskip 5pt

Let us end this subsection with an application to a lifting from $\PGL_7$ to $\SL_8$.
Observe that the spin representation of $\GSpin_7(\C)$ induces a morphism 
$$ \overline{{\rm spin}} : \SO_7(\C) \longrightarrow {\rm PGL}_8(\C).$$
On the other hand, if $\pi$ is a selfdual cuspidal automorphic representation of $\PGL_7$, 
then $\pi$ is necessarily orthogonal. In particular, if $\pi_v$ is unramified then $c(\pi_v)$ is the image under ${\rm std}: \SO_7(\C) \rightarrow \GL_7(\C)$ of a unique semisimple conjugacy class $c'(\pi_v)$ in $\SO_7(\C)$. Setting $c'(\pi)=\{ c'(\pi_v) : v \notin S\}$, it makes sense to ask for the existence of an automorphic representation $\Pi$ of $\SL_8$ satisfying $c(\Pi)=\overline{{\rm spin}}(c'(\pi))$. We call such a $\Pi$ a (weak) {\it $\overline{{\rm spin}}$ lifting } of $\pi$.

\begin{thm} \label{T:wwpinb}
If $\pi$ is a selfdual cuspidal automorphic representation of $\PGL_7$, then there exists a $\overline{{\rm spin}}$ lifting of $\pi$.
\end{thm}

\begin{proof} By \cite{GRS2}, there exists a globally generic cuspidal automorphic representation $\sigma$ of ${\rm Sp}_6$ such that ${\rm std}(c(\sigma))=c(\pi)$, or equivalently, such that $c(\sigma)=c'(\pi)$. 
Let $\widetilde{\sigma}$ be a (necessarily globally generic) cuspidal automorphic representation of $\GSp_6$ such that $\widetilde{\sigma}|_{{\rm Sp}_6(\mathbb{A})}$ contains $\sigma$; the existence of such a $\widetilde{\sigma}$ follows for instance from \cite{X1} Lemma 5.3.
The image of $c(\widetilde{\sigma})$ under $\rho : \GSpin_7(\C) \rightarrow \SO_7(\C)$ is then $c(\pi')$ by Proposition~\ref{P:isogeny}. 
Let $\Pi_0$ be spin lifting of $\widetilde{\sigma}$ given by Theorem~\ref{T:GSpin lifting}; it satisfies $c(\Pi_0)={\rm spin}(c(\widetilde{\sigma}))$. Any automorphic constituent $\Pi$ of $\Pi_0|_{\SL_8}$ has thus the required property, by Proposition~\ref{P:isogeny} again.
\end{proof}

\vskip 5pt

\subsection{\bf Rankin-Selberg lifting}
\label{sec:RSliftingcenter}
Recall that in Corollary \ref{C:tensor}, we have produced the weak functorial lifting from $\PGSp_4 \times \PGL_2$ to $\GL_8$ relative to the map
\[  \begin{CD} \Sp_4(\C) \times \SL_2(\C) @>\boxtimes>> \SO_8(\C) @>>> \GL_8(\C) \end{CD} \]
of dual groups.  In the same vein, we can extend this result by removing the hypothesis of trivial central characters. 
More precisely, we have:

\vskip 5pt

\begin{thm} \label{T:RS2}
Suppose that $\pi$ is a cuspidal representation of $\GSp_2 = \GL_2$ and $\sigma$ a cuspidal representation of $\GSp_4$, with Hecke-Satake family 
\[ c(\pi) \subset \GSp_2^{\vee} = \GSpin_3(\C)  = \GL_2(\C) \quad \text{and} \quad c(\sigma) \subset \GSpin_5(\C) = \GSp_4(\C). \]
Then there is an automorphic representation $\pi \boxtimes \sigma$ of $\GL_8$ with Hecke-Satake family 
\[    c(\pi \otimes \sigma) = \{ c(\pi_v) \boxtimes c(\sigma_v) \}_{v \notin S}. \]
\end{thm}

To prove this, note that from \S~\ref{SS:spingp} one has a natural map
\[  \iota : \GSpin_3(\C) \times \GSpin_5(\C) \longrightarrow \GSpin_8(\C).  \]
If we replace the map $\GSpin_7(\C) \longrightarrow \GSpin_8(\C)$ in (\ref{E:splice-dual}) by this map, the composite
\[   \begin{CD}
\GSpin_3(\C) \times \GSpin_5(\C) @>\iota>> \GSpin_8(\C) @>{\tilde{\rho}_2^{\vee}}>> \GSO_8(\C) @>{\rm std}>> \GL_8(\C) \end{CD} \]  
induces the Rankin-Selberg lifting $\boxtimes$.  
\vskip 5pt

To establish this Rankin-Selberg lifting, it thus suffices to establish the weak functorial lifting for the three maps above. The second and third steps have been described in the previous subsection. The first (induced by $\iota$) is an endoscopic lifting for similitude groups, from $\GSp_2 \times \GSp_4$ to $\GSO_8$.
This has been proven by Bin Xu \cite{X1, X2}.  With this, the proof of the theorem is complete. 
\vskip 10pt

\subsection{\bf  Another functorial lifting} 
\label{sec:yetanotherlifting}
  The Rankin-Selberg lifting proved in Theorem \ref{T:RS2}, based on the endoscopic transfer from $\GSp_4 \times \GL_2$ followed by an application of triality, is a physical manifestation of the equality
  \[ 3 + 5 = 4 \times 2. \]
 In this subsection, we consider another instance of this construction starting from the endoscopic transfer induced by the natural morphism
 \[   \GSpin_4(\C) \times \GSpin_4(\C) \longrightarrow \GSpin_8(\C). \] 
 
 The associated endoscopic transfer  from $\GSO_4 \times \GSO_4$ to $\GSO_8$ has largely been shown in \cite[Thm. 1.2]{X2},  under a hypothesis \cite[Defn. 4.4]{X2} on cuspidal automorphic representations. As we shall explain later on, the unavailability of this endoscopic lifting in full generality will not unduly bother us below.  Hence, for the subsequent discussion, the reader may  assume for simplicity that this endoscopic lifting is available.  The question we shall consider is what happens when one applies the triality automorphism to the resulting lift. 
 \vskip 5pt

   Let us begin by setting up some notations, for bookkeeping purposes. 
   One has the following concrete realization of the groups $\GSO_4$ and $\GSpin_4$:
   \[  \GSO_4 = (\GL_2 \times \GL_2) / \mathbb{G}_m^{\nabla} \quad \text{and} \quad \GSpin_4 = \GL_2 \times_{\det}  \GL_2 \]
   where $\mathbb{G}_m^{\nabla} = \{ (t, t^{-1}): t \in \mathbb{G}_m \}$ and the suberscript $_{\det}$ refers to the subgroup of those elements $(g_1, g_2)$ satisfying 
  $\det(g_1) = \det(g_2)$.  Given the large number of $\GL_2$'s here, we will annotate the various $\GL_2$'s with card suits, so as to help the reader and ourselves to distinguish between them. In particular, we set:
  \[  \GSO_4^{\heartsuit \diamondsuit} = (\GL_2^{\heartsuit} \times \GL_2^{\diamondsuit}) / \mathbb{G}_m^{\nabla} \quad \text{and}  \quad 
\GSpin_4^{\heartsuit\diamondsuit} = \GL_2^{\heartsuit} \times_{\det} \GL_2^{\diamondsuit}. \]
  Hence we have
  \[  (\GSO_4^{\heartsuit \diamondsuit})^{\vee} = \GSpin_4^{\heartsuit\diamondsuit}(\C) \quad \text{and} \quad (\GSpin_4^{\heartsuit\diamondsuit})^{\vee}  = \GSO_4^{\heartsuit \diamondsuit}(\C). \]
  \vskip 5pt
  
  Given the above concrete realizations of $\GSO_4$ and $\GSpin_4$, we can describe their (L-packets of) cuspidal  representations  as follows. Given cuspidal representations $\pi_{\heartsuit}$ of $\GL_2^{\heartsuit}$ and $\pi_{\diamondsuit}$  of $\GL_2^{\diamondsuit}$, the restriction of $\pi_{\heartsuit} \otimes \pi_{\diamondsuit}$ to $\GSpin_4^{\heartsuit\diamondsuit}$ gives an L-packet $[ \pi_{\heartsuit} \otimes \pi_{\diamondsuit}]$. Note that 
  \[ [\pi_{\heartsuit} \otimes \pi_{\diamondsuit}] = [ \pi_{\heartsuit} \cdot \chi  \otimes \pi_{\diamondsuit} \cdot \chi^{-1} ] \quad \text{for any Hecke character $\chi$}. \]
  On the other hand, assume now that the central characters of $\pi_{\heartsuit}$ and $\pi_{\diamondsuit}$ are equal. Then $\pi_{\heartsuit} \otimes \pi_{\diamondsuit}$ defines a cuspidal representation of $\GSO_4^{\heartsuit\diamondsuit}$.
  \vskip 5pt
  
  Again, using the above concrete realizations of  $\GSpin_4$ and $\GSO_4$, the reader can convince herself that there is no isogeny  $\GSpin_4 \longrightarrow \GSO_4$. However, we note the following curious lemma:
  \vskip 5pt
  
  \begin{lemma} \label{L:f-isogeny}
  The map  
    \[   \GL_2^{\heartsuit} \times \GL_2^{\diamondsuit} \times \GL_2^{\spadesuit} \times \GL_2^{\clubsuit} \longrightarrow  \GL_2^{\heartsuit} \times \GL_2^{\spadesuit} \times \GL_2^{\diamondsuit} \times \GL_2^{\clubsuit}  \]
  given by exchanging the second and third entries gives rise to  an isogeny
  \[  f: (\GSpin^{\heartsuit \diamondsuit}_4 \times  \GSpin^{\spadesuit \clubsuit}_4)/ \mathbb{G}_m^{\nabla}   \longrightarrow \GSO^{\heartsuit\spadesuit}_4 \times_{\sim} \GSO^{\diamondsuit \clubsuit}_4  \]
  The kernel of this isogeny is the subgroup
  \[  \mu_2^{\heartsuit \spadesuit} =  \{[(\epsilon, 1) , (\epsilon, 1)]: \epsilon \in \mu_2 \} \subset (\mu_2^{\heartsuit} \times \mu_2^{\diamondsuit} \times \mu_2^{\spadesuit} \times \mu_2^{\clubsuit})/  \mathbb{G}_m^{\nabla}. \]
   \end{lemma}
 
 The map $f$ fits into the following diagram of morphisms of Langlands dual groups:

 \begin{equation} \begin{CD} \label{E:final-dual}
 \GSO^{\heartsuit\spadesuit}_4(\C)  \times_{\sim} \GSO^{\diamondsuit\clubsuit}_4(\C)  @>\iota>> \GSO_8(\C) @>{\rm std}>> \GL_8(\C) \\
 @AfAA   @AA{\tilde{\rho}_2^{\vee}}A \\
 (\GSpin^{\heartsuit\diamondsuit}_4(\C) \times \GSpin^{\spadesuit\clubsuit}_4(\C))/\mathbb{G}_m^{\nabla} @>\tilde{\iota}>> \GSpin_8(\C) \\
 @VpVV  @VVpV \\
 \SO_4(\C) \times \SO_4(\C) @>\iota_{\flat}>> \SO_8(\C) @>{\rm std}>> \GL_8(\C) 
 \end{CD} \end{equation}

 where   the horizontal arrows are embeddings and the right upward arrow is the same map as in (\ref{E:splice-dual}).   While the embeddings $\iota$ and $\iota_{\flat}$ are the natural ones, there are in fact three choices for the embedding $\tilde{\iota}$, determined by the image of $\mu_2^{\heartsuit\spadesuit} = {\rm Ker}(f)$ in the center of $\Spin_8(\C)$. 
 For the above diagram to be commutative, we have to use the one such that 
 \[ \tilde{\iota}(\mu_2^{\heartsuit\spadesuit}) = {\rm Ker}( \tilde{\rho}_2^{\vee}).  \]
 \vskip 5pt
 
 This diagram should induce a corresponding diagram of weak Langlands functorial lifting:
 \begin{equation} \begin{CD} \label{final}  
  [\mathcal{A}( (\GSpin^{\heartsuit\spadesuit}_4 \times \GSpin^{\diamondsuit\clubsuit}_4)/ \mathbb{G}_m^{\nabla})]
 @>\iota_*>> [\mathcal{A}(\GSpin_8)] @>{\rm std}_*>> [\mathcal{A}(\GL_8)] \\ 
 @Af_*AA   @AA{({\tilde{\rho}}_2^{\vee})_*}A \\
[ \mathcal{A}(\GSO^{\heartsuit\diamondsuit} _4 \times_{\sim}\GSO^{\spadesuit\clubsuit}_4) ] @>\tilde{\iota}_*>> [\mathcal{A}(\GSO_8)] \\
 @Vp_*VV  @VVp_*V \\
[ \mathcal{A}(\SO_4 \times \SO_4)] @>(\iota_{\flat})_*>>  [\mathcal{A}(\SO_8)] @>{\rm std}_*>>  [\mathcal{A}(\GL_8)]
 \end{CD} \end{equation}
 where $[\mathcal{A}(G)]$ denotes the set of near equivalence classes of irreducible automorphic representations of $G$.
 Our goal in this subsection is to understand the composite lifitng  ${\rm std}_* \circ (\tilde{\rho}_2^{\vee})_* \circ \tilde{\iota}_*$.
 As mentioned above, the endoscopic lifting $\tilde{\iota}_*$ is not known in full generality (see \cite[Thm. 1.2]{X2}), but our main concern  is with the composite lifting above. For this,   the above commutative diagram allows us to construct $ {\rm std}_* \circ \iota_* \circ  f_*$ instead, thus bypassing the lack of $\tilde{\iota}_*$ in full generality.
 \vskip 5pt
 
 We now compute $f_*$ and ${\rm std}_* \circ \iota_*$:
  \begin{itemize}
 \item   Since $f$ is an isogeny, $f_*$ is given by pullback of automorphic forms 
by Proposition \ref{P:isogeny}.  More precisely, suppose that 
 \[ \sigma:=  [(\pi_{\heartsuit} \otimes \pi_{\diamondsuit}) \otimes (\pi_{\spadesuit} \otimes \pi_{\clubsuit})] \in [\mathcal{A}(\GSO_4^{\heartsuit\diamondsuit} 
 \times_{\sim} \GSO_4^{\spadesuit\clubsuit})], \]
 so that
 \[  \omega_{\heartsuit} = \omega_{\diamondsuit} \quad \text{and} \quad \omega_{\spadesuit} = \omega_{\clubsuit} \]
 where $\omega_?$ denotes the central character of $\pi_?$. Then
 Lemma \ref{L:f-isogeny} shows that 
\[  f_*(\sigma) = [\pi_{\heartsuit} \otimes \pi_{\spadesuit})]  \otimes   [\pi_{\diamondsuit} \otimes \pi_{\clubsuit}] \in  [\mathcal{A}( (\GSpin^{\heartsuit\spadesuit}_4 \times \GSpin^{\diamondsuit\clubsuit}_4)/ \mathbb{G}_m^{\nabla})]. \]
\vskip 5pt

\item To determine ${\rm std}_* \circ \iota_*$, we note that there is a commutative diagram of morphisms of dual groups:
\[ \begin{CD}
 \GSO^{\heartsuit\spadesuit}_4(\C)  \times_{\sim} \GSO^{\diamondsuit\clubsuit}_4(\C)  @>\iota>> \GSO_8(\C) \\
 @V{\rm std} \otimes {\rm std} VV @VV{\rm std}V \\
\GL_4(\C) \times \GL_4(\C) @>>> \GL_8(\C)  
\end{CD} \]
which should give rise to the following diagram of liftings:
 \[ \begin{CD} 
   [\mathcal{A}( (\GSpin^{\heartsuit\spadesuit}_4 \times \GSpin^{\diamondsuit\clubsuit}_4)/ \mathbb{G}_m^{\nabla})] @>\iota_*>> [\mathcal{A}(\GSpin_8)] \\
   @V{\rm std}_* \times {\rm std}_*VV   @VV{\rm std}_*V \\
  [ \mathcal{A}(\GL_4)] \times[ \mathcal{A}(\GL_4)] @>\boxplus>>[ \mathcal{A}(\GL_8)] 
   \end{CD} \] 
   so that
   \[  {\rm std}_* \circ  \iota_* =\boxplus  \circ ({\rm std}_* \times {\rm std}_*). \]
  On the RHS of this identity,  $\boxplus$  is the formation of isobaric sum (i.e. parabolic induction)  and 
  \[  {\rm std}_*: [ \mathcal{A}( (\GSpin^{\heartsuit\spadesuit}_4) ] = [\mathcal{A}(\GL_2^{\heartsuit} \times_{\det} \GL_2^{\spadesuit})] \longrightarrow [\mathcal{A}(\GL_4)] \]
  is the Rankin-Selberg lifting $\boxtimes$ from $\GL_2 \times \GL_2$ to $\GL_4$ constructed by Ramakrishnan \cite{Ra}. Hence, we have explained the construction of ${\rm std}_* \circ \iota_*$. 
  \end{itemize}
 
 \vskip 10pt
 
 By the above discussion, we thus have
 \[   {\rm std}_* \circ (\tilde{\rho}_2^{\vee})_* \circ \tilde{\iota}_* =  \boxplus \circ ({\rm std}_* \times {\rm std}_*) \circ f_* \]
 and all three functorial liftings on the RHS have been constructed. Thus, we have shown:
 \vskip 5pt
 
 \begin{prop}
 Let
 \[ \sigma:=  [(\pi_{\heartsuit} \otimes \pi_{\diamondsuit}) \boxtimes (\pi_{\spadesuit} \otimes \pi_{\clubsuit})] \in [\mathcal{A}(\GSO_4^{\heartsuit\diamondsuit} \times \GSO_4^{\spadesuit\clubsuit})]. \]
 so that 
 \[   ({\rm std}_* \circ p_* \circ  \tilde{\iota}_*) (\sigma) = [(\pi_{\heartsuit} \boxtimes \pi_{\diamondsuit})) \boxplus  (\pi_{\spadesuit} \boxtimes \pi_{\clubsuit})] \in [\mathcal{A}(\GL_8)]. \]
 Then
 \[    ({\rm std}_* \circ (\tilde{\rho}_2^{\vee})_* \circ  \tilde{\iota}_*) (\sigma) = 
[( \pi_{\heartsuit} \boxtimes \pi_{\spadesuit}) \boxplus  (\pi_{\diamondsuit} \boxtimes \pi_{\clubsuit})] \in [\mathcal{A}(\GL_8)] . \]
  \end{prop}
\noindent     Hence, the application of triality    does not produce a fundamentally new case of functorial lifting here: its effect is to  produce a ``remixing" as depicted in the following equation:
  \[   2_{\heartsuit} \cdot 2_{\diamondsuit}  + 2_{\spadesuit} \cdot 2_{\clubsuit} = 2_{\heartsuit} \cdot 2_{\spadesuit}  + 2_{\diamondsuit} \cdot 2_{\clubsuit}. \]

 \vskip 5pt
 
 We end this subsection with a remark. In establishing the above proposition, we had only needed to establish the weak functorial lifting ${\rm std}_* \circ \iota_*$. In fact, the weak functorial lifting $\iota_*$, which is an endoscopic lifting in the setting of $\GSpin$-groups can also be established using the automorphic descent results of Hundley-Sayag \cite{HS}. 
 \vskip 5pt
 
 More precisely, 
 let 
   \[ \sigma: =  [ \pi_{\heartsuit} \boxtimes \pi_{\spadesuit}] \boxtimes  [ \pi_{\diamondsuit} \boxtimes \pi_{\clubsuit}] \]
   be an L-packet of cuspidal representations of $(\GSpin^{\heartsuit\spadesuit}_4 \times \GSpin^{\diamondsuit\clubsuit}_4)/ \mathbb{G}_m^{\nabla}$, so that we have the identity of central characters
   \[  \omega_{\heartsuit} \cdot \omega_{\spadesuit} = \omega_{\diamondsuit} \cdot \omega_{\clubsuit}  =: \mu. \]
    Then  as we saw above,
   \[    (\boxplus \circ ({\rm std}_*  \times {\rm std}_*))(\sigma)  = ( \pi_{\heartsuit} \boxtimes \pi_{\spadesuit}) \boxplus (\pi_{\diamondsuit} \boxtimes \pi_{\clubsuit}) \quad \text{on $\GL_8$.} \]
   The two  summands in the above equation satisfies
  \[    ( \pi_{\heartsuit} \boxtimes \pi_{\spadesuit})^{\vee}  =  ( \pi_{\heartsuit} \boxtimes \pi_{\spadesuit}) \cdot \mu^{-1} \quad 
  (\pi_{\diamondsuit} \boxtimes \pi_{\clubsuit})^{\vee} = (\pi_{\diamondsuit} \boxtimes \pi_{\clubsuit})\cdot \mu^{-1}, \]
  so that the (partial) twisted Rankin-Selberg L-functions
  \[  L^S(s, (\pi_{\heartsuit} \boxtimes \pi_{\spadesuit}) \times (\pi_{\heartsuit} \boxtimes \pi_{\spadesuit}) \cdot \mu^{-1} ) \quad \text{and} \quad
   L^S(s, (\pi_{\diamondsuit} \boxtimes \pi_{\clubsuit}) \times (\pi_{\diamondsuit} \boxtimes \pi_{\clubsuit}) \cdot \mu^{-1} ) \]
   have simple poles at $s=1$. On the other hand, the twisted exterior square L-function
  \[
  L^S(s,   \pi_{\heartsuit} \boxtimes \pi_{\spadesuit} , \wedge^2 \times  \mu^{-1} )
    = L^S(s, \pi_{\heartsuit}, {\rm Ad}) \cdot L^S(s, \pi_{\spadesuit}, {\rm Ad}) \]
    is holomorphic at $s=1$, and likewise for the twisted exterior square L-function of $\pi_{\diamondsuit} \boxtimes \pi_{\clubsuit}$. 
     Hence, the twisted symmetric square L-functions
     \[    L^S(s,   \pi_{\heartsuit} \boxtimes \pi_{\spadesuit} , {\rm Sym}^2 \times  \mu^{-1} ) \quad \text{and} \quad 
     L^S(s,   \pi_{\diamondsuit} \boxtimes \pi_{\clubsuit} , {\rm Sym}^2 \times  \mu^{-1} ) \]
    have poles at $s=1$.  By \cite{HS}, one concludes that $\boxplus \circ  ({\rm std}_* \times {\rm std}_*)(\sigma)$ can be descended back to $\GSpin_8$.
    In other words, there is a globally generic automorphic representation $\iota_*(\sigma)$ on $\GSpin_8$ such that
    \[  {\rm std}_* ( \iota_*(\sigma)) = ( \pi_{\heartsuit} \boxtimes \pi_{\spadesuit}) \boxplus (\pi_{\diamondsuit} \boxtimes \pi_{\clubsuit}) \quad \text{on $\GL_8$.} \]
   
    \vskip 10pt

\subsection{\bf  Where art thou,  triality?}
\label{sect:shakespear}
  Neither ${\rm GSO}_8$ nor ${\rm GSpin}_8$ has an automorphism inducing 
  triality on $\PGSO_8$ or $\Spin_8$, because of the presence of a unique central $\mathbb{G}_m$ subgroup. 
  As promised in the end of the introduction of Section~\ref{Sect:spinGSp6etc}, we now first explain how to restore an order $3$ symmetry on a variant of $\GSpin_8$. 
  %We also briefly explain in the end how this group and its triality occur in practice inside the exceptional similitude group ${\rm GE}_6$. 
  In this section, our groups are defined and split over an arbitrary field $F$, say with ${\rm char}\, F \neq 2$.

\vskip 5pt

Let $E$ be the set of order two elements in the center $Z$ of ${\rm Spin}_8$. We have $|E|=3$. 
For bookkeeping reasons we introduce, for each $e \in E$, a copy of $\mathbb{G}_m$, ${\rm SO}_8$ and ${\rm GSO}_8$, that we denote respectively by $\mathbb{G}_m^e$, ${\rm SO}_8^e$ and ${\rm GSO}_8^e$; we 
also use the notation $\mathbb{G}_m^E$ for $\prod_{e \in E} \mathbb{G}_m^e$.  We have an embedding $$\iota : Z \rightarrow \mathbb{G}_m^E$$ sending any $e \in E \subset Z$
to the element of $\mathbb{G}_m^E$ with $e$-component $1$, and two other components $-1$. 
Using this embedding, we set:  
\begin{equation}
\label{def:GGGSpin8}
\GGGSpin_8 = (\mathbb{G}_m^E \times \Spin_8)/ (\iota \times {\rm id})(Z).
\end{equation}
(pronounced {\it tri-spin}). Let us fix a triality automorphism $\theta$ of ${\rm Spin}_8$ as in \S\ref{SS:VQ}. 
It naturally induces a $3$-cycle on $E$. 
The automorphism $\theta$ thus trivially extends to $\GGGSpin_8$
by permuting the factors of $\mathbb{G}_m^E$ according to the same cycle, and we still denote by $\theta$ this extension. 

\vskip 5pt

 Let $e \in E$. It will be convenient to set $e'=\theta(e)$ and $e''=\theta^2(e)=\theta^{-1}(e)$.
We have then $E = \{e, e', e''\}$ and we may write any element $t \in \mathbb{G}_m^E$ as $(t_e,t_{e'},t_{e''})$. We finally set
$$\GSpin_8^e = \left( \mathbb{G}^e_m \times {\rm Spin}_8 \right) / \langle (-1,e) \rangle.$$
There is a unique pair of morphisms $j_e$ and $\rho_e$ as in the sequence below
\[ 
\begin{CD}
\GSpin_8^e  @>{j_e}>>  \GGGSpin_8 @>{\rho_e}>> \GSO_8^e
\end{CD}
\]
and satisfying the following properties: \begin{itemize}
\item The morphism $j_e$ induces the identity on the natural $\Spin_8$ subgroups on both sides.

\vskip 5pt

\item For $t \in \mathbb{G}_m^e$, the element $j_e(t) \in \mathbb{G}_m^E$ has $e$-component $1$, and other two components $t$. 
\end{itemize}
Such a morphism $j_e$ exists as we have $(1,-1,-1) \equiv e' e'' = e$ in $\GGGSpin_8$.
\begin{itemize}
\vskip 5pt

\item Over the natural subgroup $\Spin_8$ inside $\GGGSpin_8$, the morphism $\rho_e$ coincides with the morphism $\rho_i : \Spin_8 \rightarrow \SO_8$ defined in \S\ref{SS:VQ} for the unique $i$ such that $\rho_i(e)=1$. 

\vskip 5pt

\item For $t \in \mathbb{G}_m^e$, the element $\rho_e(t) \in {\rm GSO}_8^e$ is the multiplication by the scalar $t_{e'}/t_{e''}$. 
\end{itemize}
Such a morphism $\rho_e$ exists, as for all $f \in E$ and $t=\iota(f)$, 
the scalar $t_{e'}/t_{e''}\rho_e(f)$ is always $1$ (we have $\rho_i(e)=1$ and $\rho_i(e')=\rho_i(e'')=-1$). 
Summarizing, we have for any $e \in E$ the following diagram:

 \[
\xymatrix@R=10pt{
 \GSpin^{e}_8   \ar[rd]^{j_{e}}   &    &      \GSO^{e}_8 \\
 \GSpin^{e'}_8 \ar[r]    &   \GGGSpin_8 \ar[ru]^{\rho_{e}}  \ar[r]  \ar[rd]_{\rho_{e''}} &    \GSO^{e'}_8   \\
 \GSpin^{e''}_8 \ar[ru]_{j_{e''}}     &       &      \GSO^{e''}_8 \\ 
 }
\]
The triality $\theta$ of $\Spin_8$ induces for each $e$ an isomorphism $\GSpin_8^e \overset{\ssim}{\longrightarrow} \GSpin_8^{e'}$ that we may also harmlessly denote $\theta$.
We then have by construction the identities 
\begin{equation}
\label{eq:thetajandrho}
\rho_{e} = \rho_{e'} \circ \theta \, \, \, {\rm and}\, \, \, j_{e} = j_{e'} \circ \theta.
\end{equation}
Observe moreover that the map
 \[ \rho_e \circ j_e : \GSpin_8^e \longrightarrow \GSO_8^e \]
is trivial on $\mathbb{G}_m^e$ and factors through the natural inclusion $\SO_8^e  \subset \GSO_8^e$ (in particular, it is not an isogeny), whereas the morphism
 \[    \rho_{e'} \circ j_{e} : \GSpin_8^e \longrightarrow \GSO_8^{e'} \]
coincides with $\rho_{e'}$ on ${\rm Spin}_8$ and maps $t \in \mathbb{G}_m$ to the multiplication by $t$ in ${\rm GSO}_8^{e'}$. It follows that $\rho_{e'} \circ j_{e}$ is an instance of the map $\tilde{\rho}_2$ in \eqref{E:tildef}. 
More precisely, if we label $E$ by $\{1,2,3\}$ as in \S~\ref{SS:VQ}, with $e$ labelled by $1$ and $\theta$ inducing the cycle $(1\,\,2\,\,3)$ on $E$, then we have $$\tilde{\rho}_2=\rho_{e'} \circ j_{e}.$$ The promised factorisation~\eqref{eq:factorisationtilderho2}
follows then from the identity 
$$\rho_{e'}  \circ j_e = \rho_e \circ \theta^{-1} \circ j_e,$$
which in turn follows from the first equation in \eqref{eq:thetajandrho}. \vskip 5pt

    \vskip 5pt
  
 Without going in too much details, we finally discuss how the group $\GGGSpin_8$ and its automorphism $\theta$ occur when studying the exceptional group $$H={\rm GE}_6 = (\mathbb{G}_m \times {\rm E}_6^{sc})/\mu_3^\Delta,$$
 whose derived group is simply-connected of type $E_6$. Looking at the Dynkin diagram of $E_6$, one sees that there is a (non-maximal) parabolic subgroup $P =MN$ of $H$ whose Levi subgroup $M$ is of semisimple type $D_4$. Indeed, the derived subgroup of $M$ is isomorphic to $\Spin_8$, and it can be shown that we have 
 $$M \simeq \GGGSpin_8.$$
 The associated Weyl group $N_H(M)/M$ is isomorphic to $S_3$. Hence, there is an order $3$ element $h$ of $H$ normalizing $M$ and we can show that we may choose this element and the isomorphism above so that $h$ induces the automorphism $\theta$ of $\GGGSpin_8$. Let us also mention that the adjoint action of $M$ on ${\rm Lie}(N)$ decomposes into the direct sum of the three $8$-dimensional irreducible representations $\rho_e \oplus \rho_{e'} \oplus \rho_{e''}$ of $M$.
 
\vskip 5pt

\section{\bf An Arithmetic Application: Siegel modular forms for ${\rm Sp}_6(\Z)$ }
\label{S:Siegel}
In this section, we specialize the results of Section~\ref{subsect:spinliftingnequals3} to the case of automorphic representations of ${\rm PGSp}_6$ over $\Q$ 
generated by holomorphic Siegel modular forms for the full Siegel modular group ${\rm Sp}_6(\Z)$. In this setting, 
we shall show that Theorems~\ref{thm:weakspinlift} and~\ref{T:strong} continue to hold 
without the genericity  (of A-parameters) assumption there, so that the Spin lifting always exists on $\GL_8$.
In addition, we shall determine precisely the shape of the A-parameter of the Spin lifting on $\GL_8$; this amounts to showing a {\it cuspidality criterion} for the Spin lifting. 
These improvements will be possible by the use of Galois representations arguments, 
and in particular, of the Minkowski theorem 
({\it a nontrivial number field always has a ramified prime}). 
We conclude with an application to spinor $L$-functions.\ps
\vskip 5pt

\subsection{\bf Cuspidal representations of Siegel type}
In all of this section, we assume that $\pi$ is a cuspidal automorphic representation of ${\rm PGSp}_6$ over $\Q$ 
generated by a holomorphic Siegel modular eigenform\footnote{We mean here that $f$ is an eigenform for the full Hecke algebra of ${\rm PGSp}_{6}$, not only of ${\rm Sp}_6$.} $f$ for ${\rm Sp}_6(\Z)$. 
Such a $\pi$ will be called {\it of Siegel type}. 
In other words, we have $\pi_p$ unramified for each prime $p$ 
and $\pi_\infty$ is a holomorphic discrete series. 
In the classical language (see e.g. \cite{vdg}), this means that 
the weights $k_1 \geq k_2 \geq k_3$ of $f$, which is possibly vector-valued, 
satisfy $k_3 \geq 4$.
The relation between those $k_i$ and the infinitesimal character 
${\rm c}(\pi_\infty) \subset \mathfrak{spin}_7(\C)$ of $\pi_\infty$ is as follows: 
the $7$ eigenvalues of ${\rm std}({\rm c}(\pi_\infty))$ are 
$0,\pm a, \pm b, \pm c$ with 
\begin{equation}
\label{eq:abc}
(a,b,c) \,=\,(k_1-1,k_2-2,k_3-3),
\end{equation}
and the $8$ eigenvalues of ${\rm spin}({\rm c}(\pi_\infty))$ are 
$\pm w_1,\pm w_2,\pm w_3,\pm w_4$ with
\begin{equation}
\label{eq:weights}
(w_1,w_2,w_3,w_4)=(\,(a+b+c)/2,\,(a+b-c)/2, \,(a-b+c)/2, \,|a-b-c|/2).
\end{equation}
As $\sum_i k_i \equiv 0 \bmod 2$, 
the $w_i$ are in $\Z$. Moreover, one has $w_1>w_2>w_3>w_4 \geq 0$. \par 
\medskip

Let $\pi$ be of Siegel type.
We denote by $\Psi(\pi,{\rm std})$ the standard Arthur parameter of $\pi|_{{\rm Sp}_6}$. 
A detailed examination of all the possibilities for $\Psi(\pi,{\rm std})$ 
has been carried out in \cite[\S 9.3]{CR} (assuming the main result in \cite{AMR}): 
see also \cite[\S 4.2.2]{taibisiegel} and \cite[\S 8.5.1]{CL}. 
In this situation, in which standard Galois representations are available 
by the works of many authors (Arthur, Chenevier, Clozel, Labesse, Harris, Shin, Taylor..., see the discussion in~\cite[\S 8.2.16]{CL}), 
it is known that $\Psi(\pi,{\rm std})$ is generic if, and only if, the representation $\pi$ is tempered (Clozel, Shin). 
Of course, note that $\pi$ itself is never generic, since $\pi_\infty$ is not.  \ps
\vskip 5pt

\subsection{\bf Generic case}
Let us first assume $\Psi(\pi,{\rm std})$ is generic and
 denote by $\Psi(\pi,{\rm spin})$ the spin lift of $\pi$ to ${\rm GL}_8$ 
furnished by Theorem~\ref{thm:weakspinlift}. 
\vskip 5pt

By \cite[\S 9.3]{CR}, there are actually only two possibilities for $\Psi(\pi,{\rm std})$: 
either it is cuspidal (the most important case), or we have ({\it endoscopic tempered case})
\begin{equation}
\label{eq:tempendosiegelsp6} \Psi(\pi,{\rm std}) \,=\, (\pi_1 \boxtimes \pi_2) \boxplus\, {\rm Sym}^2 \pi_3,
\end{equation}
where $\pi_1,\pi_2,\pi_3$ are cuspidal automorphic representations of ${\rm PGL}_2$ 
generated by holomorphic cuspidal eigenforms for ${\rm SL}_2(\Z)$ 
(see {\it loc.\ cit.} for the precise constraints on their weights in terms of the weights of $f$), 
with $\pi_1 \not \simeq \pi_2$. Here $\pi_1 \boxtimes \pi_2$ and ${\rm Sym}^2 \pi_3$ denote of course 
the automorphic tensor product and symmetric square, 
constructed respectively by Ramakrishnan \cite{Ra} and Gelbart-Jacquet \cite{GJ}, 
and they are cuspidal. 
\vskip 5pt

In both these cases, the automorphic representation $\Theta(\pi')$ produced by global similitude theta lifting in the proof of Theorem \ref{thm:weakspinlift} is a tempered cuspidal representation of $\PGSO_8$ and hence so is its pullback $f_2^*(\Theta(\pi'))$ on $\SO_8$.
 By Theorem \ref{T:strong}(iii), we may view both 
$\Psi(\pi,{\rm std})$ and $\Psi(\pi,{\rm spin})$ 
as a formal direct sum of self-dual orthogonal cuspidal automorphic representations $\pi_i$ of some ${\rm GL}_{n_i}$ over $\Q$, 
with $\sum_i n_i=7$ or $8$ accordingly. 
The $\pi_i$ have level $1$ ({\it i.e.} are unramified at all finite places), by Thm.~\ref{T:strong} in the spin case, 
hence have trivial central characters (they are selfdual). 
Also, in both cases the $\pi_i$ are algebraic, by Thm.~\ref{T:strong} and ~\eqref{eq:weights} in the spin case 
(see \cite[Prop. 8.2.13]{CL}).
\par \smallskip

\vskip 5pt

\subsection{\bf Shape of Spin lifting: tempered endoscopic case}
We would like to determine precisely what the isobaric sum $\Psi(\pi, {\rm spin}) = \boxplus_i \pi_i$ looks like. In this subsection, we handle the case when $\pi$ satisfies~\eqref{eq:tempendosiegelsp6}; the case when $\Psi(\pi, {\rm std})$ is cuspidal is handled in \S \ref{SS:cuspidal}.

 \begin{prop} 
 \label{prop:tempendosiegelsp6spin}If $\Psi(\pi,{\rm std})$ satisfies~\eqref{eq:tempendosiegelsp6}, then we have 
\begin{equation}
\label{eq:spinparendosiegeltempered}
\Psi(\pi,{\rm spin}) = (\pi_1 \boxplus \pi_2) \boxtimes \pi_3 = (\pi_1 \boxtimes \pi_3) \boxplus (\pi_2 \boxtimes \pi_3).
\end{equation}
\end{prop}

This is a very natural guess, since the spin representation of ${\rm Spin}_7$, when restricted to a natural 
\begin{equation}
\label{eq:morphnu}
\nu : {\rm Spin}_4 \times {\rm Spin}_3 \rightarrow {\rm Spin}_7,
\end{equation}
 is isomorphic to the tensor product of the direct sum of the two spin representations of ${\rm Spin}_4 \simeq {\rm SL}_2 \times {\rm SL}_2$, with the spin representation of ${\rm Spin}_3 \simeq {\rm SL}_2$ (\S \ref{SS:spingp}). 
But this only implies that, for each prime $p$, both Satake parameters at $p$ on the left and right sides of \eqref{eq:spinparendosiegeltempered} agree {\it up to a sign}, and our main claim is that this sign is $+1$. 
\par \medskip

For the proof of this proposition, and of others below, we will use certain Galois representations that we first briefly review.
Fix a prime $\ell$ and, for convenience, an isomorphism $\iota : \C \overset{\ssim}{\rightarrow} \overline{\Q}_\ell$. 
For any $\pi$ of Siegel type, recall that by the aforementioned collective works, we have a continuous semisimple Galois representation 
$${\rm r}_{\pi,{\rm std},\iota} : {\rm Gal}(\overline{\Q}/\Q) \rightarrow {\rm SO}_7(\overline{\Q}_\ell),$$
which is unramified outside $\ell$ and such that for all $p \neq \ell$ the semi-simplified conjugacy class
of ${\rm r}_{\pi,{\rm std},\iota}({\rm Frob}_p)^{\rm ss}$ coincides with $\iota({\rm c}(\pi'_p))$, where $\pi'$ is any level one automorphic constituent of $\pi|_{{\rm Sp}_6}$. Recall that we have ${\rm c}(\pi'_p)=\rho({\rm c}(\pi_p))$ where $\rho : {\rm Spin}_7 \rightarrow {\rm SO}_7$ is the standard morphism. The representation ${\rm r}_{\pi,{\rm std},\iota}$ is unique up to conjugacy and known to be crystalline at $\ell$ in the sense of Fontaine. \par

On the other hand, by\footnote{We stress that the Minkowski theorem is one of the ingredients of Ta\"ibi's construction of ${\rm r}_{\pi, {\rm spin},\iota}$.} \cite[Thm. 2]{taibiAn}, there exists also a continuous semisimple morphism\footnote{
We use here that for $g=3$, $g(g+1) \equiv 0 \bmod 4$. 
For general ${\rm PGSp}_{2g}$, Ta\"ibi's statement involves ${\rm GSpin}$ instead of ${\rm Spin}$.}
$${\rm r}_{\pi, {\rm spin},\iota}  : {\rm Gal}(\overline{\Q}/\Q) \rightarrow {\rm Spin}_7(\overline{\Q}_\ell)$$
which is unramified outside $\ell$ and satisfies
\begin{equation}
\label{eq:proptaibi}
{\rm r}_{\pi, {\rm spin},\iota}({\rm Frob}_p)^{\rm ss} = \iota({\rm c}(\pi_p)), \, \,\, \textrm{for all primes $p \neq \ell$.}
\end{equation}
Again, this morphism 
${\rm r}_{\pi, {\rm spin},\iota}$ is unique up to conjugacy, and we may assume it satisfies $\rho \circ {\rm r}_{\pi, {\rm spin},\iota} =  {\rm r}_{\pi, {\rm std},\iota}$. Ta\"ibi also shows that ${\rm r}_{\pi, {\rm spin},\iota}$ is crystalline at $\ell$.

\begin{proof} (of Prop.~\ref{prop:tempendosiegelsp6spin})
Fix $\ell$ and $\iota$ as above. For $i=1,2,3$, let us denote ${\rm r}_{i} :  {\rm Gal}(\overline{\Q}/\Q) \rightarrow {\rm GL}_2(\overline{\Q}_\ell)$
the semisimple Galois representation attached by Deligne to $\pi_i$, or more precisely, to the algebraic representation $\pi'_i:=\pi_i|.|^{1/2}$ of ${\rm GL}_2$, and to $\iota$. It is unramified outside $\ell$ and satisfies ${\rm r}_{i}({\rm Frob}_p)^{\rm ss} = \iota({\rm c}((\pi'_i)_p))$ for $p\neq \ell$, hence ${\rm r}_{i}^\vee \simeq {\rm r}_{i} \otimes \omega_\ell^{-1}$ where $\omega_\ell$ is the $\ell$-adic cyclotomic character. The morphism $\nu$ (Formula~\ref{eq:morphnu}) extends to $\nu : {\rm GSpin}_4 \times {\rm GSpin}_3 \rightarrow {\rm GSpin}_7$, and we have exceptional 
isomorphisms $${\rm GSpin}_4 \simeq ({\rm GL}_2 \times {\rm GL}_2)^{\det_1 = \det _2}\,\,\,\textrm{and}\,\,\, {\rm GSpin}_3 \simeq {\rm GL}_2.$$ 
Using these isomorphisms, we may view $({\rm r}_{1} ,{\rm r}_{2} )$ and ${\rm r}_{3}^\vee$ as ${\rm GSpin}_n(\overline{\Q}_\ell)$-valued with $n=4,3$ respectively, hence composing with $\nu$ we obtain a semisimple morphism 
$\sigma   : {\rm Gal}(\overline{\Q}/\Q) \rightarrow {\rm GSpin}_7(\overline{\Q}_\ell)$ satisfying by construction
\begin{equation}
\label{eq:cookrho}
\left\{
\begin{array}{c}
{\rm std} \circ \sigma \,\, \simeq \,\,{\rm r}_{1} \otimes {\rm r}_{2}^\vee \,\oplus\, {\rm Sym}^2 {\rm r}_3 \otimes \omega_\ell^{-1}\, \,\simeq\,\, {\rm std} \circ {\rm r}_{\pi, {\rm std},\iota},\\ \\
{\rm spin} \circ \sigma \,\, \simeq \,\,({\rm r}_{1} \oplus {\rm r}_{2})\,\otimes {\rm r}_3^\vee\,\,\simeq\,\, ({\rm r}_{1} \otimes {\rm r}_{3}^\vee)\, \oplus\,( {\rm r}_{2} \otimes {\rm r}_{3}^\vee).
\end{array}\right.
\end{equation}
\par \noindent In particular, we have ${\rm sim}\, \circ \sigma = \det {\rm r}_1 \det {\rm r}_3^\vee= \omega_\ell \cdot \omega_\ell^{-1}=1$, {\it i.e.} ${\rm Im}\, \sigma  \subset {\rm Spin}_7(\overline{\Q}_\ell)$.
Up to conjugating $\sigma$ if necessary, we may thus assume $\rho \circ \sigma = {\rm r}_{\pi, {\rm std},\iota}$.
But this implies that there is a continuous character $\chi : {\rm Gal}(\overline{\Q}/\Q) \rightarrow \{ \pm 1\}$ satisfying
$$\sigma \,\,= \,\,\chi\,{\rm r}_{\pi, {\rm spin},\iota}.$$
As $\sigma$ and ${\rm r}_{\pi, {\rm spin},\iota}$ are unramified outside $\ell$ (resp. crystalline at $\ell$), so is $\chi$, since $\chi$ is a subrepresentation of $({\rm spin} \circ \sigma) \otimes ({\rm spin} \circ {\rm r}_{\pi, {\rm spin},\iota})^\vee$. As $\chi$ has finite order, it is thus unramified at $p$, hence everywhere. But there is no quadratic extension of $\Q$ unramified at all finite primes, so we have $\chi=1$ and $\sigma ={\rm r}_{\pi, {\rm spin},\iota}$. Observe that we have
$$(\pi_1 \boxplus \pi_2) \boxtimes \pi_3 = (\pi'_1 \boxplus \pi'_2) \boxtimes (\pi_3')^\vee.$$
By \eqref{eq:proptaibi} and the definition of $\sigma$, this proves ${\rm c}(\pi_p)=({\rm c}((\pi_1)p) \boxplus {\rm c}((\pi_2)_p) \boxtimes {\rm c}((\pi_3)_p)$ for all $p \neq \ell$.
We conclude by this same argument applied to any other $\ell$. 
\end{proof}

\vskip 5pt

\subsection{\bf Shape of Spin lifting: cuspidal $\Psi(\pi, {\rm std})$} \label{SS:cuspidal}
We now consider the shape of the Spin lifting  $\Psi(\pi, {\rm spin})$ when $\Psi(\pi, {\rm std})$ is (tempered) cuspidal.
Here, the  exceptional group ${\rm G}_2$ plays a crucial role. Recall that there is an embedding (well-defined up to conjugacy) 
$$\eta : {\rm G}_2(F) \rightarrow {\rm Spin}_7(F)$$ 
over any algebraically closed field $F$ (below, of char. 0). As is well-known, the stabilizers in ${\rm Spin}_7(F)$ 
of the non-isotropic vectors in its Spin representation (which we recall is orthogonal) are exactly the conjugate of 
$\eta({\rm G}_2(F))$.

\begin{definition} 
\label{def:typeG2} 
Let $\pi$ be an automorphic representation of ${\rm GSp}_6$ over the number field $k$. 
We say that $\pi$ is of {\rm type ${\rm G}_2$} if for almost all finite places $v$ of $k$, the Satake parameter ${\rm c}(\pi_v)$ meets 
$\eta({\rm G}_2(\C))$. 
\end{definition}

If $r : \Gamma \rightarrow {\rm Spin}_7(F)$ is any group homomorphism, we also say that  $r$ is {\it of type ${\rm G}_2$} if $r(\Gamma)$ is conjugate to a subgroup of $\eta({\rm G}_2(F))$.

\begin{prop} 
\label{prop:critG2}
Let $\pi$ be a cuspidal representation of ${\rm GSp}_6$ over $\Q$ of Siegel type.
The following are equivalent:
\begin{itemize}
\item[(i)] $\pi$ is of type ${\rm G}_2$,\ps
\item[(ii)] ${\rm spin}({\rm c}(\pi_p))$ has the eigenvalue $1$ for almost all primes $p$,\ps
\item[(iii)] $\Psi(\pi,{\rm spin}) ={\bf 1}\boxplus \Psi(\pi,{\rm std})$,\ps
\item[(iv)] for some $\ell$ and $\iota$, ${\rm r}_{\pi,{\rm spin},\iota}$ is of type ${\rm G}_2$,\ps
\item[(v)] for all $\ell$ and $\iota$, ${\rm r}_{\pi,{\rm spin},\iota}$ is of type ${\rm G}_2$.\ps
\end{itemize}  
\end{prop}

\begin{proof}
Recall the following two facts : 
\vskip 5pt
\begin{itemize}
\item[(a)] a semisimple element $g$ in ${\rm Spin}_7(F)$ has the eigenvalue $1$ in the spin representation if and only if it is conjugate to an element of $\eta({\rm G}_2(F))$; 
\item[(b)]  if ${\rm std}$ and ${\rm spin}$ are the standard and spin representations of ${\rm Spin}_7$, then the representation ${\rm spin} \circ \eta$ of ${\rm G}_2$ is isomorphic to $1 \oplus {\rm std} \circ \eta$. 
\end{itemize}
Applied to $F=\C$, fact (a) shows (i) $\Leftrightarrow$ (ii). Moreover, by fact (b), assertion (i) implies the equality of semisimple $\GL_8(\C)$-conjugacy classes ${\rm spin}({\rm c}(\pi_p)) = 1 \oplus {\rm std}({\rm c}(\pi_p))$ for all but finitely many $p$. But this implies $\Psi(\pi,{\rm spin}) ={\bf 1}\boxplus \Psi(\pi,{\rm std})$ by the Jacquet-Shalika theorem, so we have proved (i) $\Rightarrow$ (iii).  Observe that (v) $\Rightarrow$ (iv) is trivial, and that (iv) $\Rightarrow$ (i) follows from fact (b) for $F=\overline{\Q}_\ell$, so it only remains to show (iii) $\Rightarrow$ (v).
 
Assume (iii) holds. Fix $\ell$ and $\iota$, and set $r={\rm r}_{\pi,{\rm std},\iota}$ and $r'={\rm r}_{\pi,{\rm spin},\iota}$. 
Let $S$ be a Spinor module for ${\rm Spin}_7(\overline{\Q}_\ell)$; recall that $S$ is equipped with a nondegenerate invariant symmetric bilinear form.
By (iii) and the Chebotarev theorem, we have an isomorphism of semisimple representations 
${\rm spin}\circ r' \simeq 1\oplus {\rm std} \circ r$.  In particular, the subspace $U \subset S$ of fixed points under $r'({\rm Gal}(\overline{\Q}/\Q))$ is nonzero; it is also nondegenerate as $r'$ is semisimple.
It follows that $U$ contains a nonzero nonisotropic vector, hence that $r'$ is of type ${\rm G}_2$. 
See~\cite[Thm. 3.4]{cheg2} for similar ideas.
\end{proof}

\begin{remark} Let $\pi$ be a cuspidal representation of ${\rm PGSp}_6$ over $\Q$ of Siegel type.
\label{rem:g2weaklift}
\begin{itemize} 
\item[(i)] There is an equivalence:
\[ \text{$\pi$ is of type ${\rm G}_2$  $\iff$ $\pi$ is a functorial lift from (a form of) ${\rm G}_2$.} \]
 This follows from \cite{Vo} for generic $\Psi(\pi,{\rm std})$ and by \cite[Thm 10.1 and Thm 10.2]{GS1} in general.
\vskip 5pt

\item[(ii)] It is not difficult to prove that, for any $\ell$ and $\iota$, ${\rm r}_{\pi,{\rm spin},\iota}$ is of type ${\rm G}_2$ if, and only if, the image of ${\rm r}_{\pi,{\rm std},\iota}$ in ${\rm SO}_7(\overline{\Q}_\ell)$ is contained in a subgroup isomorphic to ${\rm G}_2(\overline{\Q}_\ell)$.

\vskip 5pt
\item[(iii)] If $\pi$ is of type ${\rm G}_2$ and of weights $k_1\geq k_2\geq k_3$, then we have $w_4=0$ by Proposition~\ref{prop:critG2} (iii) (at the Archimedean place) and \eqref{eq:weights}, and thus $k_1=k_2+k_3$.
\end{itemize}
\end{remark}
\vskip 5pt

The following proposition determines the possible shape of $\Psi(\pi, {\rm spin})$.
\vskip 5pt

\begin{prop} 
\label{prop:spincusp}
Assume $\Psi(\pi,{\rm std})$ is cuspidal, then $\Psi(\pi,{\rm spin})$ is cuspidal if, and only if, 
$\pi$ is not of type ${\rm G}_2$.
\end{prop}

\begin{proof} Assume $\Psi(\pi,{\rm std})$ is cuspidal and set $r={\rm r}_{\pi,{\rm std},\iota}$.
By Theorem D in~\cite{PT}, we may choose $\ell$ and $\iota$ such that ${\rm std} \circ r$ is irreducible.
Set $F=\overline{\Q}_\ell$ and let $\Gamma \subset {\rm SO}_7(F)$ be the Zariski-closure of the image of $r$.
As $r$ is crystalline at $\ell$ and unramified outside $\ell$, so are all tensor powers $r^{\otimes n}$ with $n\geq 1$, and the Minkowski theorem implies that $\Gamma$ is connected.\par

As $\Gamma$ acts irreducibly in ${\rm std}$, it is well-known that we have either 
$\Gamma={\rm SO}_7(F)$, or $\Gamma$ is a principal ${\rm PGL}_2(F)$ in ${\rm SO}_7(F)$, or $\Gamma$
is a conjugate of $\eta({\rm G}_2(F)$). Set now $r'={\rm r}_{\pi,{\rm spin},\iota}$ and $\Gamma'={\rm Im}\, r'$.
Then $\Gamma'$ is connected as well, for the same reason as above, so we must have $\Gamma'={\rm Spin}_7(F)$, or $\Gamma'$ is a principal ${\rm PGL}_2(F)$ in ${\rm Spin}_7(F)$, or $\Gamma$ is a conjugate of $\eta({\rm G}_2(F)$), respectively. In the first situation, ${\rm spin} \circ r'$ is thus irreducible. 
But if we write $$\Psi(\pi,{\rm spin}) = \pi_1 \boxplus \cdots \boxplus \pi_k,$$ the representation ${\rm spin} \circ r'$ is also the direct sum of that associated to the $\pi_i$'s (which are algebraic, selfdual and essentially regular by \eqref{eq:weights}). This forces $k=1$, and we are done. 
In the remaining two cases, $r'$ is of type ${\rm G}_2$, since a principal ${\rm PGL}_2$ in ${\rm Spin}_7$ (or ${\rm SO}_7$) may be conjugate inside ${\rm G}_2$, and we conclude by Prop.~\ref{prop:critG2}.
\end{proof}

See Section~\ref{sect:complementandexamples} for some examples and additional results about the $\pi$ such that $\Psi(\pi,{\rm std})$ is cuspidal.

\subsection{\bf Non-generic case}
Finally,  we consider the case where the A-parameter $\Psi(\pi,{\rm std})$ is not generic, which is not a priori covered by Thm.~\ref{thm:weakspinlift}.
By \cite[\S 9.3]{CR}, we have\footnote{For our purposes here, we could replace $S_2$ by $|.|^{1/2} \boxplus |.|^{-1/2}$, but Arthur's notation is more suggestive.} 
\begin{equation}
\label{eq:nontempsiegel}
\Psi(\pi,{\rm std}) \,\,= \,\,\pi_1 \boxtimes S_2 \,\boxplus\, {\rm Sym}^2 \pi_3,
\end{equation}
where $\pi_1$ and $\pi_3$ are cuspidal automorphic representations of ${\rm PGL}_2$ 
generated by holomorphic cuspidal eigenforms for ${\rm SL}_2(\Z)$ 
(again, the precise constraints on the weights of $\pi_1$ and $\pi_3$ are given {\it loc.\ cit.}). 
In this situation we simply {\it define} 
\begin{equation}
\label{eq:psispinontempered}
\Psi(\pi,{\rm spin}) \,\,:= \,\,(\pi_1 \boxplus S_2) \,\boxtimes \,\pi_3 \,\,= \,\,(\pi_1 \boxtimes \pi_3)\, \boxplus\, \pi_3 \boxtimes S_2.
\end{equation}
In Theorem \ref{thm:fullcasesiegeltype} below, we verify that $\Psi(\pi, {\rm spin})$ is indeed the A-parameter of a Spin lifting of $\pi$. 
Together with Theorem~\ref{thm:weakspinlift},  this establishes the existence of the Spin lifting  $\Psi(\pi,{\rm spin})$ {\it for all $\pi$ of Siegel type}.
 
 \begin{thm} 
 \label{thm:fullcasesiegeltype}
 For all cuspidal $\pi$ of Siegel type, $\Psi(\pi,{\rm spin})$ is the A-parameter of a Spin lifting of $\pi$. 
 Moreover, this lifting is strong at all unramified places, as well as at the Archimedean place in the sense of infinitesimal characters. 
 \end{thm}

\begin{proof} We may assume 
$\Psi(\pi,{\rm std})=\pi_1 \boxtimes S_2 \,\boxplus\, {\rm Sym}^2 \pi_3$ as above. 
We apply exactly the same argument as in the proof of Prop.~\ref{prop:tempendosiegelsp6spin},
with $\pi_2$ there replaced with the trivial representation of ${\rm PGL}_2$, and setting ${\rm r}_2:=1 \oplus \omega_\ell$. 
Defining $\sigma : {\rm Gal}(\overline{\Q}/\Q) \rightarrow {\rm Spin}_7(\overline{\Q}_\ell)$ as in that proof, 
the argument shows verbatim that $\sigma$ is conjugate to ${\rm r}_{\pi,{\rm spin},\iota}$, hence that $\Psi(\pi,{\rm spin})$ 
is a spin lifting of $\pi$, which is strong at all finite places except maybe $\ell$ (the assertion about infinitesimal characters is obvious from the shape of $\Psi(\pi,{\rm spin})$). Using another $\ell$ gives the full result. 
\end{proof} 
\vskip 5pt
\subsection{\bf Application to Spin L-functions}
We end with an interesting corollary about spinor ${\rm L}$-functions.
Assume $\pi$ is of Siegel type and generated by a Siegel cuspidal eigenform of weights $k_1 \geq k_2 \geq k_3 \geq 4$.
Following Langlands, recall that for any prime $p$ the local spin $L$-factor is defined as 
$${\rm L}(s,\pi_p,{\rm spin}) = \frac{1}{\det ( 1-{\rm spin}({\rm c}(\pi_p) )p^{-s})}.$$
Set $\Gamma_\C(s)= 2 (2\pi)^{-s} \Gamma(s)$, with $\Gamma(s)$ the Euler $\Gamma$-function, define the $w_i$ as in \eqref{eq:weights}, 
and set $${\rm L}(s,\pi_\infty,{\rm spin}) \,\,:=\,\, \prod_{i=1}^4 \Gamma_\C(s+w_i).$$ 
We then define ${\rm L}(s,\pi,{\rm spin})$ as the product, over all places $v$ of $\Q$, of the local L-factors ${\rm L}(s,\pi_v,{\rm spin})$. 
We know since Langlands that this Euler product is absolutely convergent for ${\rm Re}(s)$ big enough.
Recall also the standard $L$-function ${\rm L}(s,\pi,{\rm std})$ of $\pi$, whose analytic properties are now well-known (for any genus).

\begin{thm} 
\label{thm:Lspinsiegel}
Assume $\pi$ is a cuspidal of Siegel type. Then:
\begin{itemize}
\item[(i)] ${\rm L}(s,\pi,{\rm spin})$ has a meromorphic continuation to all of $\C$, 
with at most a simple pole at $s=0$ and $1$, and no other poles. It satisfies ${\rm L}(s,\pi,{\rm spin})={\rm L}(1-s,\pi,{\rm spin})$. \ps
\item[(ii)] Moreover, ${\rm L}(s,\pi,{\rm spin})$  has a pole at $s=1$ if, and only if, $\pi$ is of type ${\rm G}_2$, in which case  ${\rm L}(s,\pi,{\rm spin})=\zeta(s) \cdot  {\rm L}(s,\pi,{\rm std})$. 
\end{itemize}
\end{thm}

There has been an number of past works on the spinor $L$-functions 
of cuspidal automorphic representations $\pi$ of ${\rm PGSp}_6$ over number fields. 
For a globally generic $\pi$, a partial representation by a Rankin-Selberg type integral 
was found in \cite{BG} and studied in \cite{Vo}. 
For our Siegel type $\pi$'s, a weaker statement had also been proved by Pollack in \cite[Thm. 1.2]{pollack}, 
who assumed that the associated Siegel modular form has a nonzero Fourier coefficient 
at the maximal order of a definite quaternion algebra. Our method here is  quite different, 
and ultimately relies on the properties of the Godement-Jacquet $L$-functions; 
it also provides the most complete result for Siegel type $\pi$.

\begin{proof} 
In all cases, there are integers $k\geq 1$ and $n_i,d_i\geq 1$, 
with $\sum_{i=1}^k n_i d_i=8$, and
$$\Psi(\pi,{\rm spin}) = \boxplus_{i=1}^k  \pi_i \boxtimes S_{d_i},$$ 
for some selfdual level $1$ cuspidal automorphic representations $\pi_i$ of ${\rm GL}_{n_i}$. 
We have either $d_i=1$ and $\pi_i$ is orthogonal, or $d_i=2$ and $\pi_i$ is symplectic (with $n_i=2$).
Recall that the Godement-Jacquet standard $L$-function ${\rm L}(s,\pi_i)$ is entire if $\pi_i \neq {\bf 1}$,
and is equal to the completed $\zeta(s)$ otherwise. Moreover, it  satisfies the functional equation ${\rm L}(s,\pi_i)=\epsilon(\pi_i,1/2){\rm L}(1-s,\pi_i)$
for some sign $\epsilon(\pi_i,1/2)=\pm 1$, equal to $1$ if $\pi_i$ is orthogonal (Arthur).\par \smallskip
When $\Psi(\pi,{\rm spin})$ is generic, {\it i.e.} $d_i=1$ for all $i$, 
Formula~\eqref{eq:weights} and {\it Clozel's purity lemma} (see \cite[Prop. 8.2.13]{CL}) 
show that the Langlands parameter of $\boxplus_{i=1}^k (\pi_i)_\infty$ is 
$\oplus_{i=1}^4 {\rm I}_{w_i}$ with ${\rm I}_w = {\rm Ind}_{{\rm W}_\C}^{{\rm W}_\R} (z \mapsto (z/|z|)^{2w})$ 
for $w \in \frac{1}{2}\Z$. We have thus (equality at all places)
\begin{equation}
\label{eq:equalspingodementjacquet}
{\rm L}(s,\pi,{\rm spin}) = \prod_{i=1}^k {\rm L}(s,\pi_i).
\end{equation}
As ${\bf 1}$ cannot appear twice as a $\pi_i$ ({\it e.g.} at the Archimedean place), this proves (i). 
Part (ii) follows from ${\rm L}(1,\pi_i) \neq 0$ (Jacquet-Shalika) and Prop.~\ref{prop:critG2}.\par
\smallskip
Assume now $\Psi(\pi,{\rm spin})$ is nongeneric, {\it i.e.} satisfies \eqref{eq:psispinontempered}, or in Langlands form
\begin{equation}
\label{eq:parnontemplform}
\Psi(\pi,{\rm spin})\,\, =\,\, \pi_1 \,\boxtimes\, \pi_3\, \boxplus\, \pi_3|.|^{1/2} \,\boxplus\, \pi_3|.|^{-1/2}.
\end{equation}
Denote by $\pm u$, with $u \in \frac{1}{2}\Z_{>0} \smallsetminus \Z$, the $2$ eigenvalues of ${\rm c}((\pi_3)_\infty)$.
An inspection of the infinitesimal character shows that we may write $\{w_1,w_2,w_3,w_4\}=\{u_1,u_2,u+\frac{1}{2},u-\frac{1}{2}\}$. The Langlands parameter of $\Psi(\pi,{\rm spin})_\infty$ is thus 
${\rm I}_{u_1} \oplus {\rm I}_{u_2} \oplus {\rm I}_u |.|^{1/2} \oplus {\rm I}_u|.|^{-1/2}$, whose standard $L$-function coincides with our definition of ${\rm L}(s,\pi_\infty,{\rm spin})$, showing again 
$${\rm L}(s,\pi,{\rm spin}) = {\rm L}(s,\pi_1 \boxtimes \pi_3) {\rm L}(s-1/2,\pi_3){\rm L}(s+1/2,\pi_3).$$
Again, an unramified character $|.|^s$ appears at most once in the $L$-parameter of $\Psi(\pi,{\rm spin})_\infty$, and only for $s=0$.
The $\epsilon$-factor of ${\rm L}(s,\pi,{\rm spin})$ is $\epsilon(\pi_1 \otimes \pi_3,1/2) \cdot \epsilon(\pi_3,1/2)^2 = 1 \cdot (\pm 1)^2=1$. We conclude as above.
\end{proof}
 
\vskip 5pt

\subsection{\bf Additional remarks and some examples} 
\label{sect:complementandexamples}
In this section, we are interested in the set $\Pi_c$ of cuspidal automorphic representations $\pi$ of Siegel type of ${\rm PGSp}_6$ such that $\Psi(\pi,{\rm std})$ is cuspidal. For a prime $\ell$, we denote by ${\rm I}_\ell$ the set of field isomorphisms $\iota : \C \overset{\ssim}{\rightarrow} \overline{\Q}_\ell$, and for $\pi \in \Pi_c$, we define ${\rm I}_{\pi,\ell}$ as the set of $\iota \in {\rm I}_\ell$ such that ${\rm r}_{\pi,{\rm std},\iota}$ is irreducible. 
We know by \cite{PT} that ${\rm I}_{\pi,\ell} \neq \emptyset$ for infinitely many $\ell$ (and conjecturally, we should have ${\rm I}_{\pi,\ell} = {\rm I}_\ell$ for all $\ell$). 
\vskip 5pt
Recall that for any algebraically closed field $F$ (say of characteristic $0$), there are only $3$ conjugacy classes of connected reductive $F$-algebraic subgroups of ${\rm Spin}_7(F)$ with irreducible standard representation, namely that of ${\rm Spin}_7(F)$, of ${\rm G}_2(F)$, and of the principal ${\rm PGL}_2(F)$; we denote by $\mathcal{Z}_F$ this three-element set.   
 For $\pi \in \Pi_c$, and $\iota \in {\rm I}_\ell$, we denote by 
$${\rm Z}_{\pi,\iota} \subset {\rm Spin}_7(\overline{\Q}_\ell)$$ 
the Zariski-closure of the image of ${\rm r}_{\pi,{\rm spin},\iota}$; this is a reductive $\overline{\Q}_\ell$-algebraic subgroup of ${\rm Spin}_7(\overline{\Q}_\ell)$ well-defined up to conjugation. As explained in the proof of Prop.~\ref{prop:spincusp} we have ${\rm Z}_{\pi,\iota} \in \mathcal{Z}_{\overline{\Q}_\ell}$ for $\iota \in {\rm I}_{\pi,\ell}$ (and ${\rm Z}_{\pi,\iota}$ is connected for all $\ell$ and $\iota \in {\rm I}_\ell$).

\begin{prop} 
\label{prop:defZpi}
For all $\pi \in \Pi_c$, there is a unique ${\rm Z}_\pi \in \mathcal{Z}_{\C}$ satisfying ${\rm Z}_\pi \times_\iota \overline{\Q}_\ell \simeq {\rm Z}_{\pi,\iota}$ for all $\ell$ and $\iota \in {\rm I}_{\pi,\ell}$. 
\end{prop}

\begin{proof} Fix $\pi \in \Pi_c$ and let $E \subset \C$ be the number field generated by the Hecke eigenvalues of a Siegel eigenform for ${\rm PGSp}_6(\Z)$ generating $\pi$. The collection of semisimple ($8$-dimensional) Galois representations $r_\iota:={\rm spin}\circ \rho_{\pi,\iota}$, with $(\ell,\iota)$ varying, is {\it compatible} in the following sense: for any prime $p$, there is a polynomial $Q_p \in E[t]$ such that for all $\ell \neq p$ and all $\iota \in {\rm I}_\ell$, we have $\det (t - r_\iota({\rm Frob}_p) ) = \iota(Q_p(t))$. By a classical argument of Serre \cite[p .58]{serre} (see also \cite[Rem. 6.13]{LP}), in such systems the rank of the Zariski closure of ${\rm Im}\, r_\iota$ (here $3$, $2$, or $1$) does not depend on the choice of $(\ell,\iota)$. The proposition follows.
\end{proof}
\begin{remark}
The group ${\rm Z}_\pi(\C) \subset {\rm Spin}_7(\C)$ is important as its compact form $K$ is the so-called {\it Sato-Tate group} of $\pi$ {\rm (}here, either ${\rm Spin}(7)$, the compact ${\rm G}_2$, or ${\rm SO}(3)${\rm )}:  the Satake parameters ${\rm c}(\pi_p) \cap K$ are conjecturally equidistributed in the space of conjugacy classes of $K$.
\end{remark}

 We now discuss some examples of $\pi \in \Pi_c$ with all possible ${\rm Z}_\pi$. By \cite[Cor. 1.10]{CR} and \cite[Thm. B]{taibisiegel}, there is an explicit formula for the number ${\rm N}(a,b,c)$
of $\pi \in \Pi_c$ with weights $(k_1, k_2, k_3)=(a+1,b+2,c+3)$. For instance, for $a \leq 13$ there are exactly $8$ triples $(a,b,c)$ such that ${\rm N}(a,b,c)$ is nonzero, 
and actually equal to $1$, namely 
$$(12, 8, 4), \, \, (13, 8, 5), \, \, (13,10, 3), \, \, (13, 10, 5), \, \, (13, 10, 7), \, \, (13, 12, 5), \, \, (13, 12, 7), \, \, (13, 12, 9).$$
%(24; 16; 8); (26; 16; 10); (26; 20; 6); (26; 20; 10); (26; 20; 14); (26; 24; 10); (26; 24; 14); (26; 24; 18)
\vskip 5pt
(i) By Proposition~\ref{prop:spincusp}, we know that ${\rm Z}_\pi = {\rm Spin}_7$ if and only if $\pi$ is not of type ${\rm G}_2$, and 
by Remark~\ref{rem:g2weaklift} (iii), we also know that if $\pi$ is of type ${\rm G}_2$ then we have $a=b+c$. 
It follows that for all but the first $3$ triples $(a,b,c)$ above, the corresponding representation $\pi$ satisfies ${\rm Z}_\pi = {\rm Spin}_7$. More generally, by Ta\"ibi's formula loc.\ cit. there are infinitely many $\pi$ with ${\rm Z}_\pi = {\rm Spin}_7$, the asymptotics of their number for ${\rm Min} \{a-b,b-c,c\} \to +\infty$ being polynomial in $(a,b,c)$ and equivalent to ${\rm N}(a,b,c)$.
\vskip 5pt

(ii) The case ${\rm Z}_\pi = {\rm G}_2$ also appears. Indeed, by \cite[Thm. 6.12]{cheg2}, for the first $3$ triples above the corresponding $\pi$ is of type ${\rm G}_2$. Moreover, using Galois theoretic arguments we can show that if ${\rm Z}_\pi = {\rm PGL}_2$ then we have $(a,b,c)=(3c,2c,c)$ and $c$ is odd (see Proposition~\ref{prop:sym6} below for a stronger result). This shows that for the first $3$ triples above, the corresponding $\pi$ actually satisfies ${\rm Z}_\pi = {\rm G}_2$. 
See \cite[Ch. 8]{CR} for a conjectural formula for the number, denoted ${\rm G}_2(2b,2c)$ {\it loc.\ cit.}, of $\pi \in \Pi_c$ of type ${\rm G}_2$ and given $(a,b,c)$. 

\vskip 5pt

(iii) Finally, the following proposition classifies the $\pi \in \Pi_c$ with ${\rm Z}_\pi={\rm PGL}_2$. 
Let $\Sigma$ be the set of cuspidal automorphic representations of ${\rm PGL}_2$ generated by a holomorphic cusp form for ${\rm SL}_2(\Z)$. As a special case of the work of Newton-Thorne \cite[Thm. A]{NT}, for any $\sigma \in \Sigma$ there 
is a cuspidal automorphic representation ${\rm Sym}^6 \sigma$ of ${\rm PGL}_7$ over $\Q$ which is a strong functorial lifting of $\sigma$ for ${\rm Sym}^6 : {\rm SL}_2(\C) \longrightarrow {\rm SL}_7(\C)$. This representation ${\rm Sym}^6 \sigma$ is selfdual orthogonal and algebraic regular. 
By Arthur's multiplicity formula \cite{A} (see also \cite[\S 8.5.1]{CL}), there is thus a unique, level $1$, cuspidal automorphic representation $\pi^\flat$ of ${\rm Sp}_6$ over $\Q$ such that $\pi_\infty^\flat$ is a holomorphic discrete series and $\Psi(\pi^\flat,{\rm std})={\rm Sym}^6 \sigma$. 
By \cite[Prop. 3.2]{CR}, there is a unique $\pi_\sigma$ of Siegel type such that ${\pi_\sigma}|_{{\rm Sp}_6}$ contains $\pi^\flat$.
\par

\begin{prop} 
\label{prop:sym6} 
The map $\sigma \mapsto \pi_\sigma, \Sigma \rightarrow \Pi_c,$ is a bijection onto the subset of $\pi \in \Pi_c$ with ${\rm Z}_\pi={\rm PGL}_2$.
\end{prop}

\begin{proof} Assume $\pi \in \Pi_c$ satisfies ${\rm Z}_\pi={\rm PGL}_2$ and has weights $(a+1,b+2,c+3)$. It only remains to show the existence of $\sigma \in \Sigma$ such that $\pi=\pi_\sigma$. For $\iota \in {\rm I}_{\pi,\ell}$, we may view ${\rm r}_{\pi,{\rm std},\iota}$ as a $3$-dimensional Galois representation $$\rho: {\rm Gal}(\overline{\Q}/\Q) \longrightarrow {\rm PGL}_2(\overline{\Q}_\ell) \simeq {\rm SO}_3(\overline{\Q}_\ell)$$
with Zariski-dense image, unramified outside $\ell$ and crystalline at $\ell$. 
Recall that the standard representation of the principal ${\rm PGL}_2$ in ${\rm Spin}_7$ is the irreducible $7$-dimensional representation. As the Hodge-Tate weights of ${\rm r}_{\pi,{\rm std},\iota}$ are $0,\pm a,\pm b,\pm c$, and as ${\rm Z}_{\pi,\iota}={\rm PGL}_2$, the Hodge-Tate weights
of $\rho$ must be $-c,0,c$ (and we must have $b=2c$ and $a=3c$). 
By \cite[Prop. 5.1.1]{taibiAn}, $\rho$ lifts to a continuous representation $\widetilde{\rho} : {\rm Gal}(\overline{\Q}/\Q) \longrightarrow {\rm GL}_2(\overline{\Q}_\ell)$ which is unramified outside $\ell$, crystalline at $\ell$ with Hodge-Tate weights $0$, $c$ (and necessarily irreducible). 
By construction, we have an isomorphism 
\begin{equation}
\label{eq:sym6iso}
{\rm Sym}^6 \widetilde{\rho} \otimes (\det \widetilde{\rho})^{-3} \simeq {\rm std} \circ {\rm r}_{\pi,{\rm std},\iota}.
\end{equation}
As the complex conjugations $g \in {\rm Gal}(\overline{\Q}/\Q)$ have a non-trivial image in ${\rm r}_{\pi,{\rm std},\iota}$ by \cite[Prop. 1]{taylor}, we also have $\rho(g) \neq 1$, and then $\det \widetilde{\rho}(g)=-1$ (so $c$ is odd). 
We may assume $\ell \geq 5$, so by the Fontaine-Mazur conjecture for $\GL_2$ in that case (Emerton, Kisin, Pan... see e.g. \cite{pan}), $\widetilde{\rho}$ is an $\ell$-adic representation associated by Deligne to some $\sigma \in \Sigma$. 
By \eqref{eq:sym6iso}, this implies $\Psi(\pi,{\rm std})= {\rm Sym}^6 \sigma$, hence $\pi=\pi_\sigma$, concluding the proof.
\end{proof}

In particular, the first $\pi \in \Pi_c$ with ${\rm Z}_\pi = {\rm PGL}_2$ is a lift from Ramanujan's $\Delta$ function and occurs for $(a,b,c)=(33,22,11)$. 
  
  \section{\bf Appendix A: Similitude Theta Correspondence}
In this appendix, we collect and establish some results from the theory of  local theta correspondence for similitude classical groups that we need in the main body of the article, especially  concerning the theta lifts of unramified representations. Throughout this appendix, $F$ will denote a non-Archimedean local field with ring of integers $\mathcal{O}_F$ and residue field of cardinality $q$. 

\vskip 5pt
\subsection{\bf Setup}
We briefly recall the setup of isometry and similitude theta correspondences over the local field $F$. Fix a nontrivial additive charatcer $\psi: F \rightarrow \C^{\times}$. Suppose that $(W, \langle -,-\rangle_W)$ is a symplectic vector space of dimension $2n$ and $(V,q)$ a quadratic space of dimension $2m$ with associated symmetric bilinear form $\langle-,-\rangle_V$ .  For simplicity, we shall assume the following running hypotheses throughout this appendix: 
\vskip 5pt

\begin{itemize}
\item $(V,q)$ is split, so that  ${\rm disc}(V,q) =1 \in F^{\times}/ F^{\times 2}$ and the orthogonal group $\O(V)$ is split;
\item $ m =\dim V/2  > n = \dim W/2$.
\end{itemize}  
 Then we have the isometry dual pair
\[ i: \Sp(W) \times \O(V) \longrightarrow \Sp(V \otimes W). \]
Depending on $\psi$, the map $\iota$ can be lifted to the metaplectic cover $\Mp(V \otimes W)$:
\[  i_{\psi}: \Sp(W) \times \O(V) \longrightarrow \Mp(V \otimes W). \]
The Weil representation $\omega_{\psi}$ of $\Mp(V \otimes W)$ can then be pulled back via $i_{\psi}$ to yield the Weil representation $\omega_{V,W, \psi}$ of $\Sp(W) \times \O(V)$.  
\vskip 5pt

For a representation $\pi \in {\rm Irr}(\Sp(W))$, we consider its big theta lift to $\O(V)$, defined by:
\[  \Theta_{\psi}(\pi) = (\omega_{V,W, \psi} \otimes \pi^{\vee})_{\Sp(W)}, \]
which is a smooth representation of $\O(V)$ (for a representation $U$ of a group $G$ we denote by $U_G$ the largest quotient of $U$ over which $G$ acts trivially). The Howe duality theorem says that $\Theta_{\psi}(\pi)$ has finite length (possibly zero) and a unique irreducible quotient $\theta_{\psi}(\pi)$ (possibly zero). Hence, one has a map
\[ \theta_{\psi}:  {\rm Irr}(\Sp(W)) \longrightarrow {\rm Irr}(\O(V)) \cup \{0 \}. \]
Moreover, the Howe duality theorem further asserts that this map is injective on the domain where it does not vanish.  The map $\theta_{\psi}$ is the local theta lifting for isometry groups.
\vskip 5pt

 Under our hypothesis that $m > n$, it turns out that if $\theta_{\psi}(\pi) \in {\rm Irr}(\O(V))$ is nonzero, then it remains irreducible when restricted to the special orthogonal group $\SO(V)$. Hence, composing with restriction to $\SO(V)$, one has in fact a map
\[  \theta_{\psi} :  {\rm Irr}(\Sp(W)) \longrightarrow {\rm Irr}(\SO(V)) \cup \{0 \}. \]
  
 \vskip 5pt
 We would like to extend the above theory and results to the setting of similitude groups. 
 These similitude groups are defined by:
 \[ \GSp(W) = \{ (g, \lambda) \in \GL(W) \times \mathbb{G}_m: g^* (\langle -, - \rangle_W) = \lambda \cdot \langle -,  - \rangle_W \} \]
 and
 \[  \GO(V) = \{ (h, \lambda) \in \GL(V) \times \mathbb{G}_m: h^* (\langle -, - \rangle_V) = \lambda \cdot \langle -,  - \rangle_V \}. \]
 The similitude characters
 \[  {\rm sim}:  \GSp(W) \longrightarrow \mathbb{G}_m \quad \text{and} \quad {\rm sim}: \GO(V) \longrightarrow \mathbb{G}_m \]
 are given by the second projection.  By our hypotheses, these similitude characters are surjective onto $F^{\times}$ on taking $F$-valued points. 
 \vskip 5pt
 
   Consider now the group
  \[ R =  ( \GSp(W) \times \GO(V))^{{\rm sim}} = \{ ((g, h) \in \GSp(W) \times \GO(V):  {\rm sim}(g) \cdot {\rm sim}(h) =1\}. \]  
 One has the short exact sequences:
 \[  \begin{CD}
1 @>>> \Sp(W) @>>>  (\GSp(W) \times \GO(V))^{\sim} @>p>> \GO(V) @>>> 1. 
\end{CD} \]
\[   \begin{CD}
1 @>>> \O(V) @>>>  (\GSp(W) \times \GO(V))^{\sim} @>q>> \GSp(W) @>>> 1
\end{CD} \]
where $p$ and $q$ are the natural projections onto the relevant factors. Moreover, observe that
one has an exact sequence
\[  \begin{CD}
1 @>>> \Sp(W) \times \O(V) @>>> R @>>>  \mathbb{G}_m @>>> 1. \end{CD} \]

 \vskip 5pt
It turns out that the Weil representation $\omega_{V,W, \psi}$ can be extended to the slightly larger group $R= R(F)$.  
The extension is not unique but we fix one as in \cite{GT1} (which differs from the normalization in \cite{Ro}).  Then we define the similitude Weil representation as:
\[  \Omega_{V,W} := {\rm ind}_R^{\GSp(W)    \times \GO(V)} \omega_{V,W, \psi}. \]
It turns out that this representation is independent of the choice of $\psi$ (because $V$ has trivial discriminant).
For a representation $\pi \in {\rm Irr}(\GSp(W))$, we can then define its big theta lift (which is a smooth representation of $\GO(V)$) by
\[  \Theta(\pi) := (\Omega_{V,W} \otimes \pi^{\vee})_{\GSp(W)}. \]
One may also define $\Theta(\pi)$ as
\[  \Theta(\pi) = (\omega_{V,W, \psi} \otimes \pi^{\vee})_{\Sp(W)}. \]
 
\vskip 5pt

It follows from \cite{Ro, GT1} that the analog of the Howe duality theorem holds for similitude groups. In other words, for each $\pi \in {\rm Irr}(\GSp(W))$,
$\Theta(\pi)$ has finite length and unique irreducible quotient $\theta(\pi)$  (if nonzero). Hence, one obtains a map
\[ \theta: {\rm Irr}(\GSp(W)) \longrightarrow {\rm Irr}(\GO(V)) \cup \{ 0 \}, \]
which is injective on the domain of nonvanishing. This map $\theta$ is the local theta lifting for similitude groups. Moreover, under our definition of the extension of $\omega_{\psi}$ to $R$, it turns out that $\Theta(\pi)$ has central character equal to that of $\pi$.
As in the isometry case, for $\pi \in {\rm Irr}(\GSp(W))$, $\theta(\pi)$ is irreducible (or zero) when restricted to $\GSO(V)$ under our hypothesis that $m > n$. Hence, we have a map
\begin{equation} \label{E:GSO}
    \theta: {\rm Irr}(\GSp(W)) \longrightarrow {\rm Irr}(\GSO(V)) \cup \{ 0 \}. \end{equation}
 \vskip 10pt
 
\subsection{\bf Unramified representations}
\label{sec:unramifiedrep}
We will need to give an explicit  description of the maps $\theta_{\psi}$ and $\theta$ on the subset of unramified representations. 
 Let us recall this notion more precisely.
 \vskip 5pt

Suppose that $G$ is a split connected reductive group over $F$. Then $G$ has a reductive model (the Chevalley model) over the ring of integers $\mathcal{O}_F$, so that $K := G(\mathcal{O}_F)$ is a hyperspecial maximal compact subgroup.  Fix the tuple
\[ T \subset B \subset G \]
consisting of a maximal split torus $T$ contained in a Borel subgroup $B$ defined over $\mathcal{O}_F$. 
\vskip 5pt

The subset ${\rm Irr}_K(G(F)) \subset {\rm Irr}(G(F))$ of $K$-unramified representations consists of those irreducible representations $\pi$ such that $\pi^K \ne 0$, in which case $\dim \pi^K =1$.   The following are the main facts about unramified representations we need; they are largely consequences of the so-called Satake isomorphism.
\vskip 5pt

\begin{itemize}
\item  If $\pi \in {\rm Irr}_{K}(G(F))$, then $\pi$ is a constituent of a principal series representation ${\rm Ind}_{B}^{G} (\chi)$, where  
\[ \chi : T(\mathcal{O}_F) \backslash T(F)  =X_*(T) \longrightarrow \C^{\times} \]
is an unramified character of $T(F)$, well defined up to the action of the Weyl group.  Moreover, $\dim {\rm Ind}_{B}^{G} (\chi)^K =1$, so that $\pi$ is the unique $K$-unramified subquotient of ${\rm Ind}_{B}^{G} (\chi)$.
 \vskip 5pt

\item  With $T^{\vee}$ as the dual torus of $T$, note that
\[  \Hom(X_*(T), \C^{\times}) = X^*(T) \otimes \C^{\times} = X_*(T^{\vee}) \otimes \C^{\times} = T^{\vee}(\C). \]
Hence the Weyl-orbit of the  unramified character $\chi$ of $T(F)$  corresponds to the Weyl-orbit of an element  $c({\pi}) \in T^{\vee}(\C)$, or equivalently
   a semisimple conjugacy class in the Langlands dual group  $G^{\vee}(\C)$ of $G$. This  semisimple class $c({\pi}) $   is the Satake parameter of $\pi$.  
 \end{itemize}
\vskip 5pt

As an example,   the trivial representation of $G(F)$ is certainly unramified and its Satake parameter is described as follows. Let 
\[  \iota_{\rm reg}: \SL_2(\C) \longrightarrow G^{\vee}(\C) \]
be a principal or regular $\SL_2$. Then the Satake parameter $c({\rm triv})$ of the trivial representation is 
the conjugacy class of
\[   c({\rm triv}) = \iota_{\rm reg} \left( \begin{array}{cc} 
q^{1/2} & 0  \\
0 & q^{-1/2} \end{array} \right), \]
where $q$ is the cardinality of the residue field of $F$. 
\vskip 5pt

 \vskip 5pt
\subsection{\bf Unramified isometry theta correspondence} 
 Now let's return to our setting where $W$ and $V$ are symplectic and quadratic spaces,  with $\dim W = 2n$ and $\dim V =2m$.
Since $V$ and $W$ are both split with trivial discriminants, we may fix self-dual lattices
\[  \Lambda \subset W \quad \text{and} \quad \Lambda' \subset V. \]
These lattices endow the groups $\GSp(W)$, $\Sp(W)$,  $\GSO(V)$ and $\SO(V)$ with its Chevalley structure over  $\mathcal{O}_F$.
The hyperspecial maximal compact subgroups
 \[ K  = \GSp(W) (\mathcal{O}_F) \quad \text{and} \quad K' =  \GSO(V)(\mathcal{O}_F)  \]
 are the stabilizers of these lattices  in the respective similitude groups. Likewise, we have the hyperspecial maximal compact subgroups
\[  K_{\flat} = K \cap \Sp(W) \quad \text{and} \quad K'_{\flat} = K' \cap \SO(V) \]
in the respective isometry groups.    

Let us fix a Witt basis 
\[ \{e_1,\dots,e_n, f_n, \dots,f_1\} \quad \text{of $\Lambda$} \]
so that $\langle e_i, f_j \rangle_W = \delta_{ij}$.  Likewise, we may fix such a Witt basis for the lattice $\Lambda'$.  These Witt bases define 
 maximal tori  over $\mathcal{O}_F$: 
 \[  T \subset  \GSp(W)  \quad \text{and} \quad T_{\flat}  = T \cap \Sp(W) \subset \Sp(W) \]
 and
 \[  T'  \subset  \GSO(V) \quad \text{and}  \quad   T'_{\flat} = T' \cap \SO(V) \subset \SO(V), \] 
with the property that the relevant Witt basis is a basis of eigenvectors for the action of the relevant maximal torus. Thus, the Witt bases provide isomorphisms
\begin{equation} \label{E:tori}
   T_{\flat}  \simeq  \prod_{i=1}^n \mathbb{G}_m \quad \text{and} \quad   T \simeq \left( \prod_{i=1}^n \mathbb{G}_m \right)  \times \mathbb{G}_m   \end{equation}
where the first isomorphism is given by the eigencharacters  of $T$ on the ordered basis $\{e _1,\dots,e_n \}$  and likewise for the projection of the second isomorphism to the first factor; the second projection $T \rightarrow \mathbb{G}_m$ is given by the similitude character. 
In particular, this provides a complementary $\mathbb{G}_m$  to $T_{\flat}$ in $T$:
\begin{equation} \label{E:tori2}
  T  = T_{\flat} \times \mathbb{G}_m. \end{equation}
In this identification, the center $Z$ of $\GSp(W)$ is given by
\[  
\left( (z,\dots,z), z^2 \right) \in T_{\flat} \times \mathbb{G}_m,   \quad z \in \mathbb{G}_m. \] 
\vskip 5pt

 In view of (\ref{E:tori}) and (\ref{E:tori2}),  we may write a  character $\chi_{\flat}: T_{\flat}(F) \rightarrow \C^{\times}$  as $\chi_{\flat} = \chi_1 \times\dots\times \chi_n$ and a character $\chi$ of $T(F)$ as $\chi = \chi_{\flat} \times \mu$.  Restricted to the center $Z$ of $\GSp(W)$,   
 \[  \chi|_Z =  \mu^2 \cdot \prod_{i=1}^n \chi_i . \]
Of course, we have the analogous discussion for the tori $T'_{\flat}$ and $T'$ in $\SO(V)$ and $\GSO(V)$. 
 \vskip 5pt

Containing  each maximal torus above, we also have a Borel subgroup over $\mathcal{O}_F$, which is upper triangular with respect to the relevant Witt basis, so that we have:
 \[   T \subset B \subset \GSp(W)  \quad \text{and} \quad T_{\flat}  = T \cap \Sp(W) \subset B_{\flat}   \subset \Sp(W), \]
and likewise 
\[  T'  \subset B' \subset \GSO(V) \quad \text{and}  \quad   T'_{\flat} \subset B'_{\flat} \subset \SO(V). \] 
Given a character $\chi_{\flat} = \chi_1 \times\dots\times \chi_n$  of $T_{\flat}(F)$, we will denote the associated (normalized) principal series  representation as: 
\begin{equation} \label{E:unramps}
   i_{B_{\flat}}(\chi_1 \times\dots\times \chi_n) := {\rm Ind}_{B_{\flat}}^{\Sp(W)} \chi_{\flat}. \end{equation}
 Likewise, we will write $i_B(\chi_1 \times \dots\times \chi_n \times \mu)$ for a principal series representation of $\GSp(W)$.

\vskip 5pt

After the preparation above, we can now recall the following result for isometry theta correspondence.
\vskip 5pt

\begin{prop} \label{P:unram1}
Assume that $m \geq n+1$.
\vskip 5pt

 \vskip 5pt

\noindent  (i) Let  $\pi_{\flat} \in {\rm Irr}(\Sp(W))$ be a constituent of a principal series representation $i_{B_{\flat}}(\chi_1 \times \chi_2 \times\dots\times \chi_n)$. If $\theta_{\psi}(\pi_{\flat}) \ne 0$, then $\theta_{\psi}(\pi_{\flat})$ is a constituent of the induced representation 
\[  i_{B_{\flat}'}(\chi_1 \times \chi_2\times\dots\times \chi_n \times |-|^{m-n-1} \times |-|^{m-n-2} \times\dots\times 1)   \]
of $\SO(V)$.
\vskip 5pt

\noindent (ii) Suppose that $\psi$ has conductor $\mathcal{O}_F$. For $\pi_{\flat} \in {\rm Irr}_{K_{\flat}}(\Sp(W))$, we have $\theta_{\psi}(\pi_{\flat}) \in {\rm Irr}_{K'_{\flat}}(\SO(V))$ (in particular it is nonzero). 
\vskip 5pt

\noindent (iii) In the context of (ii), the Satake parameters  of $\pi_{\flat}$ and $\theta_{\psi}(\pi_{\flat})$ are related as follows. 
Consider the natural embedding
\[  \iota_{\flat} : \SO_{2n+1}(\C) \times \SO_{2m-2n-1}(\C) \longrightarrow \SO_{2m}(\C) = \SO(V)^{\vee}. \]
Then
\[  c({\theta_{\psi}(\pi_{\flat})}) =   \iota( c({\pi_{\flat}}), c({\rm triv})) \]
where $c({\rm triv})$ is the Satake parameter of the trivial representation of the split group $\Sp(W')$, with $\dim W' = 2m-2n-2$, whose dual group is $\SO_{2m-2n-1}(\C)$. 

\end{prop}
\vskip 5pt

\begin{proof}
The statement (i)  is a special case of a result of Kudla \cite[Thm. 2.5, Cor. 2.6 and Cor. 2.7]{K}, obtained as a consequence of his computation of the Jacquet modules of the Weil representation \cite[theorem 2.8]{K}.
Statements (ii) and (iii) are results of Howe  and Rallis \cite[\S 6]{R1}.
\end{proof}
\vskip 5pt

For the sake of concreteness, let us describe the map $\iota_{\flat}$ in Proposition \ref{P:unram1}(iii) on the level of maximal tori, as this is all that is needed for the unramified correspondence. Recall that we have fixed identifications:
\[  T_{\flat} = \prod_{i=1}^n \mathbb{G}_m \quad \text{and} \quad T_{\flat}' = \prod_{i=1}^m \mathbb{G}_m \]
so that one has
\[  T_{\flat}^{\vee} = (\C^{\times})^n \quad \text{and} \quad {T_{\flat}'}^{\vee} = (\C^{\times})^m. \]
We likewise have a maximal torus $S_{\flat} \subset \Sp(W')$ where $W'$ is as in Proposition \ref{P:unram1}(iii), with fixed identification
\[  S_{\flat} = \prod_{i=1}^{m-n-1} \mathbb{G}_m \quad \text{and hence } \quad S_{\flat}^{\vee} = (\C^{\times})^{m-n-1}. \]
Restricted to these maximal tori, the map 
\begin{equation}   \label{E:iotaflat}
  \iota_{\flat}: T_{\flat}^{\vee} \times S_{\flat}^{\vee} \longrightarrow {T'_{\flat}}^{\vee} \end{equation}
is given explicitly by:
\[  \iota \left( (t_1,\dots,t_n),  (s_1,\dots,s_{m-n-1}) \right)  =  (t_1,\dots,t_n, s_1,\dots, s_{m-n-1}, 1). \]
Now,  if $\pi_{\flat}$ is an unramified representation of $\Sp(W)$ contained in a principal series representation $i_{B_{\flat}}(\chi_1,\dots,\chi_n)$,  then its Satake
parameter is the conjugacy class of
\[ c({\pi_{\flat}}) = (\chi_1(\varpi), \dots,\chi_n(\varpi)) \in (\C^{\times})^n. \]
On the other hand, the Satake parameter of the trivial representation is 
\[ c({\rm triv}) = (q^{m-n-1},\dots, q) \in (\C^{\times})^{m-n-1}.  \]
Hence,
\[ \iota_{\flat}(c({\pi_{\flat}}), c({\rm triv}) )= ( \chi_1(\varpi), \dots,\chi_n(\varpi), q^{m-n-1},\dots,q, 1) \in (\C^{\times})^m. \]
In view of Proposition \ref{P:unram}(i), this is precisely the Satake parameter of $\theta_{\psi}(\pi_{\flat})$.

\vskip 10pt
\subsection{\bf Unramified similitude theta correspondence}
 We would now like to establish the analog of the above proposition for the similitude theta correspondence. Before that, 
let us make an observation about the interaction between the principal series representations for the similitude and isometry groups. 
 \vskip 5pt
 
 Since $B_{\flat} \backslash \Sp(W) = B \backslash \GSp(W)$,  the natural restriction map of functions define a $\Sp(W)$-equivariant isomorphism
 \[   i_B(\chi_1 \times\dots\times \chi_n \times \mu)  \simeq i_{B_{\flat}} (\chi_1 \times\dots\times \chi_n). \]
 Suppose now that the characters $\chi_i$ and $\mu$ are unramified, then one has
 \[ 
1 =  \dim i_B(\chi_1 \times \dots\times \chi_n \times \mu)^K  \leq  \dim i_B(\chi_1 \times \dots\times \chi_n \times \mu)^{K_{\flat}} = 
 \dim  i_{B_{\flat}}(\chi_1 \times\dots\times \chi_n)^{K_{\flat}} =1, \]
 so that  equality holds.  We record the consequence of this observation as a lemma.
 
 \begin{lemma} \label{L:interaction}
 Suppose that each $\chi_i$ and $\mu$ are unramified characters. Then
  \[
   \dim i_B(\chi_1 \times \dots\times \chi_n \times \mu)^{K} =  \dim i_B(\chi_1 \times \dots\times \chi_n \times \mu)^{K_{\flat}} =1 \]
  In particular, if $\pi$ is a constituent of $i_B(\chi_1 \times \dots\times \chi_n \times \mu)$, then
 \[
    \dim \pi^K = \dim \pi^{K_{\flat}} = 0 \quad \text{or} \quad 1. \]
    \end{lemma}
\noindent  The analogous lemma holds in the context of  $\SO(V)$ and $\GSO(V)$.

\vskip 5pt

Here is the main result of this appendix.
\vskip 5pt

\begin{prop}  \label{P:unram2}
 Assume that $m \geq n+1$.
 \vskip 5pt
 
 \noindent (i) Let $\pi \in {\rm Irr}(\GSp(W))$ be a constituent of a principal series representation 
 \[ i_B(\chi_1 \times \chi_2 \times\dots\times \chi_n \times \mu).\]
  If $\Theta(\pi) \ne 0$, then every irreducible subquotient of $\Theta(\pi)$ is a subquotient  of the principal series representation 
\[ i_{B}(\chi_1 \times \chi_2\times\dots\times \chi_n \times |-|^{m-n-1} \times |-|^{m-n-2} \times\dots\times 1 \times \mu |-|^{- \frac{(m-n)(m-n-1)}{4}}) \]
of $\GSO(V)$.
\vskip 5pt

\noindent (ii) For $\pi \in {\rm Irr}_{K}(\GSp(W))$, we have $\theta(\pi) \in {\rm Irr}_{{K'}}(\GSO(V))$ (in particular it is nonzero). 
\vskip 5pt

\noindent (iii) In the context of (ii), the Satake parameters $c({\pi})$ and $c({\theta(\pi)})$ of $\pi$ and $\theta(\pi)$ are related as follows. One has a natural commutative diagram
 \begin{equation} 
   \begin{CD}
\GSpin_{2n+1}(\C) \times \GSpin_{2m-2n-1}(\C) @>\iota>>  \GSpin_{2m}(\C)  \\
@VVV  @VVV  \\
\SO_{2n+1}(\C) \times \SO_{2m-2n-1}(\C)  @>\iota_{\flat}>> \SO_{2m}(\C).  \end{CD} \end{equation}
Then
\[  c({\theta(\pi)}) = \iota(c({\pi}), c({\rm triv})), \]
where $c({\rm triv})$ is the Satake parameter of the trivial representation of the split group $\GSp(W')$, with $\dim W' = 2m-2n-2$, whose dual group is $\GSpin_{2m-2n-1}(\C)$.
 \end{prop}
 
 \begin{proof} 
 (i) This follows from  the results of \cite{GT2}. More precisely, \cite[Thms A.1 and A.2]{GT2}
determine the Jacquet modules of the induced Weil representation  with respect to the maximal parabolic subgroups of $\GSp(W)$ and $\GO(V)$, analogous to what Kudla did in \cite[Thm. 2.8]{K}  in the setting of  isometry theta correspondence. With these Jacquet modules at hand, the same argument as in the proof of \cite[Thm. 2.5 and Cor. 2.6]{K}   allows one to determine the behaviour of cuspidal support under the similitude theta correspondence.   The statement (i) is a special case of this result. 
\vskip 5pt

 We shall give the proof of (i) for the sake of completeness, proceeding  by induction on $n = 1/2 \cdot \dim W$. 
The base case when $n=0$ holds  trivially. Assume now that $n \geq 1$ and (after applying a Weyl group element if necessary), we may assume that
\[  \pi \hookrightarrow i_B(\chi_1 \times \dots \times \chi_n \times \mu). \]
By induction in stages, we have
\[  \pi \hookrightarrow  {\rm Ind}_{P_1}^G \chi_1 \boxtimes \pi' \]
where 
\begin{itemize}
\item $P_1$ is the maximal parabolic subgroup of $\GSp(W)$ stabilizing the isotropic line $X_1 = F e_1$ in $W$, which has Levi subgroup $\GL(X_1) \times \GSp(W')$, with $W' = \langle e_2,\dots, e_n, f_n,\dots,f_2 \rangle$;
\item $\pi' \in {\rm Irr}(\GSp(W'))$ is a subquotient of $i_{B'}(\chi_2 \times \dots \times \chi_n \times  \mu)$. 
\end{itemize}
Now it follows  that 
\[ 0 \ne 
\Theta(\pi)^{\vee} \subset \Theta(\pi)^* = \Hom_{\GSp(W)}(\Omega_{V,W}, \pi)  \subset \Hom_{\GSp(W)}(\Omega_{V,W}, {\rm Ind}_{P_1}^G \chi_1 \boxtimes \pi' ), \]
where the superscript $^\vee$ indicates contragredient of a $\GSO(V)$-module whereas the superscript $^*$ indicates the full linear dual of a vector space. 
By Frobenius reciprocity,
\[  \Hom_{\GSp(W)}(\Omega_{V,W}, {\rm Ind}_{P_1}^G \chi_1 \boxtimes \pi' ) = \Hom_{\GL(X_1) \times \GSp(W')} ( R_{P_1}(\Omega_{V,W}), \chi_1 \boxtimes \pi'), \]
where $R_{P_1}(\Omega_{V,W})$ denotes the normalized Jacquet module of $\Omega_{V,W}$ with respect to $P_1$.  This normalized Jacquet module is what  \cite[Thm A.2]{GT2} computes. 
\vskip 5pt

More precisely,   \cite[Thm A.2]{GT2} shows that there is a short exact sequence
\begin{equation} \label{E:jacquet}
 \begin{CD}
0 @>>> J_0 @>>> R_{P_1}(\Omega_{V,W}) @>>> J_1 @>>> 0 \end{CD} \end{equation}
of representations of $\GSO(V) \times \GL(X_1) \times \GSp(W')$. The submodule $J_0$ and quotient $J_1$ are described as follows:
\vskip 5pt

\begin{itemize}
\item  Let $Q_1 \subset \GSO(V)$ be the maximal parabolic subgroup stabilizing an isotropic line $Y_1$, so that its Levi subgroup is of the form $\GL(Y_1) \times \GSO(V')$, with $\dim V' = 2m-2$. Then
\[  J_0 =  {\rm Ind}_{Q_1 \times \GL(X_1) \times \GSp(W')}^{\GSO(V) \times \GL(X_1) \times \GSp(W')} C^{\infty}_c(F^{\times}) \otimes \Omega_{V',W'}  \]
where the action on the inducing data is as follows:
\vskip 5pt

\begin{itemize}
\item[-]  $\Omega_{V', W'}$  is the induced Weil representation of $\GSO(V') \times \GSp(W')$;
\item[-] $(a, h, b) \in \GL(Y_1) \times \GSO(V') \times \GL(X_1)$ acts on $S(F^{\times})$ via:  
\[    (a,h,b) \cdot f (x) = f( a^{-1} \sim_{V'}(h) \cdot x \cdot b). \]
\end{itemize}

\item  the quotient $J_1$ is given by: 
\[ J_1 = |-|^{m-n} \cdot |\sim_{W'}|^{\frac{n-m}{2}} \cdot \Omega_{V, W'} \]
where $\Omega_{V,W'}$ is   the induced Weil representation   of $\GSO(V) \times \GSp(W')$ and $|-|^{m-n}$ is a character of $\GL(X_1)$. 
\end{itemize}
\vskip 5pt

Applying $\Hom_{\GL(X_1) \times \GSp(W')}( -, \chi_1 \boxtimes \pi')$ to the short exact sequence (\ref{E:jacquet}), one obtains the exact sequence of (not-necessarily smooth) $\GSO(V)$-modules:
\begin{equation} \label{E:jacquet2}
 \begin{CD}
0 @>>> \Hom_{\GL(X_1) \times \GSp(W')}(J_1, \chi_1 \boxtimes \pi') @>>> \Hom_{\GL(X_1) \times \GSp(W')}(R_{P_1}(\Omega_{V,W}), \chi_1 \boxtimes \pi') \\
@.  @. @VVV  \\
 @. @. \Hom_{\GL(X_1) \times \GSp(W')}(J_0, \chi_1 \boxtimes \pi').
\end{CD} \end{equation}
On considering the subspace of $\GSO(V)$-smooth vectors in the above exact sequence, one deduces that $\Theta(\pi)^{\vee}$ lies in a short  exact sequence
\[ \begin{CD}
 0 @>>> A_1 @>>> \Theta(\pi)^{\vee} @>>> A_0  @>>> 0  \end{CD} \]
 where $A_1$ (resp. $A_0$) is a smooth submodule of the first (resp. last) Hom-space in (\ref{E:jacquet2}).
Hence, if $\sigma$ is any irreducible subquotient of $\Theta(\pi)$, then $\sigma$ is a subquotient of $A_0^{\vee}$ or $A_1^{\vee}$. To establish the inductive step of the argument, it remains to verify that any irreducible subquotient of $A_0^{\vee}$ or $A_1^{\vee}$ is an irreducible subquotient of 
\[ i_{B}(\chi_1 \times \chi_2\times\dots\times \chi_n \times |-|^{m-n-1} \times |-|^{m-n-2} \times\dots\times 1 \times \mu |-|^{- \frac{(m-n)(m-n-1)}{4}}). \]
To this end, we now explicitly determine the $\GSO(V)$-modules 
\[   \Hom_{\GL(X_1) \times \GSp(W')}(J_1, \chi_1 \boxtimes \pi')  \quad \text{and} \quad  
 \Hom_{\GL(X_1) \times \GSp(W')}(J_0, \chi_1 \boxtimes \pi') \]
in (\ref{E:jacquet2})  in turn.
\vskip 5pt

For the first Hom space in (\ref{E:jacquet2}), we see that for
\[  \Hom_{\GL(X_1) \times \GSp(W')}(J_1, \chi_1 \boxtimes \pi')  \ne 0, \]
we need 
\[  \chi_1 = |-|^{m-n}. \]
When this holds, 
\[ \Hom_{\GL(X_1) \times \GSp(W')}(J_1, \chi_1 \boxtimes \pi') = \Hom_{\GSp(W')} (\Omega_{V, W'},  \pi' \cdot |\sim_{W'}|^{\frac{m-n}{2}}) =  \Theta( \pi' \cdot |\sim_{W'}|^{\frac{m-n}{2}})^*, \]
so that
\[    \Theta( \pi' \cdot |\sim_{W'}|^{\frac{m-n}{2}}) \twoheadrightarrow A_1^{\vee}. \]
By induction hypothesis (applied to $\GSO(V) \times \GSp(W')$), every irreducible subquotient of $ \Theta( \pi' \cdot |\sim_{W'}|^{\frac{m-n}{2}}) $ is a subquotient of 
the principal series 
\begin{align}
 & i_B( \chi_2 \times  \dots \times \chi_n \times |-|^{m-n} \times |-|^{m-n-1} \times \dots \times  1 \times \mu |-|^{-\frac{(m-n+1)(m-n)}{4}} \cdot |-|^{\frac{m-n}{2}} ) \notag \\
  = \quad  & i_B( \chi_2 \times  \dots \times \chi_n \times \chi_1\times  |-|^{m-n-1} \times \dots \times  1 \times \mu |-|^{-\frac{(m-n-1)(m-n)}{4}}). \notag
  \end{align}
  Hence, we see that (after permuting the $\chi_i$'s), any irreducible subquotient of $A_1^{\vee}$  is an irreducible subquotient of  
  \[   i_B( \chi_1 \times  \dots \times \chi_n \times  |-|^{m-n-1} \times \dots \times  1 \times \mu |-|^{-\frac{(m-n-1)(m-n)}{4}}) \]
  as desired.

 For the last Hom-space in  (\ref{E:jacquet2}),   we have
\[  \Hom_{\GL(X_1) \times \GSp(W')}(J_0, \chi_1 \boxtimes \pi') = \left( {\rm Ind}_{Q_1}^{\GSO(V)}  \chi_1^{-1} \otimes \Theta_{V', W'}(\pi' ) \cdot  (\chi_1 \circ \sim_{V'})\right)^*, \]
so that 
\[  A_0 \subset  \left( {\rm Ind}_{Q_1}^{\GSO(V)}  \chi_1^{-1} \otimes \Theta_{V', W'}(\pi' ) \cdot  (\chi_1 \circ \sim_{V'})\right)^{\vee},  \]
or equivalently
\[  {\rm Ind}_{Q_1}^{\GSO(V)}  \chi_1^{-1} \otimes \Theta_{V', W'}(\pi' ) \cdot  (\chi_1 \circ \sim_{V'}) \twoheadrightarrow A_0^{\vee}. \]
By induction hypothesis (applied to the similitude theta correspondence for $\GSO(V')\times  \GSp(W')$), one sees that any irreducible subquotient of $\Theta_{V', W'}(\pi' )$
is a subquotient of the principal series representation
\[ i_{B'}(\chi_2 \times  \dots \times \chi_n \times  |-|^{m-n-1} \times\dots\times 1 \times \mu |-|^{-\frac{(m-n)(m-n-1)}{4}}). \]
Hence, any irreducible subquotient of $\Theta_{V', W'}(\pi' ) \cdot (\chi_1 \circ \sim_{V'})$ is an irreducible subquotient of
\[   i_{B'}(\chi_2 \times  \dots \times \chi_n \times  |-|^{m-n-1} \times\dots\times 1 \times  \chi_1 \mu |-|^{-\frac{(m-n)(m-n-1)}{4} }). \]
It follows that any irreducible subquotient of $A_0^{\vee}$ is an irreducible subquotient of 
\[  i_B( \chi_1^{-1} \times  \chi_2 \times  \dots \times \chi_n, |-|^{m-n-1} \times\dots\times 1 \times  \chi_1 \mu |-|^{-\frac{(m-n)(m-n-1)}{4}} ). \]
Applying a Weyl element (exchanging $e_1$ and $f_1$), we see that the above principal series representation has the same irreducible subquotients as
\[  i_B( \chi_1 \times  \chi_2 \times \dots \times \chi_n, |-|^{m-n-1} \times\dots\times 1 \times  \mu |-|^{-\frac{(m-n)(m-n-1)}{4} }), \]
as desired.
\vskip 5pt

  \vskip 5pt
 
We have thus established the inductive step of the argument and completed the proof of (i).

\vskip 10pt

\noindent (ii) Suppose that $\pi \in {\rm Irr}_{K}(\GSp(W))$ is a submodule of an unramified principal series representation 
  \[ i_B(\chi_1 \times \chi_2 \times\dots \times \chi_n \times \mu).\]
 By Lemma \ref{L:interaction},  as an $\Sp(W)$-module, $\pi$ contains a unique irreducible summand which is $K_{\flat}$-unramified. 
By Proposition \ref{P:unram}(iii),  $\pi_{\flat}$ has nonzero theta lift to  an irreducible $K_{\flat}'$-unramified  representation  $\theta_{\psi}(\pi_{\flat})$ of   $\GSO(V)$.
It follows by  \cite[Lemma 2.2]{GT1} that $\theta(\pi)$ is nonzero irreducible, has the same central character as $\pi$ and contains $\theta_{\psi}(\pi_{\flat})$ as a 
$\SO(V)$-summand. 
 \vskip 5pt
 
 It remains to see that $\theta(\pi)$ is $K'$-unramified. By (i), $\theta(\pi)$ is a constituent of the $K'$-unramified principal series representation 
 \[ i_{B'}(\chi_1 \times \chi_2\times\dots\times \chi_n \times |-|^{m-n-1} \times |-|^{m-n-2} \times\dots\times 1 \times \mu |-|^{- (m-n)(m-n-1)/4}). \]
 Since $\theta(\pi)$ contains $\theta_{\psi}(\pi_{\flat})$ as a  $\SO(V)$-summand, one has
 \[  \dim \theta(\pi)^{K'_{\flat}}  \ne 0. \] 
It follows by Lemma \ref{L:interaction} (or rather its analog for $\GSO(V)$) that $\theta(\pi)$ is $K'$-unramifiied. 
   \vskip 5pt 

\noindent (iii) To understand the theta lift of unramified representations in terms of Satake parameters, let us describe the map $\iota$  on the level of maximal tori.  Recall from (\ref{E:tori2}) that we have an identification $T  =  T_{\flat} \times \mathbb{G}_m$. This induces an identification
\[ T^{\vee}(\C) =  T_{\flat}^{\vee}(\C) \times \C^{\times} \]
so that the natural map  $T^{\vee} \rightarrow T_{\flat}^{\vee}$ induced by the inclusion $T_{\flat} \hookrightarrow T$ is given by the first projection. 
If $W'$ is a symplectic space of dimension $2m-2n-2$,  with associated maximal torus $S \subset \GSp(W')$, then we likewise have
\[   S^{\vee}(\C)   = S_{\flat}(\C) \times \C^{\times}. \]
Similarly, for the group $\GSO(V)$, one has
\[   {T'}^{\vee}(\C)  = {T'_{\flat}}^{\vee}(\C) \times \C^{\times}. \]
On the level of isometry groups, one has the inclusion 
\[
 \iota_{\flat}  : T_{\flat}^{\vee} \times  S_{\flat}^{\vee} \longrightarrow {T'_{\flat}}^{\vee}(\C) \]
 described by  (\ref{E:iotaflat}).  Using the above description of the dual tori, the map $\iota$ is given by
\[ \begin{CD} 
  T^{\vee} \times S^{\vee}(\C) @>\iota>>  {T'}^{\vee}(\C) \\
  @|   @| \\
  (T_{\flat}^{\vee} \times \C^{\times}) \times (S_{\flat}^{\vee} \times \C^{\times}) @>\iota_{\flat} \times {\rm mult}>>  {T'_{\flat}}^{\vee} \times \C^{\times} \end{CD} \]
  where $\iota_{\flat}$ is as given in (\ref{E:iotaflat}) and ${\rm mult}: \C^{\times} \times \C^{\times} \longrightarrow \C^{\times}$ is the multiplication map.
\vskip 5pt

Now one has:
\[  c({\pi}) = \left( (\chi_1(\varpi), \dots,\chi_n(\varpi)),  \mu(\varpi) \right)  \in T_{\flat}^{\vee} \times \C^{\times}  = T^{\vee} \]
and
\[  c({\rm triv}) = \left( ( q^{m-n-1}, \dots, q), q^{-\frac{(m-n)(m-n-1)}{4}} \right) \in S_{\flat}^{\vee} \times \C^{\times}, \]
so that
\[ \iota( c({\pi}) , c({\rm triv}) ) = \left( (\chi_1(\varpi),\dots,\chi_n(\varpi), q^{m-n-1},\dots,q,1), \mu(\varpi) \cdot q^{-\frac{(m-n)(m-n-1)}{4}}  \right) \in {T'_{\flat}}^{\vee} \times \C^{\times}. \]
By (i), this is the Satake parameter of $\theta(\pi)$.
\end{proof}
\vskip 5pt

\subsection{\bf Proof of Proposition \ref{P:unram}}
When we specialize Proposition \ref{P:unram2} to the case $m = n+1$ and take into account of (\ref{E:GSO}),  we obtain Proposition \ref{P:unram}.


\vskip 15pt
 
\begin{thebibliography}{99}

\bibitem[AMR]{AMR} N. Arancibia, C. Moeglin and D. Renard, {\em Paquets d'Arthur des groupes classiques et unitaires}, 
Ann. Fac. Sci. Toulouse Math. (6) 27 no. 5 (2018), 1023--110.

\bibitem[A]{A} J. Arthur,
\emph{The endoscopic classification of representations: orthogonal and symplectic groups.}
American Mathematical Society Colloquium Publications \textbf{61},
A. M. S., Providence, RI (2013).

\bibitem[AS]{AS} M. Asgari and F. Shahidi, {\em Generic transfer for general spin groups},  Duke Math. J. 132 (2006), no. 1, 137-190. 

\bibitem[AGIKMS]{AGIKMS} H. Atobe, W.T. Gan, A. Ichino, T. Kaletha, A. Minguez and S.W. Shin, {\em Local intertwining relations and cotempered A-packets of classical groups}, arxiv preprint, arXiv:2410.13504.

%\bibitem[BCh]{BCh} J. Bella\"iche and G. Chenevier, {\em Families of Galois representations and Selmer groups}, Soci\'et\'e Math. de France, Ast\'erisque 324 (2009).

%\bibitem[B]{B} D. Blasius, {\em On multiplicities for $\SL(n)$}. Israel J. Math. 88 (1994), no. 1-3, 237-251

\bibitem[BG]{BG} D. Bump \& D. Ginzburg, {\em Spin $L$-functions on symplectic groups}, IMRN 8 (1992), 153--160.

\bibitem[CFK]{CFK} Y.Q. Cai, S. Friedberg and E. Kaplan, {\em Doubling constructions: global functoriality for non-generic cuspidal representations},  Ann. of Math. (2) 200 (2024), no. 3, 893-966.

%\bibitem[CG]{CG} P.S. Chan and W.T. Gan, {\em The local Langlands conjecture for $\GSp(4)$ III: Stability and twisted endoscopy},  J. Number Theory 146 (2015), 69-133.

%\bibitem[CG]{chgan} G. Chenevier and W. T. Gan, {\em ${\rm Spin}(7)$ is unacceptable}, arXiv preprint 2023.

%\bibitem[CH]{CH} G. Chenevier and M. Harris, {\em Construction of automorphic Galois representations II},
%Cambridge Math. Journal 1 (2013), 53--73.

\bibitem[CZ]{CZ} R. Chen and J.L. Zou, {\em Arthur multiplicity formula for even orthogonal and unitary groups}, to appear in Journal EMS, arXiv:2103.07956. 

\bibitem[C]{cheg2} G. Chenevier, {\it Subgroups of ${\rm Spin}(7)$ or ${\rm SO}(7)$ with each element conjugate to some element of ${\rm G}_2$, and applications to automorphic forms},  Documenta Math. 24 (2019), 95--161. 

\bibitem[CL]{CL} G. Chenevier and J. Lannes, {\em 
Automorphic forms and even unimodular lattices},
Ergebnisse der Mathematik und ihrer Grenzgebiete 69, Springer Verlag, xxi + 417 pp. (2019).
 
\bibitem[CR]{CR} G. Chenevier and D. Renard, {\em  Level one algebraic cusp forms of classical groups of small rank}, 
Mem. Am. Math. Soc., 237 (2015), no. 1121, v+122. 

\bibitem[CKPSS]{CKPSS} J.W. Cogdell, H. Kim, I. Piatetski-Shapiro, F. Shahidi, {\em Functoriality for the classical groups.} 
 Publ. Math. IHES. No. 99 (2004), 163--233. 

\bibitem[D]{sga} M. Demazure, {\em Groupes r\'eductifs : d\'eploiements, sous-groupes, groupes quotients}, Expos\'e XIII in M. Demazure, A. Grothendieck \& al., {\em {\rm SGA3} }, 1st ed. Spinger (1970), new ed. SMF Doc. Math. 8 (2011).
 
%\bibitem[GQT]{GQT} W.T. Gan, Y.N. Qiu and S. Takeda, {\em The regularized Siegel-Weil formula (the second term identity) and the Rallis inner product formula},  Invent. Math. 198 (2014), no. 3, 739-831.
 
 \bibitem[G]{G} W.T. Gan, {\em Automorphic forms and the theta correspondence}, in  {\em Automorphic forms beyond $\GL_2$}, 59-112, Math. Surveys Monogr., 279, Amer. Math. Soc., Providence, RI (2024).
 
 \bibitem[GS1]{GS1} W.T. Gan and G. Savin, {\em An exceptional Siegel-Weil formula and poles of the Spin L-function of $\PGSp_6$}. Compos. Math. 156 (2020), no. 6, 1231-1261. 
  
 \bibitem[GS2]{GS2} W.T. Gan and G. Savin, {\em The local Langlands conjecture for ${\rm G}_2$},  Forum Math. Pi 11 (2023), Paper No. e28, 42 pp.
 
\bibitem[GT1]{GT1} W.T. Gan and S. Takeda, {\em The local Langlands conjecture for $\GSp(4)$},  
Ann. of Math. (2) 173 (2011), no. 3, 1841-1882.
 
\bibitem[GT2]{GT2}  W.T. Gan and S. Takeda, {\em Theta correspondences for $\GSp(4)$}.  Represent. Theory 15 (2011), 670-718. 
 
\bibitem[GT3]{GT3} W.T. Gan and S. Takeda, {\em   On the regularized Siegel-Weil formula (the second term identity) and non-vanishing of theta lifts from orthogonal groups}, J. Reine Angew. Math. 659 (2011), 175-244. 

\bibitem[GJ]{GJ} S. Gelbart and H. Jacquet, {\em A relation between automorphic representations of ${\rm GL}(2)$ and
${\rm GL}(3)$}, Ann. Sci. \'E. N. S., Ser. 4 vol 11 (1978), 471--542.

% \bibitem[GK]{GK} S. Gelbart and A. Knapp, {\em L-indistinguishability and R-groups for the special linear group},  Adv. in Math. 43 (1982), no. 2, 101-121.
 
 \bibitem[GRS1]{GRS1} D. Ginzburg, S. Rallis and D. Soudry, {\em Periods, poles of L-functions and symplectic-orthogonal theta lifts},  J. Reine Angew. Math. 487 (1997), 85-114. 
 
 \bibitem[GRS2]{GRS2} D. Ginzburg, S. Rallis and D. Soudry, {\em The descent map from automorphic representations of $\GL(n)$  to classical groups},  World Scientific Publishing Co. Pte. Ltd., Hackensack, NJ (2011).   
 
 %\bibitem[HH]{HH} G. Henniart and R. Herb, {\em Automorphic induction for $\GL(n)$ (over local
%non-Archimedean fields)},  Duke Math. J. 78 (1995), no. 1, 131-192.
 
% \bibitem[HiS1]{HiS1} K. Hiraga and H. Saito, {\em On restriction of admissible representations},
%Algebra and number theory, 299-326, Hindustan Book Agency, Delhi, 2005.

% \bibitem[HiS2]{HiS2} K. Hiraga and H. Saito, {\em  On L-packets for inner forms of $\SL_n$},  Mem. Amer. Math. Soc. 215 (2012), no. 1013, vi+97 pp.
 
  \bibitem[HS]{HS} J. Hundley and E. Sayag, {\em Descent construction for $\GSpin$ groups},  Mem. Amer. Math. Soc. 243 (2016), no. 114.
 
  
 \bibitem[I]{I} T. Ikeda, {\em On the lifting of elliptic cusp forms to Siegel cusp forms of degree $2n$}, Ann. of Math. (2) 154 (2001), no. 3, 641-681.
 
 \bibitem[IY]{IY} T. Ikeda and S. Yamana, {\em On the lifting of Hilbert cusp forms to Hilbert-Siegel cusp forms.}, Ann. Sci. ENS (4) 53 (2020), no. 5, 1121-1181.
 
\bibitem[KMRT]{KMRT} M. Knus, A. Merkurjev, M. Rost, and J.-P. Tignol,
  {\em The Book of Involutions.}  AMS Colloquium Publications,
  Vol. 44 (1998).
  
%\bibitem[KS]{KS} R. Kottwitz and D. Shelstad, {\em  Foundations of twisted endoscopy},  Ast\'erisque No. 255 (1999), vi+190 pp.
 
% \bibitem[KS]{KS} A. Kret and S.W. Shin, {\em Galois representations for general symplectic groups}, preprint, available at https://arxiv.org/pdf/1609.04223.pdf.
 
%\bibitem[LR]{LR}  E. Lapid and S. Rallis,  \emph{On the local factors of representations of classical groups},
%Automorphic representations, $L$-functions and applications: progress and prospects,
%Ohio State Univ. Math. Res. Inst. Publ. \textbf{11}, de Gruyter, Berlin, 2005, pp.~309--359.

%\bibitem[S]{S} S. W. Shin, {\em Automorphic Plancherel density theorem}, Israel J. Math. 192 (2012), no.1, 83-120. 

%\bibitem[St]{St}  R. Steinberg, {\em  Regular elements of semisimple algebraic groups}, Publ. Math. IHES  Vol. 25 (1965), Pg. 49-80.
 
%\bibitem[SZ]{SZ} B.Y. Sun and C.B. Zhu, {\em Conservation relations for local theta correspondence},  J. Amer. Math. Soc. 28 (2015), no. 4, 939-983.

\bibitem[K]{K}  S. Kudla, {\em On the local theta-correspondence},  Invent. Math. 83 (1986), no. 2, 229-255.

%\bibitem[La]{La}. E. Lapid, {\em Some results on multiplicities for $\SL(n)$},  Israel J. Math. 112 (1999), 157-186.

\bibitem[LP]{LP} M. Larsen \& R. Pink, {\em On $\ell$-independence of Algebraic Monodromy Groups in
Compatible Systems of Representation}, Invent. Math. 107 (1992), 603--636.

\bibitem[Li]{Li} J.S. Li, {\em The correspondences of infinitesimal characters for reductive dual pairs in simple Lie groups},
 Duke Math. J. 97 (1999), no. 2, 347-377. 


\bibitem[MW1]{MW1} C. Moeglin and J.L. Waldspurger,
{\em Stabilisation de la formule des traces tordue. Vol. 1},  Progress in Mathematics, 316. Birkh\"auser/Springer, Cham (2016), xviii+587 pp.

\bibitem[MW2]{MW2} C. Moeglin and J.L. Waldspurger,
{\em Stabilisation de la formule des traces tordue. Vol. 2},  Progress in Mathematics, 317. Birkh\"auser/Springer, Cham (2016).

\bibitem[NT]{NT} J. Newton \& J. Thorne, {\rm Symmetric power functoriality for holomorphic modular forms}, Publ. Math. I.H.\'E.S. 134 (2021), 1--116.

%\bibitem[AIP]{AIP} F. Andreatta, A. Iovita and V. Pilloni, {\em $p$-adic families of Siegel modular cusp forms}, Annals of Math. 181 (2015), 623--697.

\bibitem[P]{pan} L. Pan, {\rm The Fontaine-Mazur conjecture in the residually reducible case}, J. A. M. S. 35 (2022), 1031-1169.

\bibitem[PT]{PT} S. Patrikis \& R. Taylor, {\em Automorphy and irreducibility of some $\ell$-adic representations}, Compositio Math. 151 (2013).

\bibitem[Po1]{pollack} A. Pollack, {\em The Spin $L$-function on ${\rm GSp}_6$ for Siegel modular forms}, Compositio Math, 153 (2017), no. 7, 1391--1432.

\bibitem[Po2]{pollack2} A. Pollack \& S. Shah, {\it The Spin $L$-function of ${\rm GSp}_6$ via non unique model}, American Journal of Mathematics 140 (2018), no. 3, 753--788.

\bibitem[Pr]{Pr} T. Przebinda, {\em The duality correspondence of infinitesimal characters}, Colloq. Math. 70 (1996), no. 1, 93-102.

\bibitem[R1]{R1} S. Rallis, {\em Langlands' functoriality and the Weil representartion}, American Journal of Math., Vol. 104 (1982), No. 3,  469-515.

\bibitem[R2]{R2} S. Rallis, {\em L-functions and the oscillator representation}, Lecture Notes in Math. 1245, Springer-Verlag (1987).

\bibitem[Ra]{Ra} D. Ramakrishnan, {\em Modularity of the Rankin-Selberg L-series, and multiplicity one for $\SL(2)$},  Ann. of Math. (2) 152 (2000), no. 1, 45-111.

\bibitem[Ro]{Ro} B. Roberts, {\em The theta correspondence for similitudes}, Israel J. Math. 94 (1996), 285-317.

\bibitem[S]{serre} J.-P. Serre, {\em Groupes de Galois et repr\'esentations $\ell$-adique}, lectures at Coll\`ege de France no. 5 (1984),  notes taken by J.F. Boutot \& J. Osterl\'e 107 p. \url{https://numdam.org/item/CJPS_1984__5_/}.

\bibitem[T1]{taibisiegel} O. Ta\"ibi, {\em Dimensions of spaces of level one automorphic forms for split classical groups using the trace formula}, Ann. Sci. ENS  50 (2017), no. 2, 269-344.

\bibitem[T2]{taibiAn} O. Ta\"ibi, {\em The Euler characteristic of $\ell$-adic local systems on $\mathcal{A}_n$}, arXiv:2510.00656 (2025).

\bibitem[Ta]{taylor} R. Taylor, {\rm The image of complex conjugation in $\ell$-adic Galois representations}, Algebra and Number Theory  6 (2012).

\bibitem[Ti]{tits} J. Tits, {\em Reductive groups over local fields}, in Proc. Symp. in Pure Math. 33, {\em Automorphic forms, Representations and L-functions, Part I}, A. Borel \& W. Casselman ed., A.M.S. (1977). 

\bibitem[\textsc{vdG}]{vdg} G. van der Geer, {\em Siegel modular forms and their applications}, in {\em The 1--2--3 of modular forms}, Universitext, Springer Verlag,  (2008), 181--245. 

\bibitem[Vo]{Vo} S. C. Vo, {\em The spin $L$-function on the symplectic group ${\rm GSp}(6)$}, Israel J. Math. 101, 1--71 (1997).

\bibitem[X1]{X1} Bin Xu, {\em $L$--packets of quasisplit ${\rm GSp}(2n)$ and ${\rm GO}(2n)$.} Math. Ann. 370 (2018), no. 1-2, 71-189. 

\bibitem[X2]{X2} Bin Xu, {\em Global L-packets of quasi-split ${\rm GSp}(2n)$ and ${\rm GSO}(2n)$}, Amer. J. Math. 147 (2025), no 2, 401-464.

 \end{thebibliography}
\end{document}